\documentclass[11pt]{article}

\usepackage[utf8]{inputenc}
\usepackage[utf8]{inputenc}

\usepackage[round]{natbib}   % omit 'round' option if you prefer square brackets
\usepackage{graphicx}
\usepackage{float}
\usepackage{subcaption}

\usepackage{amsmath, amssymb, amsfonts, amsthm}
\usepackage{mathrsfs}
\usepackage{bbm}
\usepackage{mathtools}
\usepackage{derivative}
\usepackage{braket}
\usepackage{algorithm}      % For pseudocode with captions
\usepackage{algpseudocodex}  % For pseudocode
\usepackage{outlines}   % for sublists

\usepackage{setspace}

\usepackage{hyperref}
\hypersetup{
  colorlinks   = true,  %Colours links instead of ugly boxes
  urlcolor     = blue,  %Colour for external hyperlinks
  linkcolor    = blue,  %Colour of internal links
  citecolor    = red    %Colour of citations
}

\usepackage[toc,page]{appendix}
\usepackage[nottoc,numbib]{tocbibind}
\usepackage{placeins}

\usepackage{url}
\usepackage{hyperref}

\usepackage{textcomp}
\usepackage{gensymb}
\usepackage[version=4]{mhchem}
\usepackage[T1]{fontenc}

\usepackage{siunitx}
\newtheorem{theorem}{Theorem}[section]
\newtheorem{lemma}[theorem]{Lemma}
\newtheorem{corollary}[theorem]{Corollary}
\newtheorem{proposition}[theorem]{Proposition}

\newtheorem{remark}[theorem]{Remark}

\theoremstyle{definition}
\newtheorem{definition}[theorem]{Definition}

\newcommand{\dsdel}[2]{\frac{\partial #1}{\partial #2}}

\usepackage{enumerate}

\renewcommand{\le}{\leqslant} 
\renewcommand{\ge}{\geqslant} 
\renewcommand{\leq}{\leqslant} 
\renewcommand{\geq}{\geqslant}

\newcommand{\ind}{\mathds{1}}

\newcommand{\mc}[1]{\mathcal{#1}}

\DeclareMathOperator{\E}{\mathds{E}}
\DeclareMathOperator{\pr}{\mathds{P}}

{}
\usepackage{comment}

\usepackage{amsmath,amssymb,amsfonts,amsbsy,latexsym,dsfont}
\usepackage{amsmath,bm}
\usepackage{amsthm}

\numberwithin{equation}{section}  

\theoremstyle{plain}

\newtheorem{Remark}[theorem]{Remark}

\makeatletter
\renewcommand*{\@fnsymbol}[1]{\ensuremath{\ifcase#1  \mathsection\or  *\or \dagger\or \ddagger\or
  \mathparagraph\or \|\or **\or \dagger\dagger
   \or \ddagger\ddagger \else\@ctrerr\fi}}
\makeatother

\usepackage{authblk}

\title{The Keller--Segel Model for Chemotaxis on Metapopulation Graphs and Graphons}
\author[1,$\mathsection$]{Daecheol Kim}
\author[1,$\mathsection$]{Yefei Zhang}
\author[1,$\mathsection$]{Yuxuan Zhao}
\author[1,2]{Daniel B. Cooney}
\affil[1]{Department of Mathematics, University of Illinois Urbana-Champaign, Urbana, IL, USA}
\affil[2]{Carl R. Woese Institute for Genomic Biology, University of Illinois Urbana-Champaign, Urbana, IL, USA}
\affil[$\mathsection$]{\small{These authors contributed equally and should be considered joint first authors. }}

\begin{document}

\maketitle
\begin{abstract}
Many models of spatial populations use directed motion towards environmental attractants as a mechanism for spatial pattern formation. While prior approaches typically focus on partial differential equation models for directed motion in continuous space, it is natural to consider the dynamics of biological and social systems on a metapopulation featuring well-mixed local patches and spatial motion between patches. In this paper, we formulate a version of the Keller--Segel model for chemotaxis on a graph metapopulation, modeling a chemoattractant that can diffuse across patches and a population of individuals who move through a biased random walk through comparing chemoattractant levels along edges of the metapopulation network. We also derive a continuum limit for our model describing chemotaxis in a population on patches connected by a graphon, showing that our resulting integro-differential equation model arises as the limit of our discrete patch model for families of large deterministic or random graphs. Through linear stability analysis, we derive conditions on the chemotactic sensitivity required to produce instability of a spatially uniform equilibrium for our Keller--Segel models in terms of the reaction parameters and the eigenvalues of the Laplacian operator for the metapopulation graph or graphon. Through a mix of weakly nonlinear stability analysis, numerical simulations, and analytical bounds on solutions, we can also characterize the form of pattern-forming instability in our graphon model and show how our choice of demographic dynamics for the population and the graphon adjacency function can impact the long-time total population size and chemoattractant concentration achieved in emergent spatial patterns. 
\end{abstract}
\begin{spacing}{0.01}
\renewcommand{\baselinestretch}{0.1}\normalsize
\tableofcontents
\end{spacing}
%\singlespacing
\newpage

\section{Introduction}

Directed motion of individuals plays an important role in producing the emergence of spatial pattern and collective behaviors in a range of natural and social systems, from the aggregation of slime molds through chemotaxis \citep{keller1970initiation} to herding and flocking behaviors in groups of large animals \citep{topaz2006nonlocal,cucker2007emergent}. Mathematical models of chemotaxis-like phenomena often consist of systems of partial differential equations describing the spatial profile of a population and a substance that serves as a chemical or social signaling mechanism for the population, with the diffusion of both population and signal helping to smooth out the population and the directed motion of individuals towards the signaling substance serving to promote spatial pattern formation and the emergence of ``hot spots'' featuring a high density of population and signal. A substantial modeling literature has developed regarding these questions of directed motion, exploring how the rules for motion in individual-based models can promote various mean-field partial differential equation (PDE) limits and characterizing threshold quantities for the onset of pattern formation for a wide variety of functional forms for reaction dynamics and chemotactic sensitivity in the PDE models \citep{painter2002volume,hillen2009user,painter2019mathematical}. 

In the past several decades, researchers have used the Keller--Segel model and related models for chemotaxis as motivation for exploring the role of directed motion in human social systems \citep{hasan2020transport,neto2019returns,juchem2015capital,kong2025boundary,short2008statistical,alsenafi2018convection}. This has led to the formulation of PDE models with directed motion that have been used to describe economic aggregation and labor-driven migration towards cities, spatial patterns of residential burglaries \citep{short2008statistical,alsenafi2018convection}, and the spread of wealth within cities and the emergence of gentrification \citep{hasan2020transport,lindstrom2020qualitative,mauro2025dynamic}.  However, in the context of human social systems, it is not necessarily the case that the assumption of a continuous spatial domain or local motion governed by a PDE would be the appropriate way to describe directed motion, as individuals could move to cities far away or move within a city using a public transportation network. For this reason, it may be more appropriate to model such social systems using nonlocal spatial motion or a spatial domain consisting of a metapopulation featuring cities as nodes and a graph of connectivity describing possible motion between the nodes. For this purpose, it is helpful to formulate rules for directed and diffusive motion for individuals between patches in a metapopulation in terms of biased random walks along the metapopulation network and then to look to derive continuum limits of these models that take into account the nonlocal rules for motion between spatial locations.

Metapopulation models have also been used to study ecological competition coupled with purely diffusive motion, exploring how patterns in prey and predator density can emerge in systems of ODEs describing reaction-diffusion dynamics on a metapopulation graph \citep{levin1974,SegelLevin1976,gibert2019laplacian,slavik2020lotka,kumar2026stability}. Similar spatial patterns have been observed in nonlinear reaction-diffusion models for evolutionary games on networks \citep{fahimipour2022sharp,roy2026diffusion} and in models for the spread of contagion in city-structured populations \citep{yu2025spatial,yu2026dynamics,magrogan2026reaction}, and related reaction-diffusion models have explored how the structure of the metapopulation graph influences the speed of spatial spread \citep{rayfield2023spread}, Turing instability thresholds on graphs \citep{pranesh2024effect,othmer1971instability,SegelLevin1976,philipp2026emergent}, and basic reproduction numbers for epidemic models with spatial, age, or social structure \citep{delmas2022infinite,pang2026stochastic}. Other network metapopulation models have extended the work on metapopulation dynamics to incorporate cross-diffusive effects and cross-diffusive instabilities \citep{kuehn2024cross,lawton2024interspecific}, exploring the possibility of pattern-forming instabilities for reaction-diffusion models on families of random graphs \citep{ide2016turing,singh2026emergence}, or describing the effects of directed motion between patches using piecewise movement rates in which individuals always move along edges of the metapopulation network towards nodes that are deemed as more desirable for individuals. In this paper, we will look to extend this approach to describing directed motion through biased random walks on networks, and to derive continuum limits of metapopulation models that provide an analytically convenient description of directed motion with a dispersal network.

For large dense networks, graph limit theory offers a principled continuum representation: the connectivity of a graph sequence can be encoded by a measurable kernel $W:[0,1]^2\to[0,\infty)$ (a graphon), and discrete graph couplings converge to integral operators built from $W$. Rigorous graphon limits for diffusion-type evolution equations on deterministic dense graphs and on $W$-random graphs were developed by Medvedev \citep{medvedev2014nonlinear,medvedev2014Wrandom}, with extensions to other semilinear models including sparse random graphs \citep{kaliuzhnyi2017semilinear}. Graphon dynamical systems have since been used as continuum models for consensus and opinion dynamics \citep{lee2018consensus,bonnet2022consensus,during2024breaking,qiao2025consensus}, synchronization in oscillator networks \citep{medvedev2015stability,medvedev2020kuramoto,chiba2018bifurcations,kuehn2019power,arches2026predicting}, and pattern-forming bifurcations on large graphs \citep{bellieregraphon,bramburger2023pattern,medvedev2024turing,bramburger2025persistence,yagasaki2026bifurcations}. Building on this framework, we derive a graphon limit for a Keller--Segel-type chemotaxis model on graphs and use the continuum description to characterize linear instability thresholds and weakly nonlinear pattern selection on large networks.

In this paper, we formulate a model of metapopulation dynamics that combines local reaction dynamics within patches with directed motion along the edges of the metapopulation network formulated in terms of biased random walks. As a baseline model of directed motion, we will consider a version of the Keller--Segel model for chemotaxis, which describes the spatial dynamics of a population with density $n(t,\vec{x})$ and the concentration of a chemical $c(t,\vec{x})$ that individuals use as a signal for spatial aggregation \citep{keller1970initiation,keller1971traveling,hillen2009user,painter2019mathematical}. In $d$-dimensional continuous space, the Keller--Segel model can be written in terms of the system of partial differential equations 
\begin{equation} \label{eq:KS-PDE}
    \begin{aligned}
        \frac{\partial n}{\partial t}&=\eta_n\nabla\cdot\big(\nabla n-\chi n\nabla c\big)+f(n)\\
        \frac{\partial c}{\partial t}&=\eta_c\Delta c +\alpha n -\beta c,
    \end{aligned}
\end{equation}
capturing the assumptions that individuals have net reproduction with rate $f(n)$, produce chemoattractant with rate $\alpha$, and perform both undirected diffusive motion and directed motion up gradients of the chemoattractant, while the chemoattractant decays at rate $\beta$ and diffuses within the spatial domain. We measure the overall mobility of the individuals and chemoattractant with the respective diffusivities $\eta_n$ and $\eta_c$, and we can describe the strength of directed motion for individuals through the chemotactic sensitivity $\chi$. For a positive number $\bar n$ satisfying $f(\bar n)=0$ and $f'(\bar n)<0$, the spatially uniform state $(n(x),c(x)) = \left(\bar n, \frac{\alpha \bar n}{\beta} \right)$ will be an equilibrium solution to this system of PDEs, and this equilibrium will be stable to all spatially uniform perturbations. However, by choosing sufficiently large chemotactic sensitivity $\chi$, this uniform state will be unstable to spatially heterogeneous perturbations and spatial patterns can emerge in the profiles of the population and the chemoattractant. 

We extend this exploration of pattern formation due to directed motion of a population towards a chemical attractant by formulating an analogue of the Keller--Segel model on a discrete metapopulation. In this model, we describe the population's motion between patches in our metapopulation as a biased random walk along the edges of the metapopulation graph, we consider a simple diffusive motion between edges for the chemoattractant, and we use identical reaction terms to the PDE model from Equation \eqref{eq:KS-PDE}. For our biased random walk, we assume that individuals choose to move between a given pair of patches with a probability based on the levels of chemoattractant at the two patches, and we parameterize the strength of the bias in motion using a sensitivity parameter $\chi$ that allows us to interpolate between a completely unbiased movement rule for $\chi = 0$ and deterministic movement to the higher quality patch as $\chi \to \infty$. By linearizing around the spatially uniform state, we are able to show that spatial patterns can emerge in our metapopulation model for sufficiently strong bias in motion towards patches with greater chemoattractant concentrations, and we show that the threshold sensitivity required for the onset of patterns depends on the parameters of the reaction dynamics and the eigenvalues of graph Laplacian describing the connectivity between patches in our metapopulation. We also perform numerical simulations of our graph metapopulation model for the case of two patches and logistic reaction dynamics, showing that the total population size and chemoattractant decreases with the onset of patterns and the increase in chemotactic sensitivity. 

While the discrete graph model provides a nice baseline for understanding how the structure of metapopulation network can contribute to onset of spatial patterns, it becomes more cumbersome to obtain analytically interpretable results for our model as the number of patches grows. To obtain more analytical understanding of pattern formation and chemotactic motion on metapopulations, we turn to graphons and a system of integro-differential equations to describe the behavior of our model in the limit of infinitely many patches. We show that the resulting Keller--Segel model on graphs serves as the continuum limit of the discrete metapopulation model over any fixed time interval $[0,T]$ both for dense families of deterministic and random graphs. In particular, we prove convergence for deterministic weighted graphs and for independent-edge random graphs sampled from cell averages of the limiting kernel, as well as a relabeling-invariant limit under cut-distance convergence of the kernels and compatible convergence of the initial profiles (Theorems~\ref{thm:det-Lp-conv}, \ref{thm:RG-limit}, and~\ref{thm:det-cut}). These convergence results for fixed time intervals allow us to understand how the early growth of spatial patterns from initial conditions close to the uniform equilibrium can occur in both our graph and graphon Keller--Segel models, and the convergence results for families of random graphs gives us additional motivation to use small-world graphons as a model for metapopulation structure in the presence of frequent movement to nearby patches and rarer long-range movement throughout the metapopulation. 

Our graphon Keller--Segel model is particularly amenable to analytical calculations, and we perform a mix stability analysis, numerical simulation, and a priori estimation to understand properties of emerging parameters. Our linear stability analysis allows us to demonstrate how the threshold chemotactic sensitivity for pattern formation depends on reaction parameters, movement rates, and the eigenvalues of the graphon Laplacian operator. Through weakly nonlinear analysis, we are able to derive conditions for whether our pattern-forming instability produces a supercritical or subcritical pitchfork bifurcation. In particular, we find that using a uniform band kernel with sufficiently wide radius or a small-world / Watts--Strogatz graphon with sufficiently high rewiring probability can produce subcritical instabilities, allowing for the possibility of bistability of a patterned state and the spatially uniform equilibrium for chemotactic sensitivities below the threshold for instability of the uniform state. In addition, we use the asymptotic approximations of our patterned equilibria to predict the total population size and the total chemical concentration across the spatial domain, showing that these quantities can achieve lower or higher levels in the patterned state relative to the uniform state depending on whether the reaction function $f(n)$ is locally concave or convex at the equilibrium point $\bar n$ for the reaction dynamics. This prediction is confirmed by numerical simulation, and our numerical results suggest that the trend of the increase or decrease in population size and chemical concentration will continue as the chemotactic sensitivity increases further beyond the onset of instability. 

Beyond our stability analysis calculations, we also explored the total population size and chemoattractant concentration by deriving long-time upper bounds for the mass of solutions to our graph and graphon Keller--Segel models. We show that long-time solutions feature at most the mass achieved in the spatially uniform state whenever the reaction function $f(n)$ is strictly concave for non-negative populations sizes, and that any patterned steady state will feature a lower population size across the spatial domain than in the spatially uniform state for such reaction functions (as discussed in Proposition~\ref{prop:logistic-total-mass-deficit}). This result applies to a range of functions $f(n)$ that appear often in the ecological literature and in related population models for social systems, such as logistic growth, the theta-logistic model, and models featuring saturating birth rates in the form of Holling Type II functional response. More generally, these insights into qualitative behavior of patterns through formal asymptotic expansions and estimates of the mass of solutions suggest that the use of graphon models of pattern formation via directed motion can be particularly helpful to understand how populations gain or lose through following different rules for movement between patches in metapopulations.

\subsection{Structure of the Paper}

The remainder of the paper is organized as follows. Section~\ref{sec:notation-conventions} collects the notation and conventions we use throughout the paper. Section~\ref{sec:model-formulation} derives the graph and graphon Keller--Segel models from the individual biased random walk rules. Section~\ref{sec:graph-LSA} analyzes the linear instability on finite graphs, highlighting the two-patch reduction. Section~\ref{sec:graphon-limits} rigorously establishes the deterministic, random, and cut-distance graphon limits. Section~\ref{sec:graph-model-stability} presents the linear and weakly nonlinear analyses for the continuum graphon model, concluding with the exact global mass deficit result. Finally, Section~\ref{sec:numericalsimulations} provides numerical simulations that confirm our analytical predictions, and Section~\ref{sec:discussion} discusses our main conclusions and outlines directions for future work.

We have organized the paper so that the core sections on pattern formation in the graph model (Section \ref{sec:graph-LSA}), the results on convergence to the graphon limits (Section \ref{sec:graphon-limits}), and the stability analysis of the graphon Keller--Segel model (Section \ref{sec:graph-model-stability}) are relatively self-contained. These sections can be read in any order depending on reading interest, with the sections depending primarily on the material presented in Section \ref{sec:model-formulation} on the formulation of the graph and graphon models. The simulation results in Section \ref{sec:numericalsimulations} present simulation results from both the graph and graphon models, and complement the analytical characterizations provided in Sections \ref{sec:graph-LSA} and \ref{sec:graph-model-stability}.

%%%%%%%%%%%%%%%%%%%%%%%%%%%%%%%%%%%%%%%%%%%%%%%%%%%%%%%%%%%%%%%%%%%%%%%%%%%%%%%%%%%%%%%%%%%%%%%%%%%%%%%%%%%%%%%%%%%%
\subsection{Notation and Conventions}
% \addcontentsline{toc}{section}{Notation and Conventions}
\label{sec:notation-conventions}

We collect the main notation used throughout the paper. Lower-case $n$ and $c$ are reserved for the population density and the chemoattractant concentration, respectively. Generic positive constants are denoted by $C$, $C_T$, $C_{R,p}$, etc., and may change from line to line. We avoid using lower-case $c$ for generic constants.

\paragraph{States and parameters.}
Let $n_i(t)$ and $c_i(t)$ denote the population density and chemoattractant concentration on a finite graph, with continuum counterparts $n(t,x)$ and $c(t,x)$ for $x\in I:=[0,1]$. The state vector is $U=(n,c)$ with initial data $U_{\mathrm{in}}=(n_{\mathrm{in}},c_{\mathrm{in}})$. A positive homogeneous equilibrium is denoted by $\bar U=(\bar n,\bar c)$, satisfying $\bar n>0$, $f(\bar n)=0$, and $\bar c=\alpha\bar n/\beta$.

The parameters $\eta_n$ and $\eta_c$ represent mobility coefficients, $\alpha$ and $\beta$ are the chemoattractant production and degradation rates, and $\chi$ is the chemotactic sensitivity. The local population reaction is governed by $f$, with specific assumptions introduced as needed.

\paragraph{Graphs and graphons.}
For an undirected $N$-node weighted graph $G=(V,E)$, let $A=(A_{ij})$ be its symmetric nonnegative adjacency matrix and $D=\operatorname{diag}(d_1,\ldots,d_N)$ be the degree matrix, where $d_i=\sum_{j=1}^N A_{ij}$. A graphon $W:I^2\to[0,\infty)$ is a symmetric measurable function, and $K$ denotes a generic measurable kernel. The discrete and continuous Laplacian operators are defined as
\begin{align*}
    L=\frac{A-D}{N},
    \qquad
    (\mathcal L_Ku)(x)=\int_I K(x,y)\big(u(y)-u(x)\big)\,dy.
\end{align*}
These operators are non-positive for symmetric nonnegative weights or kernels.

\paragraph{Function spaces and constants.}
All $L^p$ spaces are defined with respect to the Lebesgue measure. We denote $L^p(I)^2:=L^p(I;\mathbb R^2)$ and, for $U=(n,c)$, define the norm
\begin{align*}
    \|U\|_{L^p(I)^2}
    :=\begin{cases}
        \big(\|n\|_{L^p(I)}^p+\|c\|_{L^p(I)}^p\big)^{1/p},
        & 1\le p<\infty,\\[3pt]
        \max\{\|n\|_{L^\infty(I)},\|c\|_{L^\infty(I)}\},
        & p=\infty.
    \end{cases}
\end{align*}

%%%%%%%%%%%%%%%%%%%%%%%%%%%%%%%%%%%%%%%%%%%%%%%%%%%%%%%%%%
%%%%%%%%%%%%%%%%%%%%%%%%%%%%%%%%%%%%%%%%%%%%%%%%%%%%%%%%%%

\section{Formulating the Keller--Segel Model on Graphs and Graphons}
\label{sec:model-formulation}

In this section, we formulate a Keller--Segel model for chemotaxis on finite graphs and introduce its continuum limit in terms of a graphon partial differential equation. Our approach derives directed motion from biased random walks on networks, providing a direct connection between individual-based movement rules and mean-field population dynamics. An advantage of this formulation is that it naturally applies even in very small networks,
including two-patch systems, where directed movement can still be defined through biased transitions between patches. This contrasts with existing agent-based derivations of chemotaxis-type models, which typically assume that individuals choose among multiple neighboring patches when determining their movement direction \citep{short2008statistical,alsenafi2018convection,hillen2009user}. We then explain how these discrete models may be posed on families of dense graphs and how graphons arise naturally as continuum limits of such graph sequences, following the framework developed by \cite{medvedev2014nonlinear}.

\paragraph{Chemotaxis as biased random walks on graphs.}

Let $G=(V,E)$ be an undirected weighted graph with $N$ nodes and adjacency matrix
$A=(A_{ij})_{i,j=1}^N$, where $A_{ij}\ge 0$ denotes the weight of the edge between nodes $i$
and $j$. At each node $i$ we track the population density $n_i(t)\ge 0$ and the
concentration $c_i(t)\in\mathbb{R}$ of a chemoattractant.

We model directed motion as a biased random walk along graph edges. Individuals at node $j$
attempt to move to node $i$ at rate $\eta_n A_{ij}$, and the probability of accepting such a move
depends on the relative chemoattractant concentrations at the two nodes. Here we interpret the movement rate as a combination of an overall mobility rate $\eta_n$ for individuals and, depending on the formulation of the metapopulation graph, the value $A_{ij}$ of the adjacency matrix can be interpreted as either an additional weight modifying the mobility rate along the edge connecting patches $i$ and $j$ or a probability that individuals actually consider the possibility of moving along this edge. Given that individuals are considering the possibility of moving between patches $i$ and $j$, we assume that the probability that individuals actually move is given by the following function of the levels of chemoattractant $c_i$ and $c_j$ at patches $i$ and $j$
\begin{equation}
\mathds P(j \to i)
= \frac{e^{\chi c_i}}{e^{\chi c_i} + e^{\chi c_j}},
\label{eq:transition_prob}
\end{equation}
where $\chi \geq 0$ is called the chemotactic sensitivity, parameterizing how sensitivity in the movement probability depends on the chemoattractant concentrations at the two patches. This exponential response corresponds to an exponential-type preference rule and ensures biased motion toward higher concentration, and this form of the movement probability would produce the standard Keller--Segel model for chemotaxis in Equation \eqref{eq:KS-PDE} with linear sensitivity in the case of a nearest-neighbor lattice and the limit of infinitely many patches with infinitesimal distance between patches. Such biased movement rules with exponential weights have also been used to derive chemotaxis-like PDE models in a range of other models of directed motion on topics ranging from urban dynamics to eco-evolutionary dynamics with directed motion towards beneficial environments \citep{alsenafi2018convection,yao2025spatial,yao2026pattern,zhao2026pattern}
.

\begin{remark}
Continuous-time random walks on networks are typically modeled using either edge-based transition rates or node-based neighbor selection \citep{petit2021random}. Our chemotactic rule adopts the edge-based perspective: the transition probability across an edge depends exclusively on the signal concentrations at its two endpoints. Therefore, the dynamics are well-defined on a simple two-patch graph as well as on an arbitrary metapopulation network. Although it is common to derive PDE models of directed motion using biased random walks defined by individuals using attractant-dependent probabilities of choosing to move to one of their neighboring nodes on a nearest-neighbor lattice \citep{short2008statistical,alsenafi2018convection,hillen2009user,zhao2026pattern}, our formulation of a biased random walk along edges is more conducive to studying metapopulation networks with heteregeneous patch degrees. 
\end{remark}

Under these assumptions on the movement rules for individuals, the expected
evolution of the population density at node $i$ can be described by the following ODE 
\begin{equation}
\frac{d n_i}{dt}
= \frac{\eta_n}{N} \sum_{j=1}^N A_{ij}
\left[
\frac{e^{\chi c_i} n_j - e^{\chi c_j} n_i}
     {e^{\chi c_i} + e^{\chi c_j}}
\right]
+ f(n_i),
\label{eq:KS_graph_n}
\end{equation}
where $\eta_n>0$ is the mobility rate and $f$ describes local population growth (for example,
logistic growth).

The chemoattractant is assumed to diffuse along the graph and to be produced and degraded
locally according to
\begin{equation}
\frac{d c_i}{dt}
= \frac{\eta_c}{N} \sum_{j=1}^N A_{ij}\left(\frac{c_j - c_i}{2}\right)
+ \alpha n_i - \beta c_i,
\label{eq:KS_graph_c}
\end{equation}
where $\eta_c>0$ is the diffusion rate, $\alpha>0$ the production rate, and $\beta>0$ the
decay rate.

We assume that the population reaction $f:\mathbb{R}\to\mathbb{R}$ is locally Lipschitz. For positivity and global-in-time bounds, we use the following additional assumptions on the reaction function
\begin{equation}\label{eq:f-structure}
    f(n)=n h(n)\quad(n\ge0),
    \qquad
    h\in C^1([0,\infty)),
    \qquad
    \sup_{s\ge0}h(s)\le C_h.
\end{equation}
The function $h(n)$ corresponds to a per-capita net birth rate for the population, so taking a reaction term of the form $f(n) = n h(n)$ captures the biological interpretation of birth and death dynamics in our population and ensures that no population growth or decline is possible when there are no individuals present (corresponding to the extinction state $n = 0$ being an equilibrium of the reaction dynamics. For the case of logistic growth, our reaction function is $f(n)=rn(1-n/\bar n)$ with $r,\bar n>0$, so the per-capita growth rate is given by $h(n)=r(1-n/\bar n)$ and we can check that the maximal growth rate is $C_h=r$.

Equations \eqref{eq:KS_graph_n}--\eqref{eq:KS_graph_c} form a discrete Keller--Segel system
on the graph $G$. If $f$ admits a root $\bar{n}$ with $f'(\bar{n})<0$, then we can use this root to define a spatially uniform equilibrium state for our Keller--Segel system, which is characterized by population size and chemoattractant concentration satisfying
\begin{equation}
n_i \equiv \bar{n} \: \: \mathrm{and} \: \:
\qquad
c_i \equiv \frac{\alpha \bar{n}}{\beta}
\end{equation}
at each node $i$. In particular, our choice to define our biased random walk as with attractant-dependent movement probabilities between pairs of edges $i$ and $j$ permits this equilibrium point in the reaction dynamics to be a spatially uniform equilibrium of our metapopulation model, and we can use this formulation of our graph Keller--Segel model to explore pattern-forming instabilities of the uniform state on metapopulations in terms of the Laplacian matrix for the metapopulation networks and the reaction and movement parameters for the population and chemical attractant.

\paragraph{Dense-graph scaling and graphons.}

To study large-network behavior, we consider sequences of dense graphs
$\{G_N\}_{N\to\infty}$ with adjacency matrices $A^{(N)}$. For such graphs, interaction terms
are naturally rescaled by $1/N$, leading to a mean-field description. Graphons provide a
convenient continuum framework for describing the limits of these sequences.

A graphon is a symmetric measurable function $W:[0,1]^2\to[0,\infty)$, which represents the limiting connectivity of the graph sequence. Associated with $W$ is
the graphon Laplacian operator
\begin{equation}
(\mathcal{L}_W \phi)(x)
= \int_0^1 W(x,y)\bigl(\phi(y)-\phi(x)\bigr)\,dy.
\label{eq:graphon_laplacian}
\end{equation}

\paragraph{The graphon Keller--Segel model.}

In the graphon limit, node-based quantities $n_i(t)$ and $c_i(t)$ are replaced by functions
$n(x,t)$ and $c(x,t)$ defined for $x\in[0,1]$. Formally passing to the limit in
\eqref{eq:KS_graph_n}--\eqref{eq:KS_graph_c} yields the graphon Keller--Segel system
\begin{subequations} \label{eq:our-graphon-model}
\begin{align}
\partial_t n(x,t)
&= \eta_n \int_0^1 W(x,y)
\left[
\frac{e^{\chi c(x,t)} n(y,t) - e^{\chi c(y,t)} n(x,t)}
     {e^{\chi c(x,t)} + e^{\chi c(y,t)}}
\right] \, dy
+ f(n(x,t)),
\label{eq:KS_graphon_n}
\\
\partial_t c(x,t)
&= \eta_c \int_0^1 W(x,y)\left[\frac{c(y,t)-c(x,t)}{2}\right]\,dy
+ \alpha n(x,t) - \beta c(x,t).
\label{eq:KS_graphon_c}
\end{align}
\end{subequations}

This nonlocal integro-differential system provides a deterministic continuum model for
chemotaxis on large dense graphs. The kernel $W(x,y)$ is our graphon, describing the nonlocal connectivity between patches, and corresponds to the continuum limit of the adjacency matrix for the metapopulation graph. We will show in Section \ref{sec:graphon-limits} that the graphon $W(x,y)$ and the corresponding formulation of our Keller--Segel model will arise in the large-node limit of sequences of dense
graphs, including both deterministic and random graph families in the sense of graph limit
theory \citep{medvedev2014nonlinear}. This continuum description will be particularly beneficial in performing analytical explorations of pattern formation using linear and weakly nonlinear stability analysis, as we can use tools typically applied to partial differential equations to understand how reaction dynamics and directed motion on graphons can help to produce emergent spatial patterns. 

A particularly important class of examples of adjacency functions is given by distance-based graphons on the unit
circle, where the distance between $x,y\in[0,1]$ is defined by
\[
d(x,y)=\min\{|x-y|,\,1-|x-y|\}.
\]
In this setting we use a slight abuse of notation to write our graphon adjacency function in the form $W(x,y)=W(d(x,y))$ for a nonnegative kernel $W(\cdot)$, or we can equivalently write $W(x,y)=W(x-y)$ for an even, periodic connectivity kernel $W(\cdot)$. Considering families of distance-dependent kernels has the mathematical advantage use Fourier series and the convolution theorem to perform our linear and weakly nonlinear stability analyses, and this approach has advantage from a modeling perspective in allowing us to study how the rules for nonlocal connectivity help to shape emerging patterns. 

These class of distance-dependent graphons has already been used frequently in models of collective behavior and synchronization on large graph-structured populations \citep{medvedev2014nonlinear,kaliuzhnyi2017semilinear,chiba2018bifurcations,lee2018consensus,kuehn2019power,carmona2022stochastic}, and distance-dependent graphons have been used as a point of comparison with discrete ring graphs in the study of spatial pattern formation on graphons \citep{bramburger2023pattern,bramburger2025persistence}. In particular, there are several families of distance-dependent graphons that have been frequently used for analytically characterizing nonlocal dispersal or interactions, often derived from capturing the continuum limit of established families of random graphs with distant-dependent connectivity. 

The simplest distance-dependent graphon is described by a constant kernel, which describes equal connectivity between all points $(x,y) \in [0,1]^2$. Constant kernels appear in the Erd\H{o}s--R\'enyi graphon
\[
W(x,y)=p,
\]
corresponding to the continuum limit of the expected connectivity for Erd\H{o}s--R\'enyi graphs in which edges between each pair of nodes appear independently with probability $p$. We can also describe a distance-dependent graphons with piecewise constant connectivity, considering a connectivity function
\begin{equation}
W(x,y)=
\begin{cases}
1-p, & d(x,y)\le \rho,\\[2pt]
p, & d(x,y)>\rho,
\end{cases}
\end{equation}
where $\rho \in (0,1/2)$ describes the local neighborhood radius in the graphon and $p \in [0,1/2)$ describes the probability that a given local edge rewires to connect to a randomly chosen node in the network. 
For $0<\rho<1/2$, this defines a simple small-world graphon, where $\rho$ is the local interaction radius and $p$ controls the long-range connection probability.
This graphon adjacency function was introduced by \cite{medvedev2014small} to study a continuum Kuramoto model with small-world connectivity, drawing inspiration from network models including the \cite{watts1998collective} rewiring model, the \cite{newman1999renormalization} shortcut construction, and the \cite{monasson1999diffusion} model for diffusion on small-world graphs. This form of connectivity can be helpful for clarifying differences between the effects of local neighborhood radius and the likelihood of forming long-range connections, interpolating between cases of localized movement networks for small $\rho$ and $p$ and models of near complete graphon connectivity as $\rho$ or $p$ approach $1/2$.

Although we will focus primarily on piecewise constant graphons in our analysis of our continuum Keller--Segel model, it is also possible to consider a wide range of graphon connectivity functions with different rules for decreasing connectivity over distant. One family of models studied in nonlocal Kuramoto models is an exponentially decaying connectivity function
\[
W(x,y)=e^{-\kappa d(x,y)},
\]
where $\kappa>0$ controls the decay of interaction strength over distance. This connectivity function was introduced by \cite{kuramoto2002coexistence}, and was later apply by \cite{chiba2018bifurcations} to illustrate bifurcation conditions in graphon dynamical systems.

\paragraph{Relation between discrete and continuum dynamics.}

The graphon Keller--Segel system \eqref{eq:KS_graphon_n}--\eqref{eq:KS_graphon_c} may be
interpreted as the continuum limit of the discrete graph dynamics
\eqref{eq:KS_graph_n}--\eqref{eq:KS_graph_c}. Under suitable assumptions on the graph
sequence and initial data, solutions of the discrete system converge to solutions of the
graphon model as $N\to\infty$, extending convergence results previously established for
nonlinear diffusion equations on graph limits
\citep{medvedev2014nonlinear,bramburger2023pattern,bramburger2025capturing,bramburger2025persistence,petit2021random,zhang2025ginzburg,zhang2026graphon}.

This formulation provides a unified framework for analyzing chemotaxis-driven pattern
formation across discrete graphs and their graphon limits. In the following sections, we
study well-posedness, linear and weakly nonlinear instabilities, and convergence properties
of both the discrete and continuum Keller--Segel models.

\section{Emergence of Spatial Patterns in Keller--Segel Model on Graph Metapopulations}
\label{sec:graph-LSA}

In this section, we demonstrate the onset of spatial patterns in the Keller--Segel model on graph metapopulations, showing how patterns emerge due to the bias of traversing an edge between patches $i$ and $j$ based on the difference in chemoattractant concentration at these neighboring patches. We will use a linear stability analysis for the ODE system of Equations \eqref{eq:KS_graph_n} and \eqref{eq:KS_graph_c} around the uniform equilibrium given by $(n_i,c_i) = \left(\bar n, \frac{\alpha \bar n}{\beta}\right)$ for each $i \in \{1,\cdots,N\}$. Performing this stability analysis for an $N$-patch metapopulation will require studying the signs of the eigenvalues of a $2N \times 2N$ Jacobian matrix for our system, so we will look to simplify our analysis to understand the contributions of the reaction dynamics at each patch, the rules for movement between patches, and the connectivity of the metapopulation network. 

To simplify our stability analysis, we will use the master stability function approach \citep{pecora1998master,brechtel2018master,fahimipour2022sharp}, which allows us to decompose our exploration of the stability of our system of $2N$ linearized ODEs into studying the eigenvalues of $2\times 2$ matrices corresponding to the behavior on each patch and the eigenvalues of the $N \times N$ graph Laplacian matrix for our metapopulation matrix. This approach was initially used to demonstrate the emergence of Turing instability in reaction-diffusion systems on networks \citep{othmer1971instability,levin1974,othmer1974non,SegelLevin1976,ashkenazi1977spatial}, and has later been applied to study other forms of pattern formation to study pattern formation in metapopulations under nonlinear diffusion between patches \citep{brechtel2018master,fahimipour2022sharp,lawton2024interspecific}. This approach has also been used to study a range of problems for stability of equilibria in dynamical systems on networks and hypergraphs \citep{pecora1998master,russo2009contraction,tang2019master,mulas2020coupled,acharyya2024master}, and will apply readily to our question of studying chemotactic directed motion based on biased random walks between patches.

To use this approach, we first introduce some key quantities used to describe the eigenvalues of the linearization of our graph Keller--Segel system. We describe the connectivity between patches of our metapopulation through and undirected weighted graph $G=(V,E)$ with $|V|=N$ and symmetric adjacency matrix $A=(A_{ij})_{i,j=1}^N$ (consisting of weights $A_{ij}\ge 0$ for the edge between patches $i$ and $j$). We denote the degree of patch $i$ by
\begin{equation}
    d_i=\sum_{j=1}^N A_{ij}
\end{equation}
and the corresponding $N\times N$ degree matrix
\begin{equation}
    D=\mathrm{diag}(d_1,\dots,d_N).
\end{equation}
These quantities describing connectivity and patch degree then allow us to introduce the graph Laplacian matrix 
and
\begin{equation}
    L:=\frac{A-D}{N}
\end{equation}
whose action on a vector $\vec{v}$ gives us
\begin{equation}
 (L \vec{v})_i=\frac1N \sum_{j=1}^N A_{ij}(v_j-v_i).
\end{equation}
With our choice of sign convention for the graph Laplacian $L$, we know that $L$ will have one eigenvalue satisfying $\lambda_0 = 0$ and $N-1$ eigenvalues $\lambda_k < 0$ for $k \in \{1,\cdots,N-1\}$. This sign convention emphases the stabilizing effect of diffusion along edges of the metapopulation graph, resulting in the necessity of including directed motion to induce pattern formation in our graph Keller--Segel model.

The graph Laplacian will help us to describe the decomposition of the linearization of the graph Keller--Segel system using the master stability function approach, so we now show how the patch structure of our model will impact the resulting Jacobian matrix for our linearized system of ODEs. 

\paragraph{Homogeneous equilibrium and linearization of movement probabilities.}
We recall that we formulate the probability of moving between patch $i$ and patch $j$ as the following mapping based on the chemoattractant concentrations $c_i$ and $c_j$ at the two patches
\begin{equation}\label{eq:fermi-general}
\mathds P(i\to j)=\frac{e^{\chi c_j}}{e^{\chi c_i}+e^{\chi c_j}},\quad i\sim j,
\end{equation}
where the chemotactic sensitivity $\chi\ge 0$ describes the strength of bias towards patches with greater levels of chemoattractant. We use this movement probability to write the Keller--Segel model on the metapopulation graph $G$ as the following system of ODEs
\begin{equation}\label{eq:KS-graph}
\begin{aligned}
\dot n_i &= \frac{\eta_n}{N}\sum_{j=1}^N A_{ij}\Big(n_j \mathds P(j\to i)-n_i \mathds P(i\to j)\Big) + f(n_i),\\
\dot c_i &= \frac{\eta_c}{N}\sum_{j=1}^N A_{ij}\frac{c_j-c_i}{2} + \alpha n_i - \beta c_i,
\end{aligned}
\end{equation}
for each patch $i \in \{1,\cdots,N\}$ with parameters $\eta_n,\eta_c,\alpha,\beta>0$ and population net growth rate given by $f(n)=n\,h(n)$ as described in Equation~\eqref{eq:f-structure}. We assume that there exists an equilibrium point $\bar{n}>0$ for the population dynamics satisfying $f(\bar{n})=0$ and set the corresponding equilibrium value of the chemoattractant to $\bar{c}:=\frac{\alpha}{\beta}\bar{n}$, and we can check from Equations \eqref{eq:KS-graph} that the metapopulation state $n_i\equiv \bar{n}$ and $c_i\equiv \bar{c}$ for each patch $i$ will be a homogeneous equilibrium of Equation \eqref{eq:KS-graph}. At any such spatially uniform state, the equal concentrations of chemical attractant at each patch allow us to see that $\mathds P(i\to j)=\tfrac12$ for each edge in the metapopulation network, and we can further see that the movement probabilities will satisfy 
\begin{equation}\label{eq:fermi-derivs-general}
\left.\frac{\partial \mathds P(i\to j)}{\partial c_i}\right|_{(\bar{n},\bar{c})}=-\frac{\chi}{4} \: \: \mathrm{and} \: \:
\left.\frac{\partial \mathds P(i\to j)}{\partial c_j}\right|_{(\bar{n},\bar{c})}=\frac{\chi}{4}
\end{equation}
at this spatially uniform state. These properties of the movement probabilities $\mathds P(i \to j)$ will help us to derive the linearization of our graph Keller--Segel model around the spatially uniform state.

\paragraph{Linearization of Graph Keller-Segel Model and Kronecker Decomposition of the Jacobian Matrix}
To construct the Jacobian matrix describing the linearization of our system around the spatially uniform state, we will write our variables in the following order: 
\begin{equation}
\vec{x} = (n_1, \cdots, n_N, c_1, \cdots, c_N) \in \mathbb{R}^{2N}.
\end{equation}
With this representation of the state variables, we are able to perform a Kronecker decomposition of our Jacobian matrix in terms of the reaction terms and movement rules for each patch.
We provide detailed calculations of the entries of the $2N \times 2N$ Jacobian matrix for our linearization in Section \ref{app:graph-jacobian}. Our calculations in the appendix allow us to see that the Jacobian matrix $J$ can be written in terms of the following decomposition 
\begin{equation}
    J = P \otimes I_N + C \otimes L,
\end{equation}
where $I_N$ is the $N \times N$ identity matrix, $L$ is the ($N \times N$) graph Laplacian matrix for our metapopulation network, $P$ is a $2 \times 2$ matrix that describes the effects of local reaction dynamics in the form
\begin{equation}
P = \begin{pmatrix}
f'(\bar{n}) & 0 \\
\alpha & -\beta
\end{pmatrix},
\end{equation}
and $C$ is a $2 \times 2$ matrix describing the effects of possible diffusive and directed motion of the population and chemoattractant that is given by
\begin{equation}
C=\frac{1}{2} \begin{pmatrix}
\eta_n & -\eta_n \chi \bar{n} \\
0 & \eta_c 
\end{pmatrix} .
\end{equation}

We will now use this decomposition of the Jacobian matrix to determine the eigenvalues of $J$ in terms of the eigenvalues of the graph Laplacian $L$ and the structure of the $2 \times 2$ matrices $P$ and $C$. In particular, we can apply the approach used by \cite{brechtel2018master} and \cite{fahimipour2022sharp} to see that we can determine the $2N$ eigenvalues of $J$ by computing the eigenvalues of matrices of the form
\begin{equation} \label{eq:eigenvaluematrices}
P + \lambda_k C 
= \begin{pmatrix} f'(\bar n) + \displaystyle\frac{\eta_n \lambda_k}{2} & - \displaystyle\frac{\eta_n \chi \bar n \lambda_k}{2} \\ \alpha & -\beta + \frac{\eta_c \lambda_k}{2} \end{pmatrix}
\end{equation}
for each eigenvalue $\lambda_k$ of the graph Laplacian $L$. 
\begin{Remark}
A justification for this approach for calculating the eigenvalues of our $2N \times 2N$ Jacobian matrix $J$ in terms of examining the eigenvalues of the graph Laplacian $L$ and the eigenvalues of the $2 \times 2$ matrices $P + \lambda_k C$ is provided in the papers by \cite{brechtel2018master} and \cite{fahimipour2022sharp}. Their derivation of the eigenvalues and structure of associated eigenvectors uses the Kronecker product of each eigenvector of the graph Laplacian $L$ with the eigenvectors of each corresponding matrix $P + \lambda_k C$, using the mixed product property of the Kronecker product to show that all of the eigenvectors of $J$ can be constructed by this procedure. 
\end{Remark}

\paragraph{Condition for Instability for Each Laplacian Eigenvalue} 

We will now derive conditions for instability of the spatially uniform state by calculating the conditions to have an eigenvalue with positive real part in any of the matrices given by Equation \eqref{eq:eigenvaluematrices}. In particular, for the eigenvalue $\lambda_0 = 0$, we see that the matrix in Equation \eqref{eq:eigenvaluematrices} reduces to the matrix 
\begin{align*}
P = \begin{pmatrix}
f'(\bar n) & 0 \\ \alpha & -\beta
\end{pmatrix},
\end{align*}
which corresponds to the linearization for the reaction dynamics in a single patch. Using our assumptions that $f'(\bar n) < 0$ and $\beta > 0$ to ensure that $(n,c) = \left(\bar n, \frac{\alpha \bar n}{\beta} \right)$ is a stable equilibrum under the reaction dynamics, we can deduce that the two eigenvalues of $P$ will have negative real part. In particular, this means that the uniform state will not be destabilized by a perturbation that is uniform across each of the patches in the metapopulation. 

For the other $N-1$ eigenvalues of the graph Laplacian, we can look to Equation \eqref{eq:eigenvaluematrices} to see whether the matrix $P + \lambda_k C$ can have an unstable eigenvalue for some $\lambda_k$. Noting that  $\lambda_k < 0$ for $k \in \{1,\cdots,N-1\}$, we can use our assumptions that $f'(\bar n) < 0$ and $\beta > 0$ to deduce that
\begin{equation}
  \operatorname{tr}\left(P + \lambda_k C)\right) = f'(\bar{n}) - \beta + \lambda_k\left(\frac{\eta_n}{2} + \frac{\eta_c}{2}\right) < 0
\end{equation}
for each nonzero eigenvalue $\lambda_k$ of the graph Laplacian $L$. We therefore see that the only way to achieve instability of the spatial uniform state is for the determinant to be negative for some eigenvalue $\lambda_k < 0$, which will be achieved when
\begin{equation}
\operatorname{det}\left(P + \lambda_k C)\right) = -\left(f'(\bar{n}) + \lambda_k \frac{\eta_n}{2}\right)\left(\beta - \lambda_k \frac{\eta_c}{2}\right) +\lambda_k \alpha \frac{\eta_n \chi \bar{n}}{2}.
\end{equation}
In particular, we can see that this determinant condition can be achieved for any nonzero eigenvalue $\lambda_k$ of the graph Laplacian for sufficiently strong chemotactic sensitivity $\chi$. In fact, we find that there is a family of thresholds $\chi^*_k$ for the onset of instability for each negative eigenvalue $\lambda_k$, which are given by
\begin{equation}\label{eq:ODE-threshold}
 \chi_k^* = \frac{-\frac{1}{2} \eta_n \eta_c \lambda_k^2 + \left[\beta \eta_n - \eta_c f'(\bar{n})\right] \lambda_k + 2 \beta f'(\bar{n})}{\alpha \eta_n \bar{n} \lambda_k},
 \end{equation}
 where we can use our assumptions that $f'(\bar n) < 0$ and that all other reaction movement parameters are positive to deduce that this threshold $\chi^*_k$ will always correspond to a positive number. Because we can consider a perturbation proportional to any of the eigenvectors corresponding to a nonzero eigenvalue of the graph Laplacian, we can therefore see that the minimum chemotactic sensitivity required to generate pattern formation on our metapopulation is given by
 \begin{equation}
\chi^* := \min_{k \in \{1,\cdots,N-1\}} \chi^*_{k} = \min_{k \in \{1,\cdots,N-1\}} \left(  \frac{-\frac{1}{2} \eta_n \eta_c \lambda_k^2 + \left[\beta \eta_n - \eta_c f'(\bar{n})\right] \lambda_k + 2 \beta f'(\bar{n})}{\alpha \eta_n \bar{n} \lambda_k}\right)
 \end{equation}

For arbitrary graph metapopulations, we will to consider each of the nonzero eigenvalues of the graph Laplacian to help determine the threshold $\chi^*$ required for the onset of pattern formation. This can become complicated or computationally intensive for a large number of patches, so we will turn to understanding pattern formation in the graphon Keller--Segel model for more insight into how the rules for metapopulation connectivity help to shape the onset of spatial pattern formation. However, for the case of a two-patch metapopulation, we can look to obtain a more complete characterization of pattern onset using our stability analysis in terms of the thresholds computed above.

\paragraph{Example: two patches.}
One of the simplest examples of our metapopulation model is the case of two patches and the symmetric adjacency matrix with $A_{12} = A_{21} = 1$. For this metapopulation network, the graph Laplacian is given by
\begin{equation}
L=\frac{A-D}{N}=
\frac{1}{2}
\begin{pmatrix}
-1 & 1\\
1 & -1
\end{pmatrix},
\end{equation}
whose eigenvalues are $\lambda_0=0$ and $\lambda_1=-1$. 

Because $\lambda_1 = -1$ is the only nonzero eigenvalue of the graph Laplacian, we can apply Equation \eqref{eq:ODE-threshold} to see that the threshold chemotactic instability for our two-patch network is given by
\begin{equation}
\chi^*
=
\frac{\bigl(\beta+\frac{\eta_c}{2}\bigr)\bigl(\eta_n-2f'(\bar n)\bigr)}
     {\alpha\,\eta_n\,\bar n}.
\end{equation}

We can further explore the emergence of patterns on our two-patch network by considering numerical simulations of the system of ODEs of the graph Keller--Segel model from Equations \eqref{eq:KS_graph_n} and \eqref{eq:KS_graph_c}. In particular, we can consider the long-time outcome achieved on each patch starting from a small perturbation from the spatially uniform state, and we can study how directed motion impacts the emerging differences in population and chemical concentration at each patch. In the top row of Figure \ref{fig:twopatch-simulation}, we first explore the population sizes $n_1(t)$ and $n_2(t)$ and the chemical concentrations $c_1(t)$ and $c_2(t)$ after a fixed number of time-steps for various values of the chemotactic sensitivity $\chi$ in the case of the logistic reaction function $f(n) = r n \left( 1 - n / \bar n\right)$. We find that the population and chemical return to their spatially uniform values of $\bar n$ and $\bar c$ when $\chi < \chi^*$, but we see an increasing disparity between the population sizes and chemical concentration in each patch as $\chi$ increases beyond $\chi^*$. '

We also use numerical simulations to evaluate how the total population size $n_1(t) + n_2(t)$ and the total chemical concentration $c_1(t) + c_2(t)$ changes under numerical simulations for fixed time $t$ and various values of $\chi$. In the bottom rows of Figure \ref{fig:twopatch-simulation}, we show that the total population and chemical concentration agrees with the values achieved in the spatially uniform state, but that the masses of both quantities decrease below the uniform equilibrium value as $\chi$ increases past $\chi^*$. This suggests that directed motion and the achievement of a spatially heterogeneous steady state may actually result in a decreased overall population size and a lower overall concentration of the chemoattractant. Although individuals perform biased motion to regions with greater attractant levels and those individuals subsequently produce additional attractant, the formation of spatial patterns in our metapopulation results in greater decreases in population size and chemoattractant at patch 2 than the corresponding increases seen in patch 1. We will later see similar behavior in our graphon model, and this illustrating in terms of two patches provides a particularly salient demonstration on the relative benefits and costs of spatial pattern formation on the masses of solutions at patches serving as peaks and troughs of emerging patterns.

\begin{figure}[!ht]
    \centering
    \includegraphics[width=\textwidth]{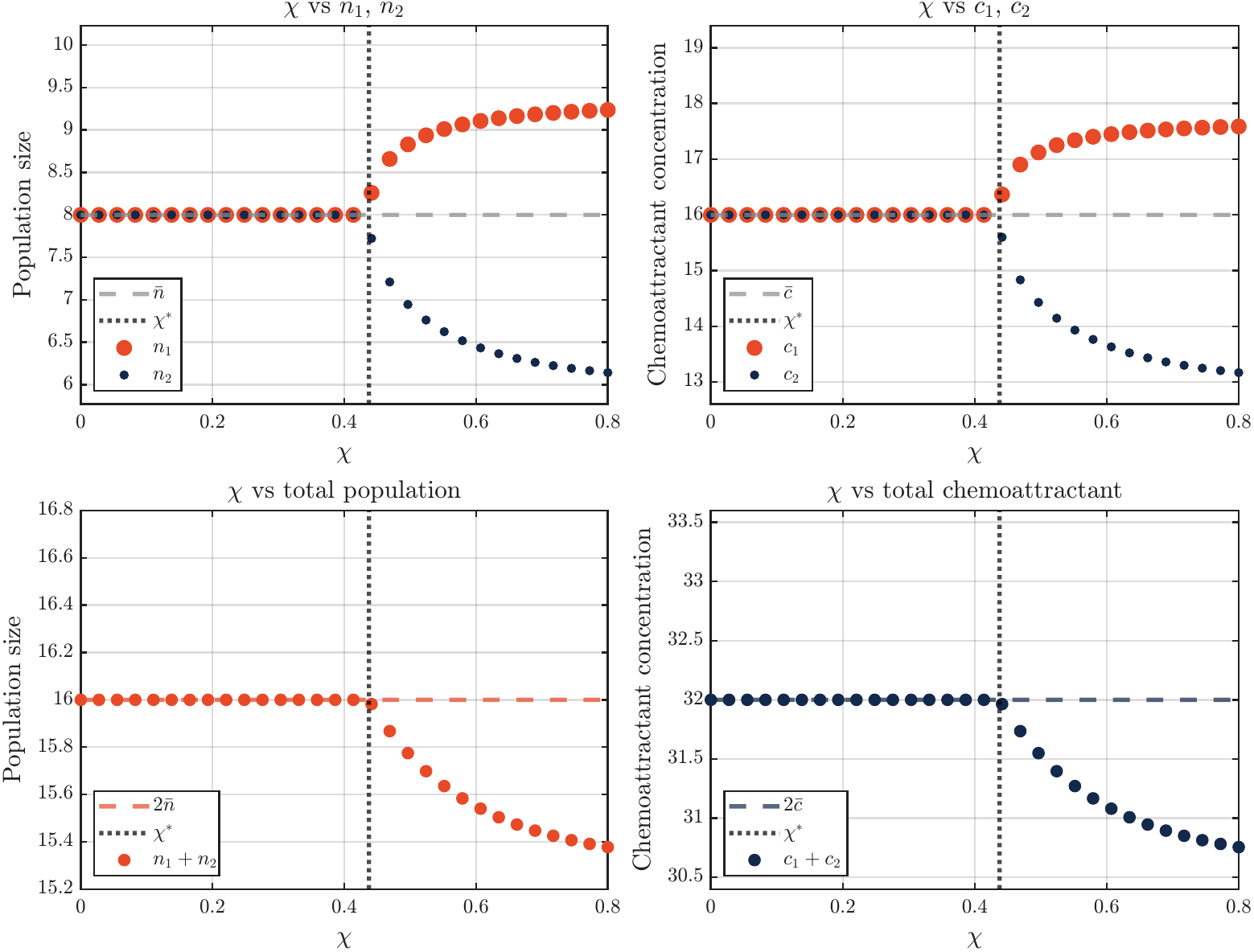}
    \caption{Comparison of long-time outcomes for simulations of the Steady states of the two-patch Keller--Segel system of Equation \eqref{eq:KS-graph}, plotted as a function of the chemotactic sensitivity $\chi$ for a fixed initial condition starting close to the uniform equilibrium state. Top row:  population sizes $n_1,n_2$ (left) and chemoattractant concentrations $c_1,c_2$ (right) on patches 1 and 2. Orange dots describe the population $n_1$ (left) and chemical concentration $c_1$ (right) in patch 1, blue dots describe the corresponding quantities in patch 2, and the gray dashed line describes the values of these quantities achieved in the spatially uniform state. Bottom row: total population size $n_1+n_2$ (left) and chemical concentration $c_1+c_2$ (right) across the two patches. The dots describe the total quantities achieved across the two patches under numerical simulations, while the dashed lines describe the corresponding total quantity achieved in the uniform equilibrium state. In each panel, the dotted vertical line describes the threshold chemotactic sensitivity $\chi^*=0.4375$. For each $\chi$, the system is integrated to $t=2\times 10^5$ from $(n_1,n_2,c_1,c_2)(0)=(\bar n+0.1,\,\bar n-0.1,\,\bar c,\,\bar c)$. The other reaction and movement parameters are fixed at $\eta_n=0.05$, $\eta_c=0.08$, $\alpha=0.2$, $\beta=0.1$,$f(n)=rn(1-n/\bar n)$ with $r=0.1$, $\bar n=8$, and $\bar c=16$.
}
    \label{fig:twopatch-simulation}
\end{figure}

We can gain further analytical intuition for the decrease in the masses of population and chemical concentration by looking at the conditions required to achieve an equilibrium state in Equations \eqref{eq:KS_graph_n} and \eqref{eq:KS_graph_n}. In particular, for any non-negative steady-state solutions, we can add together the righthand sides for the steady-state relation for population size to see that summing the equations over the two patches yields
\begin{equation} \label{eq:two-patch-n-steady-sum}
    f(n_1)+f(n_2)=0,
\end{equation}
and we can similarly sum the two steady-state equations for the chemical attractant to see that
\begin{equation} \label{eq:two-patch-c-steady-sum}
\beta(c_1+c_2)=\alpha(n_1+n_2).
\end{equation}
We note that the logistic reaction function considered in the simulations in Figure \ref{fig:twopatch-simulation} is one example of a strictly concave reaction function $f(n)$, so we can look to use this concavity property to understand the mass of steady states. In particular, for any strictly concave reaction function $f(n)$, we can apply Jensen's inequality to the righthand side of Equation \eqref{eq:two-patch-n-steady-sum} to estimate that 
\begin{equation*}
    0=\frac{f(n_1)+f(n_2)}{2} < f\left(\frac{n_1+n_2}{2}\right)
   \text{ for } n_1\ne n_2.
\end{equation*}
Since $f(s)>0$ if and only if $0<s<\bar n$, we can therefore deduce that any heterogeneous steady state (with $n_1 \ne n_2$) must have total population size satisfying
\begin{equation} \label{eq:n-sum-inequality-two-patch}
    n_1+n_2<2\bar n,
\end{equation}
and we can correspondingly use Equations \eqref{eq:two-patch-c-steady-sum} and Equation \eqref{eq:n-sum-inequality-two-patch} to see that the total chemical concentration satisfies 
\begin{equation}
    c_1+c_2=\frac{\alpha}{\beta}(n_1+n_2)<2\bar c.
\end{equation}
These estimates show that any spatially patterned uniform state will feature a lower total population size and chemical concentration than the uniform equilibrium state if the reaction term $f(n)$ is strictly concave for non-negative population sizes, regardless of the value of the chemotactic sensitivity $\chi$. We will show in Section \ref{sec:graphon-concave-mass-decay} that this decrease in mass in patterned states holds as well for the metapopulation Keller--Segel models for any number of discrete patches or in the graphon limit with infinitely many patches (see Proposition~\ref{prop:logistic-total-mass-deficit} in Section~\ref{sec:graphon-concave-mass-decay} for details).

%%%%%%%%%%%%%%%%%%%%%%%%%%%%%%%%%%%%%%%%%%%%%%%%%%%%%%%%%%
%%%%%%%%%%%%%%%%%%%%%%%%%%%%%%%%%%%%%%%%%%%%%%%%%%%%%%%%%%
\section{Graph limits for the Keller--Segel system}
\label{sec:graphon-limits}

This section establishes the graphon Keller--Segel system as a continuum limit of the discrete metapopulation dynamics on each fixed time interval $[0,T]$. Dense graph limit theory represents finite graphs by step kernels and formulates convergence in the cut metric up to measure-preserving relabelings \citep{lovasz2006denseLimits,lovasz2012large}. Further measure-theoretic and probabilistic background is provided in \cite{janson2010graphons} and \cite{bollobas2010cutMetric}.

For a limiting kernel $W:I^2\to\mathbb R$, the discrete system is embedded into the continuum via the cell-average step kernel $W^{(N)}$ on the uniform partition $\mathcal P_N$ (Definition~\ref{def:step-discretization}), which satisfies $W^{(N)}\to W$ in $L^p(I^2)$ whenever $1\le p<\infty$ and $W\in L^p(I^2)$ (Lemma~\ref{lem:PiN-Lp}). The convergence is proved using a stability--consistency framework combining local $L^p$ Lipschitz estimates of the nonlocal drift on bounded sets with Gr\"{o}nwall's inequality. For sampled graphs, Hoeffding's inequality provides moment bounds for centered row sums evaluated along the deterministic weighted-graph solution. Integrating these bounds in time and applying the stability estimate controls the error between the sampled and weighted solutions. The cut-distance limit follows from a cut-norm consistency estimate, under the condition that the kernels and initial data are compared using identical measure-preserving relabelings.

\paragraph{Main results.}
Fix $T>0$. Under the structural assumptions of Proposition~\ref{prop:graphon_bounds}\textup{(ii)}, the graphon and cell-average graph solutions satisfy $L^\infty$ bounds on $[0,T]$ that are uniform in $N$; the discrete bounds follow from Corollary~\ref{cor:discrete_bounds}. These bounds extend to the sampled graphs, whose normalized row sums are bounded by $1$, yielding bounds independent of the graph realization. For deterministic convergence with signed kernels, we retain an explicit boundedness hypothesis. Under the hypotheses of the respective results, we prove the following:

\begin{enumerate}[(R1)]
  \item \textbf{Deterministic limit in $L^p$.}
  For each $1\le p<\infty$, the step reconstruction of the $N$-node dynamics driven by $W^{(N)}$ converges to the graphon solution in $C([0,T];L^p(I)^2)$, provided that $W\in L^p(I^2)$ and the uniform integral and solution bounds hold (Theorem~\ref{thm:det-Lp-conv}).

  \item \textbf{Random graph limit in $L^p$.}
  For a symmetric kernel $0\le W\le1$, consider independent-edge graphs sampled from $W^{(N)}$ on the fixed partition $\mathcal P_N$. The reconstructed solutions converge in probability in $C([0,T];L^p(I)^2)$ for every $1\le p<\infty$. For any $\delta\in(0,1)$, the difference between the sampled and deterministic weighted-graph solutions is bounded by $C_{T,p}\sqrt{\log(2/\delta)/N}$ with probability at least $1-\delta$. The total error is bounded by the deterministic discretization error plus this fluctuation term; if $W$ is Lipschitz and $U_{\mathrm{in}}$ is suitably regular (e.g., $U_{\mathrm{in}}\in W^{1,p}(I)^2$), the total error rate is $O_{\pr}(N^{-1/2})$ (Theorem~\ref{thm:RG-limit} and Remark~\ref{rmk:rates}).

  \item \textbf{Cut-distance limit.}
  Let $K_N$ and $W$ be nonnegative kernels with uniformly bounded row and column integrals. If $K_N\to W$ in cut distance and the initial profiles converge in $L^1(I)^2$ under the same measure-preserving relabelings that align the kernels, then the relabeled solutions converge in $C([0,T];L^p(I)^2)$ for every $1\le p<\infty$. Convergence is established first in $L^1$, with higher exponents following via interpolation under uniform $L^\infty$ bounds (Theorem~\ref{thm:det-cut} and Remark~\ref{rmk:cut-initial}).
\end{enumerate}

\noindent
Statements \textup{(R1)}--\textup{(R3)} establish the primary graph limits. The deterministic and random limits in \textup{(R1)} and \textup{(R2)} adapt continuum-limit arguments for nonlinear diffusion on deterministic \citep{medvedev2014nonlinear} and $W$-random graphs \citep{medvedev2014Wrandom}. Related $L^2$ limits were obtained by \cite{kuehn2019power} for second-order oscillator equations rewritten as first-order systems. We extend this framework to coupled equations with signal-dependent transport and locally Lipschitz reactions via uniform finite-time bounds. In \textup{(R2)}, edge probabilities are defined by cell averaging on a fixed partition, and row-wise concentration yields a high-probability bound for the sampled-to-weighted solution error.

Statement \textup{(R3)} relates to cut-norm convergence results for reaction-diffusion equations with linear diffusion \citep{zhang2026graphon}. We establish an analogous finite-time limit for the chemotactic system via a cut-norm consistency estimate for the nonlinear transport term, with convergence holding under relabelings that align both kernels and initial profiles. Auxiliary results on pointwise discretization and sampled-graph cut-norm concentration are collected in Appendices~\ref{app:R4-pointwise} and~\ref{app:R5-cutconc}. These results are not used in the proofs of \textup{(R1)}--\textup{(R3)}.

\subsection{Preliminaries and notation}

We establish a unified functional-analytic framework for both the finite-dimensional graph ODEs and the continuum graphon equation. This requires specifying the kernel conventions, the nonlocal drift $\mathcal F_K$, and the step-function reconstruction that embeds the discrete $N$-node dynamics into the continuum setting.

Fix $T>0$. We adopt the state variables and product norms established in Section~\ref{sec:notation-conventions}. The deterministic limits assume real-valued kernels under uniform integral bounds, while the random-graph limits require $0\le W\le1$ a.e. For undirected graphs, $W$ is symmetric; consequently, the row and column integral bounds defined below coincide.

Since $|I|=1$, for $1\le p<\infty$ we have
\begin{align*}
    \|u\|_{L^p(I)}\le\|u\|_{L^\infty(I)},
    \qquad
    \|U\|_{L^p(I)^2}\le 2^{1/p}\|U\|_{L^\infty(I)^2}.
\end{align*}
The integral operators defined below act componentwise on vector-valued functions.

\paragraph{Integral operators, graphon Laplacian, and Schur bounds.}
For a measurable kernel $K$ on $I^2$ define
\begin{align*}
  (T_K u)(x):=\int_I K(x,y)u(y)\,dy,
  \qquad
  m_K(x):=\int_I K(x,y)\,dy.
\end{align*}
We use the non-positive graphon Laplacian convention
\begin{align*}
  (\mathcal L_K u)(x):=\int_I K(x,y)\big(u(y)-u(x)\big)\,dy
  =T_Ku(x)-m_K(x)u(x).
\end{align*}
Thus the diffusion term $\int_I K(x,y)(u(y)-u(x))\,dy$ is written as $\mathcal L_Ku$. We set
\begin{align*}
  M_x(|K|):=\operatorname*{ess\,sup}_{x\in I}\int_I |K(x,y)|\,dy,
  \qquad
  M_y(|K|):=\operatorname*{ess\,sup}_{y\in I}\int_I |K(x,y)|\,dx.
\end{align*}
If $M_x(|K|),M_y(|K|)<\infty$, then Schur's test and Riesz--Thorin interpolation yield, for all $p\in[1,\infty]$,
\begin{align*}
  \|T_{|K|}\|_{L^p\to L^p}\le M_x(|K|)^{1/p'}\,M_y(|K|)^{1/p},
  \qquad \frac1p+\frac1{p'}=1.
\end{align*}

The endpoint estimate $\|T_Ku\|_{L^\infty}\le M_x(|K|)\|u\|_{L^\infty}$ requires only $M_x(|K|)<\infty$; under this assumption, $\|m_K\|_{L^\infty}\le M_x(|K|)$ as well. If $K\in L^2(I^2)$, then $T_K$ is Hilbert--Schmidt and $\|T_K\|_{L^2\to L^2}\le\|K\|_{L^2(I^2)}$. If $K\in L^\infty(I^2)$, both integral bounds are at most $\|K\|_{L^\infty(I^2)}$, so the Schur estimate gives $\|T_{|K|}\|_{L^p\to L^p}\le\|K\|_{L^\infty(I^2)}$ for every $1\le p\le\infty$.

\paragraph{Model parameters and the exponential chemotactic preference.}
Let $W$ be a measurable kernel on $I^2$, and let $U=(n,c)$ with
$n,c:[0,T]\times I\to\mathbb R$.
We consider parameters $\eta_n,\eta_c,\alpha,\beta\ge 0$ and $\chi\in\mathbb R$.
Define the \emph{Exponential chemotaxis preference} function
\begin{align*}
  S(a,b):=\frac{e^{\chi a}}{e^{\chi a}+e^{\chi b}}
  =\frac1{1+e^{\chi(b-a)}},
\end{align*}
so that $0\le S\le 1$ and $S(a,b)+S(b,a)=1$ for all $a,b\in\mathbb R$.
Moreover, $S$ is globally Lipschitz with $|\partial_a S|=|\partial_b S|\le |\chi|/4$.

\medskip
\noindent\textbf{Interpretation of $S$.}
We interpret $S(a,b)$ as the probability (or preference) of choosing a \emph{destination}
with signal value $a$ given a \emph{source} with signal value $b$.
In particular, in the discrete model the transition probability from $k$ to $\ell$ is
$S(c_\ell,c_k)=1-S(c_k,c_\ell)$, whereas the gain term at node $k$ takes the form
$S(c_k,c_\ell)\,n_\ell$.

\medskip
Throughout Section~\ref{sec:graphon-limits} we assume that the reaction term $f:\mathbb R\to\mathbb R$ is locally Lipschitz.

\begin{definition}[Keller--Segel drift]\label{def:KS-drift}
For a measurable kernel $K$ on $I^2$ and $U=(n,c)$, define the (nonlocal) Keller--Segel drift $\mathcal F_K[U]:=(F_K^{(n)}(U),F_K^{(c)}(U))$ componentwise by
\begin{align}\label{eq:KS-conti}
\left\{
\begin{aligned}
F^{(n)}_K(U)(x)
&:=
\eta_n\int_I K(x,y)
\Big(
S\big(c(x),c(y)\big)n(y)
-\big(1-S\big(c(x),c(y)\big)\big)n(x)
\Big)\,dy
+ f\big(n(x)\big),\\
F^{(c)}_K(U)(x)
&:=
\frac{\eta_c}{2}\mathcal L_K c(x)+\alpha n(x)-\beta c(x)
=\frac{\eta_c}{2}\int_I K(x,y)\big(c(y)-c(x)\big)\,dy+\alpha n(x)-\beta c(x).
\end{aligned}
\right.
\end{align}
\end{definition}

Whenever $U\in L^\infty(I)^2$ and $K\in L^1(I^2)$, the integral terms in \eqref{eq:KS-conti}
are absolutely convergent for a.e.\ $x$ and define functions in $L^1(I)$ (componentwise).
If in addition $M_x(|K|)<\infty$, then $\mathcal F_K[U]\in L^\infty(I)^2$ and the pointwise formulation
$\partial_tU=\mathcal F_K[U]$ makes sense in $L^\infty$.

\paragraph{Graphon (continuum) system and mild formulation.}
The graphon Keller--Segel system associated with a kernel $W$ is the ODE on $L^\infty(I)^2$,
\begin{align*}
  \partial_t U(t,x)=\mathcal F_W[U(t)](x),
  \qquad
  U(0,x)=U_{\mathrm{in}}(x),
\end{align*}
equivalently the mild formulation
\begin{align*}
  U(t)=U_{\mathrm{in}}+\int_0^t \mathcal F_W[U(s)]\,ds,
  \quad t\in[0,T].
\end{align*}

The next proposition records the basic a priori information needed for all convergence arguments. Local well-posedness ensures that the continuum equation is a legitimate limiting dynamical system. The positivity and $L^\infty$ bounds are equally important: they provide deterministic bounds on the nonlinear exponential sensitivity factors and on the reaction term, so that the constants in the stability estimates below can be chosen uniformly along the approximating graph sequence.

\begin{proposition}[Graphon well-posedness and bounds]
\label{prop:graphon_bounds}
Assume $M_x(|W|)<\infty$ and let $f:\mathbb R\to\mathbb R$ be locally Lipschitz.
\begin{enumerate}
\item[\textup{(i)}]
For every $U_{\mathrm{in}}\in L^\infty(I)^2$, there is a unique maximal solution
$U\in C^1([0,t_{\max});L^\infty(I)^2)$ of
$\partial_tU=\mathcal F_W[U]$, $U(0)=U_{\mathrm{in}}$. If $t_{\max}<\infty$, then
\begin{align*}
\limsup_{t\uparrow t_{\max}}\|U(t)\|_{L^\infty(I)^2}=\infty.
\end{align*}

\item[\textup{(ii)}]
Assume in addition that $W\ge0$ a.e., $U_{\mathrm{in}}=(n_{\mathrm{in}},c_{\mathrm{in}})\ge0$
componentwise, and
\begin{align*}
f(r)=r h(r)\quad(r\ge0),
\qquad
h\in C^1([0,\infty)),
\qquad
\sup_{r\ge0}h(r)\le C_h.
\end{align*}
Writing $M_x=M_x(W)$, the solution is global and nonnegative. Moreover, for every $T>0$,
\begin{align*}
R_1(T)&:=\|n_{\mathrm{in}}\|_{L^\infty}e^{(\eta_nM_x+C_h)T},\\
R_2(T)&:=
\begin{cases}
\max\{\|c_{\mathrm{in}}\|_{L^\infty},\,\alpha R_1(T)/\beta\}, & \beta>0,\\[2pt]
\|c_{\mathrm{in}}\|_{L^\infty}+\alpha R_1(T)T, & \beta=0,
\end{cases}
\end{align*}
and
\begin{align*}
\sup_{t\in[0,T]}\|n(t)\|_{L^\infty}\le R_1(T),
\qquad
\sup_{t\in[0,T]}\|c(t)\|_{L^\infty}\le R_2(T).
\end{align*}
\end{enumerate}
\end{proposition}

\begin{proof}
See Appendix~\ref{app:graphon-bounds}.
\end{proof}

We now pass from the continuum kernel to finite graphs. The following step projection is the canonical deterministic way to represent a graphon on an $N$-node network. It is also the discretization for which the finite ODEs can be written exactly as a graphon equation with a step kernel. This exact reconstruction is what lets us avoid any formal limiting argument in the proof of Theorem~\ref{thm:det-Lp-conv}: after reconstruction, both the finite and continuum solutions live in the same Banach space and differ only through their initial data and kernels.

\begin{definition}[Cell-average discretization]
\label{def:step-discretization}
Let $\mathcal P_N=\{I_k\}_{k=1}^N$ be the uniform partition of $I$, where
$I_k=[(k-1)/N,k/N)$.
Define $\Pi_N:L^p(I)\to L^p(I)$ by cell averages,
\begin{align*}
  (\Pi_N u)(x):=\frac1{|I_k|}\int_{I_k}u(y)\,dy,
  \quad x\in I_k,
\end{align*}

For a kernel $W\in L^1(I^2)$ define its step-kernel approximation
\begin{align*}
  W^{(N)}:=(\Pi_N\otimes \Pi_N)W,
  \quad
  W^{(N)}(x,y)
  :=\frac1{|I_k||I_\ell|}\int_{I_k}\int_{I_\ell}W(u,v)\,du\,dv,
  \quad (x,y)\in I_k\times I_\ell.
\end{align*}
\end{definition}

\begin{lemma}[Cell-average convergence]
\label{lem:PiN-Lp}
For every $1\le p<\infty$,
\begin{enumerate}
\item[\textup{(i)}]
$\|\Pi_N u-u\|_{L^p(I)}\to0$ for all $u\in L^p(I)$,

\item[\textup{(ii)}]
$\|W^{(N)}-W\|_{L^p(I^2)}\to 0$ for all $W\in L^p(I^2)$.
\end{enumerate}
\end{lemma}
\begin{proof}
By Jensen's inequality, both $\Pi_N$ and $\Pi_N\otimes\Pi_N$ are $L^p$-contractions. For continuous functions, uniform continuity on the compact domains $I$ and $I^2$ guarantees uniform convergence of the step approximations. Since continuous functions are dense in $L^p$, the triangle inequality and the contraction property extend this convergence to all $L^p$ functions, proving both claims.
\end{proof}

\begin{remark}[The $L^\infty$ case]\label{rmk:Linfty-step}
In general, $W^{(N)}\to W$ need not hold in $L^\infty(I^2)$.
If $W$ is uniformly continuous on $I^2$, then $\|W^{(N)}-W\|_{L^\infty(I^2)}\to 0$.
In particular, if $W$ is Lipschitz with constant $L$, then
$\|W^{(N)}-W\|_{L^\infty(I^2)}\le L\sqrt2\,N^{-1}$.
\end{remark}

\begin{remark}[Cell averaging versus point sampling]\label{rmk:cell-vs-point}
The step kernel $W^{(N)}=(\Pi_N\otimes\Pi_N)W$ is defined by cell averages and therefore depends only on $W$ as an $L^p(I^2)$-function, that is, on its a.e.\ equivalence class. In particular, it does not require selecting a pointwise representative of $W$.

A deterministic pointwise discretization $W(x_k^{(N)},x_\ell^{(N)})$ on a fixed grid, on the other hand, can depend on how $W$ is defined on a null set and can fail to converge without additional regularity. Appendix~\ref{app:R4-pointwise} gives the corresponding convergence result when $W$ admits a uniformly continuous representative.
\end{remark}

The next step is to identify the finite graph dynamics associated with the step kernel. The factor $1/N$ in \eqref{eq:KS-discrete} is the dense-graph scaling: each node interacts with $O(N)$ possible neighbors, but the total input remains of order one. With this scaling, the sums in the finite system are exactly Riemann-type integrals against the step kernel $W^{(N)}$.

\paragraph{Discrete step-graph dynamics and step reconstruction.}
We now define the discrete $N$-node Keller--Segel system associated with $W^{(N)}$.
Set
\begin{align*}
  \mathsf w^{(N)}_{k\ell}:=N^2\int_{I_k}\int_{I_\ell}W(u,v)\,du\,dv,
  \quad 1\le k,\ell\le N,
\end{align*}
so that $W^{(N)}(x,y)=\mathsf w^{(N)}_{k\ell}$ for $(x,y)\in I_k\times I_\ell$.
The discrete dynamics are
\begin{align}\label{eq:KS-discrete}
\left\{
\begin{aligned}
\dot n^{(N)}_k
&=
\eta_n \frac1N
\sum_{\ell=1}^N
\mathsf{w}^{(N)}_{k\ell}
\Big(
S(c^{(N)}_k,c^{(N)}_\ell)n^{(N)}_\ell
-\big(1-S(c^{(N)}_k,c^{(N)}_\ell)\big)n^{(N)}_k
\Big)
+f\big(n^{(N)}_k\big),
\\[1ex]
\dot c^{(N)}_k
&=
\frac{\eta_c}{2} \frac1N
\sum_{\ell=1}^N
\mathsf{w}^{(N)}_{k\ell}\big(c^{(N)}_\ell-c^{(N)}_k\big)
+\alpha n^{(N)}_k-\beta c^{(N)}_k,
\qquad k=1,\dots,N.
\end{aligned}
\right.
\end{align}

We initialize by cell averages of $U_{\mathrm{in}}=(n_{\mathrm{in}},c_{\mathrm{in}})$:
\begin{align*}
n_k^{(N)}(0)=N\int_{I_k} n_{\mathrm{in}}(x)\,dx,
\qquad
c_k^{(N)}(0)=N\int_{I_k} c_{\mathrm{in}}(x)\,dx.
\end{align*}
We identify vectors $v=(v_k)_{k=1}^N$ with step functions on $I$ by setting $v(x):=v_k$ for $x\in I_k$;
under this identification $\Pi_N v=v$.
In particular, $\Pi_N U^{(N)}(t)$ denotes the step reconstruction of the discrete solution.

\begin{lemma}\label{lem:recon}
The step reconstruction $\Pi_N U^{(N)}$ satisfies
\begin{align*}
\Pi_N U^{(N)}(t)
=\Pi_NU_{\mathrm{in}}
+\int_0^{t}\mathcal F_{W^{(N)}}\!\big[\Pi_N U^{(N)}(s)\big]\,ds.
\end{align*}
\end{lemma}

\begin{proof}
On $I_k\times I_\ell$, both $W^{(N)}$ and $\Pi_N U^{(N)}$ are constant, so each integral in $\mathcal F_{W^{(N)}}[\Pi_N U^{(N)}]$ equals the corresponding sum in \eqref{eq:KS-discrete}.
\end{proof}

The reconstruction identifies the discrete system with the graphon equation on the space of step functions. The continuum well-posedness result therefore also provides existence, positivity, and bounds for the discrete solutions, with constants independent of $N$.

\begin{corollary}[Discrete well-posedness and uniform bounds]
\label{cor:discrete_bounds}
Under the hypotheses of Proposition~\ref{prop:graphon_bounds}\textup{(ii)}, the system \eqref{eq:KS-discrete} with cell-average initial data has a unique global nonnegative $C^1$ solution for every $N$. With $R_1(T)$ and $R_2(T)$ from that proposition, for every $T>0$,
\begin{align*}
\sup_{t\in[0,T]}\max_k n_k^{(N)}(t)\le R_1(T),
\qquad
\sup_{t\in[0,T]}\max_k c_k^{(N)}(t)\le R_2(T).
\end{align*}
\end{corollary}

\begin{proof}
Cell averaging preserves nonnegativity and satisfies
\begin{align*}
M_x(W^{(N)})\le M_x(W),
\qquad
\|\Pi_Nu\|_{L^\infty}\le\|u\|_{L^\infty}
\quad (u\in L^\infty(I)).
\end{align*}
The drift $\mathcal F_{W^{(N)}}$ maps the space of $\mathcal P_N$-step functions into itself, and its restriction is the discrete system by Lemma~\ref{lem:recon}. Applying Proposition~\ref{prop:graphon_bounds}\textup{(ii)} to $W^{(N)}$ and $\Pi_NU_{\mathrm{in}}$ gives global existence, uniqueness, nonnegativity, and the stated bounds.
\end{proof}

The bounds $R_1(T)$ and $R_2(T)$ established above depend on $T$ and grow exponentially in time. However, for logistic population dynamics, the concavity of the reaction term ensures that solutions remain uniformly bounded independent of time. This global boundedness holds uniformly across the continuous and discrete settings regardless of the chemotactic sensitivity.

\begin{corollary}[Uniform bounds for logistic growth]
\label{cor:logistic-uniform-bounds}
Assume the hypotheses of Proposition~\ref{prop:graphon_bounds}~\textup{(ii)}, with $f(n)=rn(1-n/\bar n)$, $r,\bar n>0$, and $\beta>0$. Set
\begin{align*}
    R_n:=\max\Big\{\|n_{\mathrm{in}}\|_{L^\infty},
                 \bar n\Big(1+\frac{\eta_nM_x(W)}r\Big)\Big\},
                 \qquad
    R_c:=\max\Big\{\|c_{\mathrm{in}}\|_{L^\infty},
                            \frac{\alpha R_n}{\beta}\Big\}.
\end{align*}
Then $\|n(t)\|_{L^\infty}\le R_n$ and $\|c(t)\|_{L^\infty}\le R_c$ for all $t\ge0$.
The same bounds hold for the discrete solutions with cell-average weights and initial data, uniformly in $N$. They also apply to the sampled graphs of Section~\ref{ssec:random-limit}, substituting the uniform row-sum bound $1$ for $M_x(W)$.
\end{corollary}

\begin{proof}
By definition, $R_n\ge\bar n$, which implies that $f$ is decreasing on $[R_n,\infty)$ and $\eta_nM_x(W)R_n+f(R_n)\le0$. On the set $\{n>R_n\}$, the nonnegativity of $n$ and the bound $0\le S\le1$ give
\begin{align*}
    \partial_t(n-R_n)
    \le \eta_nM_x(W)\|(n-R_n)^+\|_{L^\infty}.
\end{align*}
Lemma~\ref{lem:app-integral-comparison} and Gr\"onwall's inequality imply $n\le R_n$. Since $\alpha R_n-\beta R_c\le0$, it follows that on $\{c>R_c\}$,
\begin{align*}
    \partial_t(c-R_c)
    \le\frac{\eta_cM_x(W)}2\|(c-R_c)^+\|_{L^\infty}.
\end{align*}
The same comparison principle gives $c\le R_c$. The discrete bounds follow identically by replacing spatial integrals with normalized sums. For the sampled graphs, all normalized row sums are bounded by $1$.
\end{proof}

%%%%%%%%%%%%%%%%%%%%%%%%%%%%%%%%%%%%%%%%%%%%%%%%%%%%%%%%%%
\subsection{Deterministic finite graph limit in \texorpdfstring{$L^p$}{Lp}}
\label{ssec:Det_lim}

We first prove the deterministic graphon limit. This result establishes that when the graph is formed using the canonical weighted step approximation of $W$, the time-evolving solution of the $N$-node Keller--Segel model converges to the continuous graphon solution in $C([0,T];L^p(I)^2)$ for every $1\le p<\infty$. This represents the cleanest version of the continuum limit: it involves no random sampling noise, meaning the only sources of error arise from approximating the initial data via cell averages and approximating the kernel by $W^{(N)}$.

The proof follows a stability--consistency scheme. The stability part is Lemma~\ref{lem:Lip-Lp}: on an $L^\infty$-bounded set of states, the nonlinear drift $\mathcal F_K$ is Lipschitz in $L^p$. The consistency part is Lemma~\ref{lem:Lp-cons}: evaluating the drift with $W^{(N)}$ rather than $W$ produces an error controlled by $\|W^{(N)}-W\|_{L^p(I^2)}$. The exponential preference term does not require any special compactness argument; on bounded sets it is handled by the global bounds $0\le S\le 1$ and $|\partial_aS|=|\partial_bS|\le |\chi|/4$. Once these two estimates are available, the theorem follows by subtracting the two mild formulations and applying Gr\"{o}nwall's inequality.

\begin{lemma}[$L^p$ stability]\label{lem:Lip-Lp}
Let $1\le p\le\infty$, $\|U_i\|_{L^\infty(I)^2}\le R$, and
$M_x(|K|),M_y(|K|)<\infty$. Set
\begin{align*}
\kappa_{K,p}:=\|T_{|K|}\|_{L^p\to L^p}+M_x(|K|).
\end{align*}
Let $L_f(R)$ be a Lipschitz constant of $f$ on $[-R,R]$.
Then there exists $C_{R,p}(K)<\infty$, depending only on $R$, $L_f(R)$, the model parameters, and $\kappa_{K,p}$, such that
\begin{align*}
\|\mathcal F_K[U_1]-\mathcal F_K[U_2]\|_{L^p(I)^2}
\le C_{R,p}(K)\|U_1-U_2\|_{L^p(I)^2}.
\end{align*}
\end{lemma}

\begin{proof}
Write $\widetilde n=n_1-n_2$, $\widetilde c=c_1-c_2$, and
$S_i=S(c_i(x),c_i(y))$. The $c$-component is bounded by
\begin{align*}
\left(\beta+\frac{\eta_c}{2}\kappa_{K,p}\right)\|\widetilde c\|_{L^p}
+\alpha\|\widetilde n\|_{L^p}.
\end{align*}
For the $n$-component, use
\begin{align*}
S_1\widetilde n(y)-(1-S_1)\widetilde n(x)
+(S_1-S_2)\big(n_2(y)+n_2(x)\big)
\end{align*}
and
$|S_1-S_2|\le\frac{|\chi|}{4}(|\widetilde c(x)|+|\widetilde c(y)|)$.
The operator bounds for $T_{|K|}$ then give
\begin{align*}
\|F_K^{(n)}(U_1)-F_K^{(n)}(U_2)\|_{L^p}
\le
\big(L_f(R)+\eta_n\kappa_{K,p}\big)\|\widetilde n\|_{L^p}
+\frac{\eta_n|\chi|R}{2}\kappa_{K,p}\|\widetilde c\|_{L^p}.
\end{align*}
Combining the two components proves the claim.
\end{proof}

\begin{remark}[Uniformity for step kernels]\label{rmk:step-Lip}
If $M_x(|W|),M_y(|W|)<\infty$, then
\begin{align*}
M_x(|W^{(N)}|)\le M_x(|W|),
\qquad
M_y(|W^{(N)}|)\le M_y(|W|),
\end{align*}
and hence $\sup_N C_{R,p}(W^{(N)})<\infty$.
For $p=2$, it is enough that $W\in L^2(I^2)$ and $M_x(|W|)<\infty$, since
$\kappa_{W^{(N)},2}\le\|W\|_{L^2(I^2)}+M_x(|W|)$.
\end{remark}

\begin{lemma}[$L^p$ consistency]\label{lem:Lp-cons}
Let $1\le p<\infty$, $\|U\|_{L^\infty(I)^2}\le R$, and $W\in L^p(I^2)$. Then
\begin{align*}
\|\mathcal F_{W^{(N)}}[U]-\mathcal F_W[U]\|_{L^p(I)^2}
\le 2R\left\|\left(\frac{\eta_c}{2},\eta_n\right)\right\|_{\ell^p}
\|W^{(N)}-W\|_{L^p(I^2)}.
\end{align*}
\end{lemma}

\begin{proof}
Set $H=W^{(N)}-W$ and $a(x):=\int_I|H(x,y)|\,dy$.
Since $|n|,|c|\le R$ and $0\le S\le1$,
\begin{align*}
|F_{W^{(N)}}^{(c)}(U)(x)-F_W^{(c)}(U)(x)|\le\eta_cR\,a(x),
\qquad
|F_{W^{(N)}}^{(n)}(U)(x)-F_W^{(n)}(U)(x)|\le2\eta_nR\,a(x).
\end{align*}
H\"older's inequality gives $\|a\|_{L^p(I)}\le\|H\|_{L^p(I^2)}$. Taking the product $L^p$ norm proves the estimate.
\end{proof}

We now assemble the two estimates. The theorem is stated with an explicit $L^\infty$ boundedness hypothesis so that it also applies to signed kernels or to settings where positivity is not used.
Under the structural nonnegativity assumptions of Proposition~\ref{prop:graphon_bounds}\textup{(ii)}, this hypothesis is automatically satisfied on every finite time interval by that proposition and Corollary~\ref{cor:discrete_bounds}.

\begin{theorem}[Deterministic graphon limit in $L^p$]
\label{thm:det-Lp-conv}
Fix $T>0$ and $1\le p<\infty$. Assume
\begin{align*}
W\in L^p(I^2),
\qquad
M_x(|W|)+M_y(|W|)<\infty,
\qquad
U_{\mathrm{in}}\in L^\infty(I)^2.
\end{align*}
Let $U$ and $U^{(N)}$ be the graphon and discrete solutions driven by $W$ and $W^{(N)}$,
respectively, with initial data $U_{\mathrm{in}}$ and $\Pi_NU_{\mathrm{in}}$. Suppose that
\begin{align*}
\sup_N\sup_{t\in[0,T]}
\Big(\|U(t)\|_{L^\infty(I)^2}+\|\Pi_NU^{(N)}(t)\|_{L^\infty(I)^2}\Big)
\le R_T.
\end{align*}
Then, for some $C_T<\infty$ independent of $N$,
\begin{align*}
\sup_{t\in[0,T]}\|\Pi_NU^{(N)}(t)-U(t)\|_{L^p(I)^2}
\le
e^{C_TT}
\Big(
\|\Pi_NU_{\mathrm{in}}-U_{\mathrm{in}}\|_{L^p(I)^2}
+C_TT\|W^{(N)}-W\|_{L^p(I^2)}
\Big).
\end{align*}
In particular, $\Pi_NU^{(N)}\to U$ in $C([0,T];L^p(I)^2)$.
\end{theorem}

\begin{proof}
Set $E_N(t):=\Pi_NU^{(N)}(t)-U(t)$.
By Lemma~\ref{lem:recon}, $\Pi_NU^{(N)}$ satisfies the mild equation with kernel $W^{(N)}$. Subtracting the two mild formulations gives
\begin{align*}
E_N(t)
&=E_N(0)
+\int_0^t
\Big(
\mathcal F_{W^{(N)}}[\Pi_NU^{(N)}(s)]
-\mathcal F_{W^{(N)}}[U(s)]
\Big)\,ds
+\int_0^t
\Big(
\mathcal F_{W^{(N)}}[U(s)]
-\mathcal F_W[U(s)]
\Big)\,ds.
\end{align*}
By the assumed $L^\infty$ bounds, Lemmas~\ref{lem:Lip-Lp} and~\ref{lem:Lp-cons}, and Remark~\ref{rmk:step-Lip},
\begin{align*}
\|E_N(t)\|_{L^p(I)^2}
&\le
\|E_N(0)\|_{L^p(I)^2}
+C_{R_T,p}\int_0^t\|E_N(s)\|_{L^p(I)^2}\,ds
+C_{R_T,p}^*\,t\,\|W^{(N)}-W\|_{L^p(I^2)},
\end{align*}
with constants independent of $N$. Gr\"onwall's inequality gives the stated estimate. Finally, Lemma~\ref{lem:PiN-Lp} imply $\|E_N(0)\|_{L^p(I)^2}\to0$, and $\|W^{(N)}-W\|_{L^p(I^2)}\to0$, which proves the convergence.
\end{proof}

\begin{remark}[A general deterministic graphon limit formulation]
\label{rmk:general-KN}
The cell-average form of $W^{(N)}$ is not needed in the stability--consistency argument. Retain the assumptions of Theorem~\ref{thm:det-Lp-conv} on $W$ and its solution $U$, and let $U_N$ solve the equation with a deterministic kernel $K_N$ and initial data $U_{\mathrm{in},N}$. Suppose that the solutions are uniformly bounded in $L^\infty(I)^2$ on $[0,T]$ and
\begin{align*}
\sup_N\big(M_x(|K_N|)+M_y(|K_N|)\big)<\infty,
\qquad
\|K_N-W\|_{L^p(I^2)}
+\|U_{\mathrm{in},N}-U_{\mathrm{in}}\|_{L^p(I)^2}\to0.
\end{align*}
Then $U_N\to U$ in $C([0,T];L^p(I)^2)$ by the same estimates. This formulation also covers the pointwise discretization in Appendix~\ref{app:R4-pointwise}.
\end{remark}

\subsection{Random graph limit in \texorpdfstring{$L^p$}{Lp}}
\label{ssec:random-limit}

We next pass from deterministic weighted graphs to independent-edge dense random graphs sampled from the cell-averaged kernel $W^{(N)}$. This corresponds to the dense stochastic block model on the fixed partition $\mathcal P_N$. Under the canonical grid labeling identifying nodes with intervals ($k \leftrightarrow I_k$), adjacency matrices are naturally identified with step kernels on $I^2$. A relabeling-invariant formulation in the cut distance is provided in Section~\ref{ssec:cutnorm-limits}.

The deterministic $L^p$ convergence theorem cannot be directly applied to the empirical kernel $W^{(N)}_\xi$ because the pointwise Bernoulli fluctuations are of order one in $L^p(I^2)$. Instead, the sampled-graph solution is first compared with the deterministic weighted-graph solution. Evaluating the kernel fluctuation on the deterministic solution yields centered row sums with bounded deterministic coefficients. Fixed-time moment estimates control the time integral of this row-wise fluctuation, yielding an $O_{\pr}(N^{-1/2})$ sampled-to-weighted error for each fixed $1\le p<\infty$. Combining this high-probability estimate with the deterministic approximation error establishes convergence to the graphon solution.

\subsubsection{Random graph sampling and the associated discrete dynamics}
\label{sssec:RG-setup}

Throughout we assume that $W\in L^\infty(I^2)$ satisfies $0\le W\le 1$ a.e.\ and $W(x,y)=W(y,x)$ a.e. We also assume~\eqref{eq:f-structure} and take $U_{\mathrm{in}}\in L^\infty(I)^2$ to be deterministic and componentwise nonnegative.
We use the partition $\mathcal P_N$ and the cell-average weights defined above. Since $0\le W\le1$, these weights can be used as edge probabilities; we write $\mathsf p^{(N)}_{k\ell}:=\mathsf w^{(N)}_{k\ell}$ for $1\le k,\ell\le N$.

We sample an undirected graph from these block probabilities. For notational convenience, we also sample diagonal variables: let $\{\xi^{(N)}_{k\ell}\}_{1\le k\le \ell\le N}$ be independent Bernoulli random variables with
\begin{align*}
\E \xi^{(N)}_{k\ell}=\mathsf{p}^{(N)}_{k\ell},
\end{align*}
and set $\xi^{(N)}_{\ell k}:=\xi^{(N)}_{k\ell}$. The diagonal variables do not affect the dynamics, since all $k=\ell$ summands vanish in both the $n$- and $c$-equations. Setting $\xi^{(N)}_{kk}=0$ leaves both discrete dynamics unchanged. It changes the empirical kernel by at most $N^{-1}$ in $L^1(I^2)$ and hence in cut norm.
Define the empirical step kernel of the sampled graph by
\begin{align*}
W^{(N)}_\xi(x,y):=\xi^{(N)}_{k\ell},
\quad (x,y)\in I_k\times I_\ell.
\end{align*}
Equivalently, $W^{(N)}_\xi$ is the empirical graphon under the canonical labeling $k\leftrightarrow I_k$.

This is the independent-edge dense stochastic block model on the fixed partition $\mathcal P_N$. In the latent-position $W$-random graph model, the nodes are instead sampled as i.i.d.\ uniform points $X_1,\dots,X_N$, with edge probabilities $W(X_k,X_\ell)$ for $k<\ell$. These pairs avoid each fixed null subset of $I^2$ almost surely, so that construction depends only on the a.e.\ equivalence class of $W$. The theorem below concerns the fixed-partition construction, not the latent-position model. Here $\E W^{(N)}_\xi=W^{(N)}$, so the random-graph kernel error splits into the deterministic discretization term $W^{(N)}-W$ and the fluctuation term $W^{(N)}_\xi-W^{(N)}$.

\medskip
\noindent\textbf{DWG vs.\ sampled-graph dynamics.}
Let $U^{(N)}=(n^{(N)},c^{(N)})$ be the solution of the discrete system \eqref{eq:KS-discrete}
with weights $\mathsf w^{(N)}_{k\ell}=\mathsf{p}^{(N)}_{k\ell}$ (the DWG dynamics), and let
$U^{\xi,(N)}=\big(n^{\xi,(N)},c^{\xi,(N)}\big)$ be the solution of the same system with weights
$\mathsf w^{(N)}_{k\ell}=\xi^{(N)}_{k\ell}$ (the sampled-graph dynamics).
Both systems start from the same deterministic cell-average initial data.

By the reconstruction argument of Lemma~\ref{lem:recon}, their step reconstructions solve the mild formulations
driven by $W^{(N)}$ and $W^{(N)}_\xi$, respectively:
\begin{align*}
\partial_t\Pi_N U^{(N)}(t)=\mathcal F_{W^{(N)}}[\Pi_N U^{(N)}(t)],
\qquad
\partial_t\Pi_N U^{\xi,(N)}(t)=\mathcal F_{W^{(N)}_\xi}[\Pi_N U^{\xi,(N)}(t)].
\end{align*}

\medskip
\noindent\textbf{Uniform bounds.}
Since $0\le \mathsf{p}^{(N)}_{k\ell}\le 1$ and $0\le \xi^{(N)}_{k\ell}\le 1$, we have the row-sum bound
\begin{align*}
\max_{1\le k\le N}\ \frac{1}{N}\sum_{\ell=1}^N \mathsf{p}^{(N)}_{k\ell}\le 1,
\qquad
\max_{1\le k\le N}\ \frac{1}{N}\sum_{\ell=1}^N \xi^{(N)}_{k\ell}\le 1.
\end{align*}
The space of step functions on $\mathcal P_N$ is invariant for both kernels. Proposition~\ref{prop:graphon_bounds}~\textup{(ii)}, applied to $W^{(N)}$ and $W_\xi^{(N)}$ with row-mass bound $1$, gives bounds for both reconstructed solutions.
In particular, for each fixed $T>0$, there exists a deterministic constant $R_T<\infty$ such that
\begin{align*}
\sup_{t\in[0,T]}\Big(\|\Pi_N U^{(N)}(t)\|_{L^\infty(I)^2}+\|\Pi_NU^{\xi,(N)}(t)\|_{L^\infty(I)^2}\Big)
\le R_T,
\end{align*}
uniformly in $N$ and (for $U^{\xi,(N)}$) uniformly over the graph realization. All subsequent constants depending only on this bound can be chosen to be deterministic and independent of $N$ and of the realization of $\xi$.

\subsubsection{Row-wise fluctuation and moment estimates}
\label{sssec:rowwise}

Since the coefficients $g^{(c)}_{k\ell}(t)$ and $g^{(n)}_{k\ell}(t)$ are evaluated on the deterministic weighted-graph solution, all randomness arises from the centered Bernoulli variables $\xi^{(N)}_{k\ell}-\mathsf p^{(N)}_{k\ell}$. Bounding these deterministic coefficients allows Hoeffding's inequality to control the moments of the row-wise sums. Applying Minkowski's inequality then directly yields an integrated $L^p$ bound for the temporal fluctuation.

For the deterministic weighted-graph solution $U^{(N)}$, set
\begin{align*}
    g^{(c)}_{k\ell}(t)&:=c^{(N)}_\ell(t)-c^{(N)}_k(t),\\
    g^{(n)}_{k\ell}(t)&:=
    S\big(c^{(N)}_k(t),c^{(N)}_\ell(t)\big)n^{(N)}_\ell(t)
    -\big(1-S\big(c^{(N)}_k(t),c^{(N)}_\ell(t)\big)\big)n^{(N)}_k(t),\\
    \Xi^{(j)}_k(t)&:=\frac1N\sum_{\ell=1}^N
       \big(\xi^{(N)}_{k\ell}-\mathsf p^{(N)}_{k\ell}\big)
       g^{(j)}_{k\ell}(t),\quad j\in\{n,c\}.
\end{align*}
Define $\Xi^{(N)}(t)$ by its two components at each node:
\begin{align*}
    \Xi^{(N)}_k(t)
    :=\Big(\eta_n\Xi^{(n)}_k(t), \frac{\eta_c}{2}\Xi^{(c)}_k(t)\Big).
\end{align*}
Then $\Pi_N\Xi^{(N)}$ is the difference between the sampled and weighted vector fields evaluated on $\Pi_NU^{(N)}$.

\begin{lemma}[Integrated fluctuation estimate]\label{lem:row-conc}
Fix $T>0$ and $1\le p<\infty$. There is a deterministic constant $C_{T,p}<\infty$, independent of $N$ and $q$, such that, for $q\ge\max\{2,p\}$,
\begin{align*}
    \E\bigg[
       \Big(\int_0^T\|\Pi_N\Xi^{(N)}(t)\|_{L^p(I)^2}\,dt\Big)^q
    \bigg]^{1/q}
    \le C_{T,p}\sqrt{\frac qN}.
\end{align*}
\end{lemma}

\begin{proof}
The coefficients $g^{(j)}_{k\ell}(t)$ are deterministic and bounded by $2R_T$. For fixed $k,t,j$, the summands in $\Xi^{(j)}_k(t)$ are independent and centered. Hoeffding's inequality gives
\begin{align*}
    \pr\big(|\Xi^{(j)}_k(t)|>u\big)
    \le 2\exp\Big(-\frac{Nu^2}{8R_T^2}\Big).
\end{align*}
Integrating this bound yields $\E\big[|\Xi^{(j)}_k(t)|^q\big]^{1/q}\le C_T\sqrt{q/N}$ for $q\ge2$, uniformly in $k,t,j$. For $q\ge p$, Minkowski's inequality gives
\begin{align*}
    \big(\E\|\Pi_N\Xi^{(N)}(t)\|_{L^p(I)^2}^q\big)^{1/q}
    \le C\bigg[
       \frac1N\sum_{k=1}^N\sum_{j\in\{n,c\}}
       \big(\E|\Xi^{(j)}_k(t)|^q\big)^{p/q}
    \bigg]^{1/p}
    \le C_{T,p}\sqrt{\frac qN}.
\end{align*}
A second application of Minkowski's inequality to the time integral proves the claim. Independence between different rows or times is not used.
\end{proof}

\subsubsection{Random graphon limit: stability and convergence}
\label{sssec:RG-limit}

We now convert the row-wise fluctuation estimate into the convergence of solutions. Because the sampled and deterministic weighted graph systems share the same initial data, the evolution of their difference is governed solely by the random fluctuation $\Xi^{(N)}$ and a Lipschitz term proportional to the difference between their states. Applying Gr\"onwall's inequality ensures that the resulting dynamic error remains of the same order as the row-wise fluctuation. Finally, combining this high-probability estimate with the deterministic limit established earlier yields the full convergence to the graphon solution.

\begin{theorem}[Random graphon limit]\label{thm:RG-limit}
Assume the random-graph set-up above. Fix $T>0$ and $1\le p<\infty$. There is a deterministic $C_{T,p}<\infty$ such that, for every $\delta\in(0,1)$, with probability at least $1-\delta$,
\begin{align*}
    \sup_{t\in[0,T]}
    \|\Pi_NU^{\xi,(N)}(t)-U(t)\|_{L^p(I)^2}
    \le
    \sup_{t\in[0,T]}\|\Pi_NU^{(N)}(t)-U(t)\|_{L^p(I)^2}
    +C_{T,p}\sqrt{\frac{\log(2/\delta)}{N}}.
\end{align*}
Consequently, $\Pi_NU^{\xi,(N)}\to U$ in probability in $C([0,T];L^p(I)^2)$.
\end{theorem}

\begin{proof}
Set $d_N(t):=\|\Pi_N(U^{\xi,(N)}-U^{(N)})(t)\|_{L^p(I)^2}$. Subtracting the mild equations and evaluating the kernel fluctuation on $\Pi_NU^{(N)}$, and noting that the sampled kernels have row and column mass at most $1$, Lemma~\ref{lem:Lip-Lp} gives
\begin{align*}
    d_N(t)\le C_{T,p}\int_0^t d_N(s)\,ds
        +\int_0^t\|\Pi_N\Xi^{(N)}(s)\|_{L^p(I)^2}\,ds.
\end{align*}
Gr\"onwall's inequality therefore implies
\begin{align*}
    \sup_{t\in[0,T]}d_N(t)
    \le e^{C_{T,p}T}
       \int_0^T\|\Pi_N\Xi^{(N)}(s)\|_{L^p(I)^2}\,ds.
\end{align*}
By Lemma~\ref{lem:row-conc}, for $q\ge\max\{2,p\}$,
\begin{align*}
    \E\bigg[
       \Big(\sup_{t\in[0,T]}d_N(t)\Big)^q
    \bigg]^{1/q}
    \le C_{T,p}\sqrt{\frac qN}.
\end{align*}
Choose $q=\max\{2,p,\log(1/\delta)\}$. Markov's inequality gives
\begin{align*}
    \pr\bigg(\sup_{t\in[0,T]}d_N(t)
          >eC_{T,p}\sqrt{\frac qN}\bigg)
    \le e^{-q}\le\delta.
\end{align*}
For fixed $p$, $q\le C_p\log(2/\delta)$, so the triangle inequality gives the stated bound after appropriately increasing $C_{T,p}$. Theorem~\ref{thm:det-Lp-conv} proves the convergence in probability.
\end{proof}

\begin{remark}[On rates]\label{rmk:rates}
The estimate in Theorem~\ref{thm:RG-limit} separates the total error into a deterministic discretization term $\sup_{t\in[0,T]}\|\Pi_N U^{(N)}(t)-U(t)\|_{L^p(I)^2}$ and a random fluctuation term of order $\sqrt{\log(2/\delta)/N}$. For every fixed $T>0$ and $1\le p<\infty$, the sampled-to-weighted solution error is $O_{\pr}(N^{-1/2})$.

If $W$ is Lipschitz on $I^2$ and the components of $U_{\mathrm{in}}$ are mildly regular (e.g., Lipschitz or in $W^{1,p}(I)^2$), then the deterministic approximation error in Theorem~\ref{thm:det-Lp-conv} satisfies
\begin{align*}
    \|W^{(N)}-W\|_{L^p(I^2)}
    +\|\Pi_NU_{\mathrm{in}}-U_{\mathrm{in}}\|_{L^p(I)^2}
    =O(N^{-1}).
\end{align*}
In this regime, the fluctuation term dominates, and the total solution error is $O_{\pr}(N^{-1/2})$.
\end{remark}

%%%%%%%%%%%%%%%%%%%%%%%%%%%%%%%%%%%%%%%%%%%%%%%%%%%%%%%%%%
\subsection{Cut-norm formulation of graphon limits}
\label{ssec:cutnorm-limits}

The deterministic and random limits established above rely on the canonical labeling of vertices induced by the partition $\mathcal P_N$. While this provides a convenient setting for comparing solutions in $L^p(I)^2$, it does not reflect the intrinsic geometry of dense graph limit theory. Because graphons are defined only up to measure-preserving relabelings, the convergence of unlabeled dense graphs is most naturally expressed using the cut metric.

In this subsection, we present a relabeling-invariant formulation. We show that if a kernel sequence $K_N$ converges to $W$ in the cut distance and the initial data converge in $L^1(I)^2$ under the same measure-preserving relabelings, then the corresponding Keller--Segel solutions converge in $C([0,T];L^1(I)^2)$. Furthermore, by interpolating the difference between the solutions, the uniform $L^\infty$ bounds established previously extend this $L^1$ convergence to yield finite-time convergence in $C([0,T];L^p(I)^2)$ for every $1 \le p < \infty$. The kernel assumptions for this result require nonnegativity and uniform bounds on row and column integrals, but not a pointwise bound by $1$. This formulation is applicable to networks where vertex labels lack intrinsic meaning, connecting the sampled step graphs with the standard cut-distance topology.

\subsubsection{Cut norm, cut distance, and relabeling}

\begin{definition}[Cut norm and the $\infty\to 1$ operator norm]
Let $K\in L^1(I^2)$. Define the cut norm
\begin{align*}
  \|K\|_{\square}
  :=\sup_{A,B\subset I}
  \left|\int_A\int_B K(x,y)\,dx\,dy\right|,
\end{align*}
where the supremum is taken over all measurable sets $A,B\subset I$.
For the integral operator $T_K$ defined above, set
\begin{align*}
  \|K\|_{\infty\to 1}
  :=\sup_{\|f\|_{L^\infty(I)}\le 1}\ \|T_K f\|_{L^1(I)}
  =\sup_{\|f\|_{L^\infty}\le 1,\ \|g\|_{L^\infty}\le 1}
  \left|\int_{I^2} K(x,y)\,f(y)\,g(x)\,dx\,dy\right|.
\end{align*}
\end{definition}

\begin{remark}[Useful comparisons]\label{rmk:cut-compare}
These norms are equivalent up to absolute constants:
\begin{align*}
  \|K\|_{\square}\le \|K\|_{\infty\to 1}\le 4\|K\|_{\square}.
\end{align*}
Moreover, since $|I^2|=1$, one has
\begin{align*}
  \|K\|_{\square}\le \|K\|_{L^1(I^2)}\le \|K\|_{L^p(I^2)}
  \text{ for every }1\le p\le\infty.
\end{align*}
\end{remark}

\begin{definition}[Relabeling and cut distance]
For a measure-preserving bijection $\varphi:I\to I$, define the pullbacks
\begin{align*}
  K^\varphi(x,y):=K(\varphi(x),\varphi(y)),
  \qquad
  U^\varphi(x):=U(\varphi(x)).
\end{align*}
The cut distance between kernels is
\begin{align*}
  \delta_{\square}(K,W):=\inf_{\varphi}\ \|K^\varphi-W\|_{\square},
\end{align*}
where the infimum is taken over all measure-preserving bijections $\varphi$.
\end{definition}

\begin{lemma}[Equivariance under relabeling]
\label{lem:equivariance}
Let $K\in L^1(I^2)$. For every measure-preserving bijection $\varphi$ and every bounded state $U=(n,c)$,
\begin{align*}
  \mathcal F_{K^\varphi}[U^\varphi]=(\mathcal F_K[U])^\varphi.
\end{align*}
In particular, if $U$ solves $\partial_t U=\mathcal F_K[U]$ with initial data $U_{\mathrm{in}}$, then $U^\varphi$ solves
$\partial_t U^\varphi=\mathcal F_{K^\varphi}[U^\varphi]$ with initial data $U_{\mathrm{in}}^\varphi$.
\end{lemma}

\begin{proof}
This follows by a change of variables in the defining integrals of \eqref{eq:KS-conti} and the measure-preserving property of $\varphi$.
\end{proof}

\subsubsection{A cut-norm consistency bound}

We apply the fundamental theorem of calculus in the first argument of $S$ to write the sensitivity factor as an integral of separated functions. This gives the following kernel estimate.

\begin{lemma}[Cut-norm consistency]
\label{lem:cut-consistency}
Assume $K,W\in L^1(I^2)$, and $\|U\|_{L^\infty(I)^2}\le R$.
Then
\begin{align*}
    \|F_K^{(c)}(U)-F_W^{(c)}(U)\|_{L^1}
    &\le \eta_c R \|K-W\|_{\infty\to1},\\
    \|F_K^{(n)}(U)-F_W^{(n)}(U)\|_{L^1}
    &\le \eta_n(2R+|\chi|R^2)\|K-W\|_{\infty\to1}.
\end{align*}
Consequently,
\begin{align*}
    \|\mathcal F_K[U]-\mathcal F_W[U]\|_{L^1(I)^2}
    \le C_{R,\chi}\|K-W\|_{\infty\to1},
\end{align*}
where $C_{R,\chi}$ depends only on $R$ and the model parameters.
\end{lemma}

\begin{proof}
See Appendix~\ref{app:cut-consistency-proof}.
\end{proof}

\subsubsection{Deterministic cut-distance graphon limit}

\begin{theorem}[Deterministic cut-distance graphon limit]
\label{thm:det-cut}
Let $W,K_N\in L^1(I^2)$ be nonnegative kernels satisfying $M_x(|W|) + M_y(|W|)+\sup_N \big( M_x(|K_N|) + M_y(|K_N|) \big) < \infty$. Fix $T>0$, and assume $\delta_\square(K_N,W)\to 0$.
Let $U$ and $U_N$ solve $\partial_t U=\mathcal F_W[U]$ and $\partial_t U_N=\mathcal F_{K_N}[U_N]$ on $[0,T]$, with initial data $U_{\mathrm{in}}$ and $U_{\mathrm{in},N}$ in $L^\infty(I)^2$.
Assume there exists $R_T<\infty$ such that
\begin{align*}
    \sup_N\sup_{t\in[0,T]} \Big( \|U(t)\|_{L^\infty(I)^2} +\|U_N(t)\|_{L^\infty(I)^2} \Big) \le R_T.
\end{align*}

Let $\varphi_N$ be measure-preserving bijections satisfying
\begin{align*}
    \|K_N^{\varphi_N}-W\|_\square \le \delta_\square(K_N,W)+N^{-1}.
\end{align*}
If $\|U_{\mathrm{in},N}^{\varphi_N}-U_{\mathrm{in}}\|_{L^1(I)^2}\to0$, then there exists a constant $C_T<\infty$ independent of $N$ such that
\begin{align*}
    \sup_{t\in[0,T]}\|U_N^{\varphi_N}(t)-U(t)\|_{L^1(I)^2}
    \le
    e^{C_TT}
    \Big(
    \|U_{\mathrm{in},N}^{\varphi_N}-U_{\mathrm{in}}\|_{L^1(I)^2}
    +C_T\|K_N^{\varphi_N}-W\|_{\infty\to1}
    \Big).
\end{align*}
Consequently, $U_N^{\varphi_N}\to U$ in $C([0,T];L^p(I)^2)$ for every $1\le p<\infty$.
\end{theorem}

\begin{remark}[Relabeling of the initial data]
\label{rmk:cut-initial}
The same bijections must align the kernels and the initial data. This condition holds when the initial data are constant. Under the identity relabeling, it holds for $U_{\mathrm{in},N}=\Pi_NU_{\mathrm{in}}$ by Lemma~\ref{lem:PiN-Lp}.
\end{remark}

\begin{proof}[Proof of Theorem~\ref{thm:det-cut}]
The uniform integral bounds are invariant under relabeling:
\begin{align*}
    M_x(K_N^{\varphi_N})=M_x(K_N),
    \qquad 
    M_y(K_N^{\varphi_N})=M_y(K_N).
\end{align*}
Hence, Lemma~\ref{lem:Lip-Lp} applies with a constant independent of $N$.
By Lemma~\ref{lem:equivariance}, $\widetilde U_N:=U_N^{\varphi_N}$ satisfies the mild equation with the kernel $K_N^{\varphi_N}$. Subtracting the mild formulation for $U$ and applying Lemmas~\ref{lem:Lip-Lp} and~\ref{lem:cut-consistency} yields
\begin{align*}
    \|\widetilde U_N(t)-U(t)\|_{L^1(I)^2}
    &\le
    \|U_{\mathrm{in},N}^{\varphi_N}-U_{\mathrm{in}}\|_{L^1(I)^2}
    +C_T\int_0^t\|\widetilde U_N(s)-U(s)\|_{L^1(I)^2}\,ds\\
    &\quad+C_Tt\|K_N^{\varphi_N}-W\|_{\infty\to1}.
\end{align*}
Gr\"onwall's inequality provides the stated estimate. Norm equivalence implies
\begin{align*}
    \|K_N^{\varphi_N}-W\|_{\infty\to1}
    \le 4\|K_N^{\varphi_N}-W\|_\square
    \le 4\delta_\square(K_N,W)+4N^{-1}.
\end{align*}
Since this bound tends to zero, convergence holds in $C([0,T];L^1(I)^2)$. For $1<p<\infty$, interpolation with the uniform bounds yields
\begin{align*}
    \sup_{t\in[0,T]}
    \|U_N^{\varphi_N}(t)-U(t)\|_{L^p(I)^2}^p
    \le R_T^{p-1}\sup_{t\in[0,T]}
       \|U_N^{\varphi_N}(t)-U(t)\|_{L^1(I)^2},
\end{align*}
which completes the proof.
\end{proof}

\paragraph{Sampled graphs in cut distance.}
In the random-graph set-up of Section~\ref{ssec:random-limit}, the empirical step kernels $W^{(N)}_\xi$ concentrate around their mean step kernels $W^{(N)}$ in cut norm. Combined with the $L^1$ convergence $W^{(N)}\to W$, this gives $\|W^{(N)}_\xi-W\|_\square\to0$ in probability under the identity relabeling. The cell-average initial data are compatible with this labeling, so the stability estimate in the proof of Theorem~\ref{thm:det-cut} applies without any additional relabeling condition. The concentration estimate and the resulting convergence of solutions are stated and proved in Appendix~\ref{app:R5-cutconc} (Lemma~\ref{lem:cut-conc-step} and Corollary~\ref{cor:rand-cut}).

The results in this section justify the use of our graphon Keller--Segel model as the continuum limit of the graph metapopulation model on families of dense deterministic and random graphs, at least for the behavior of solutions on fixed time intervals $[0,T]$. This suggests that our results in Section \ref{sec:graph-model-stability} for the instability thresholds and the onset of pattern formation should reflect similar emergence of patterns on sufficiently large discrete graph models. While we have not demonstrated the convergence of long-time behavior for the graph model to the graphon model, our analysis of long-time behavior in the graphon model in the remaining sections can provide directions for future work regarding the correspondence of asymptotic behavior for the graph and graphon Keller--Segel models. We provide additional discussion of these potential research directions in Section~\ref{sec:discussiongraphonconvergence}.

%%%%%%%%%%%%%%%%%%%%%%%%%%%%%%%%%%%%%%%%%%%%%%%%%%%%%%%%%%
%%%%%%%%%%%%%%%%%%%%%%%%%%%%%%%%%%%%%%%%%%%%%%%%%%%%%%%%%%
\section{Stability Analysis for the Graphon Model}\label{sec:graph-model-stability}

In this section, we present our analytical characterization of the pattern-forming behavior of the graphon Keller--Segel model. We start by performing a linear stability analysis in Section \ref{sec:linear-stability-graphon}, illustrating the threshold condition on the strength of biased motion $\chi$ of individuals towards the patch with increased chemical concentration required to initiate the emergence of patterns in our graphon metapopulation. We then present a weakly nonlinear stability analysis for our graphon model in Section \ref{sec:graphon-weakly-general}, deriving general conditions for determining the conditions under which the pattern-forming instability constitutes a supercritical or subcritical pitchfork bifurcation and characterizing the reaction functions $f(n)$ for which spatial patterns are expected to result in increased or decreased population size relative to the uniform equilibrium value near the onset of instability. In Section \ref{sec:graphon-weakly-alternative}, we specialize our analysis to characterize the qualitative behavior predicted by the weakly nonlinear analysis for uniform-band and Watts--Strogatz graphons and for both reaction functions $f(n)$ corresponding to logistic growth or a model featuring both a weak Allee effect and logistic regulation. Finally, in Section \ref{sec:graphon-concave-mass-decay}, we use a comparison principle on the mass of the solution of our graphon model to show that, under the logistic regulation function, any long-time patterned state will feature lower overall population size than the population achieved by the spatially uniform equilibrium.

\subsection{Linear Stability Analysis for the Graphon Model}\label{sec:linear-stability-graphon}

We begin with the linear stability analysis of our graphon model, illustrating the conditions for the onset of pattern-forming instability for a special family of graphons for which we apply Fourier analysis to analyze our linearized system. Throughout this subsection, we restrict our attention to translation-invariant ring kernels of the form
\[
W(x,y) = W(y-x),
\]
where $W(\cdot) \in L^\infty([0,1])$ is nonnegative, even, $1$-periodic, and not identically zero. 

We assume there exists a population density $\bar{n} > 0$ satisfying the stable reaction conditions $f(\bar{n}) = 0$ and $f'(\bar{n}) < 0$. Letting $\bar{c} := \alpha\bar{n}/\beta$, the graphon model \eqref{eq:our-graphon-model} admits the spatially homogeneous coexistence equilibrium $(n(x), c(x)) \equiv (\bar{n}, \bar{c})$. To analyze the linear stability of this uniform state, we consider small perturbations of the form
\begin{subequations}
    \begin{align}
        n(t,x)&=\bar{n}+\varepsilon\tilde{n}(t,x),\\
        c(t,x)&=\bar{c}+\varepsilon\tilde{c}(t,x),
    \end{align}
\end{subequations}
for some small parameter $\varepsilon>0$, which helps us investigate the behavior of solutions that are close to the homogeneous coexistence equilibrium. Plugging this ansatz into Equation~\eqref{eq:our-graphon-model}, using the fact that $\bar{n}$ and $\bar{c}$ are equilibrium values of the reaction dynamics, and collecting all $O(\varepsilon)$ terms, we obtain the following system of nonlocal PDEs as the linearization of our system around the uniform coexistence equilibrium: 
\begin{equation} \label{eq:linearized-PDE}
    \frac{\partial}{\partial t}\begin{bmatrix}
        \tilde{n}(t,x)\\
        \tilde{c}(t,x)
    \end{bmatrix}=\mathcal{D}_{\chi}\begin{bmatrix}
        \tilde{n}(t,x)\\
        \tilde{c}(t,x)
    \end{bmatrix},
\end{equation}
where the nonlocal linear PDE operator $\mathcal{D}_\chi$ is given by
\begin{equation}
    \mathcal{D}_{\chi}:=\begin{bmatrix}
        \frac{\eta_n}{2}\mathcal{L}_W+f'(\bar{n}) & -\frac{\eta_n \chi \bar{n}}{2}\mathcal{L}_W\\
        \alpha & \frac{\eta_c}{2}\mathcal{L}_W-\beta
    \end{bmatrix}
\end{equation}
and the operator $\mc{L}_W$ describes the graphon Laplacian operator associated with the graphon $W(y-x)$ defined by
\begin{equation}
    \mathcal{L}_WF:=\int_{0}^{1}W(y-x)\big[F(y)-F(x)\big]\,\mathrm{d}y.
\end{equation}
We will use this linearized system of PDEs to determine the conditions required for spatial pattern formation due to the sufficiently strong chemotactic sensitivity $\chi$.

\begin{Remark}
In this paper, we typically consider the case of a reaction function $f(n)$ such that there is a single positive equilibrium value $\bar n$ satisfying $f(\bar n)=0$ and $f'(\bar n)<0$. This allows us to deduce that the linearization of the reaction dynamics will take the form 
\begin{equation}
  \begin{bmatrix}
        f'(\bar{n}) & 0\\
       \alpha & -\beta,
   \end{bmatrix}
\end{equation}
so this matrix will have two negative eigenvalues and the equilibrium $(n,c) = \left(\bar n, \frac{\alpha \bar n}{\beta}\right)$ will be stable under the reaction dynamics. We therefore look to explore how the inclusion of chemotactic motion in our graphon model can help to destabilize this spatially uniform equilibrium in the presence of small sinusoidal perturbations. 
\end{Remark}

Our spatial domain is the interval $I=[0,1]$, interpreted as the unit circle with its endpoints identified. Therefore the $1$-periodicity of the kernel $W$ and our periodic interpretation of the interval provides a natural periodic boundary conditions for our nonlocal PDE model. With periodic boundary conditions on $I = [0,1]$, each nonzero eigenvalue for our graphon Laplacian operator would be associated with a pair of eigenfunctions $\cos\left( 2 \pi k x\right)$ and $\sin\left(2\pi kx\right)$ for the relevant integer wavenumber $k$, so we will place a restriction on our initial condition so that each nonzero eigenvalue of $\mc{L}_W$ is associated only with the cosine eigenfunction. 

To eliminate the sine modes as eigenfunctions associated with each wavenumber $k$, we assume that our initial condition at time $t = 0$ and the resulting profiles $n(t,x)$ and $c(t,x)$ of the population and chemical attractant satisfy the following symmetry across the point $x = \frac{1}{2}$:
\begin{equation}
    n(t,x)=n(t,1-x),
    \qquad
    c(t,x)=c(t,1-x)
    \text{ for a.e. }x\in I.
    \label{eq:symmetric-condition}
\end{equation}
We can also interpret this symmetry in terms of our periodic domain $[0,1]$ as considering initial conditions and solutions that are even, 1-periodic functions, and therefore it is possible to expand solutions $n(t,x)$ and $c(t,x)$ in terms of a Fourier cosine series on $[0,1]$. 

We show in Appendix~\ref{sec:symmetry-invariant} that the choice of this reflection symmetry of our initial data will result in solutions $n(t,x)$ and $c(t,x)$ for the graphon Keller--Segel model for all positive time $t$. Furthermore, because this subset of initial conditions satisfies the assumptions of Proposition \ref{prop:graphon_bounds}, we know that these solutions will exist from the global well-posedness of the graphon model for $L^{\infty}([0,1])$ initial data. 

We will consider sinusoidal perturbations reflecting this symmetry of solutions, using the following ansatz for solutions to our linearized system from equation~\eqref{eq:linearized-PDE}:
\begin{equation}\label{eq:linear-perturbation} 
    \begin{aligned}
        \tilde{n}(t,x)&=q_n e^{\lambda t}\cos\left(2\pi k x\right),\\
        \tilde{c}(t,x)&=q_c e^{\lambda t}\cos\left(2\pi k x\right),
    \end{aligned}
\end{equation}
where $k$ is a positive integer wavenumber and $\lambda$ is the growth rate of of the spatial mode with wavenumber $k$. The case of $k=0$ corresponds to a spatially uniform perturbation, which will result in a growth rate $\lambda$ with negative real part due to the stability of the equilibrium $(\bar{n},\bar{c})$ under the reaction dynamics. 

Before plugging Equation~\eqref{eq:linear-perturbation} into Equation~\eqref{eq:linearized-PDE}, we provide the following proposition characterizing the properties of the graphon Laplacian operator $\mc{L}_W$.  

\begin{proposition}\label{prop:L_A-cos}
Assume that $W\in L^\infty([0,1])$ is nonnegative, even, $1$-periodic, and not identically zero, and define
\begin{equation}\label{eq:A-hat}
    \widehat{W}(k):=\int_0^1 W(z)\cos(2\pi k z)\,\mathrm{d}z,
    \qquad k\in\mathbb Z.
\end{equation}
Then, for the graphon Laplacian
\[
    \mathcal L_WF(x)=\int_0^1 W(x,y)\big(F(y)-F(x)\big)\,\mathrm{d}y,
\]
we have
\[
    \mathcal L_W\cos(2\pi kx)
    =
    -\big[\widehat{W}(0)-\widehat{W}(k)\big]\cos(2\pi kx).
\]
Furthermore, we observe that $\widehat{W}(0)-\widehat{W}(k)>0$ for all $k\ne0$, so all nonzero eigenvalues of the graphon Laplacian will be negative. Finally, $\widehat{W}(k)\to0$ as $|k|\to\infty$ by the Riemann--Lebesgue lemma.
\end{proposition}

Using Proposition~\ref{prop:L_A-cos}, we can plug the ansatz from Equation~\eqref{eq:linear-perturbation} into Equation~\eqref{eq:linearized-PDE} to see that our expression $(\tilde{n}(t,x), \tilde{c}(t,x))$ will solve our linearized system provided that the temporal growth rate $\lambda$ is a solution to the following eigenvalue problem:
\begin{equation}\label{eq:linear-perturbation-eigen}
    \lambda\begin{bmatrix}
        q_n\\
        q_c
    \end{bmatrix}=M_\chi(k)\begin{bmatrix}
        q_n\\
        q_c
    \end{bmatrix},
\end{equation}
where 
\begin{equation}\label{eq:L(k)}
    M_\chi(k):=\begin{bmatrix}
        -\frac{\eta_n}{2}\big(\widehat{W}(0)-\widehat{W}(k)\big)+f'(\bar{n}) & \frac{\eta_n\chi \bar{n}}{2}\big(\widehat{W}(0)-\widehat{W}(k)\big)\\
        \alpha & -\frac{\eta_c}{2}\big(\widehat{W}(0)-\widehat{W}(k)\big)-\beta
    \end{bmatrix}.
\end{equation}
When $M_\chi(k)$ admits at least one eigenvalue with positive real part, then the linear perturbation \eqref{eq:linear-perturbation} becomes unstable due exponential growth represented by the term proportional to $e^{\lambda t}$. 
As we have considered a spatially uniform equilibrium $(\bar{n},\bar{c})$ and reaction dynamics such that $f'(\bar{n}) < 0$, we see that the trace of our linearization matrix $M_{\chi}(k)$ will satisfying 
\begin{equation}
\operatorname{tr}(M_\chi(k))=-\beta+f'(\bar{n})-\frac{1}{2}\left(\eta_n+\eta_c\right)\Big[\widehat{W}(0)-\widehat{W}(k)\Big]<0
\end{equation}
for each wavenumber $k$. Therefore we see that a perturbation proportional to $\cos\left( 2 \pi k x\right)$ can only produce an eigenvalue of $M_\chi(k)$  with positive real part provided that $\det(M_\chi(k))<0$. We can rewrite this condition to see that instability will be possible provided that the chemotactic sensitivity $\chi$ satisfies
\begin{equation}\label{eq:threshold-w}
    \chi>\chi^*(k)=\frac{\frac{1}{2}\eta_n\eta_c\big[\widehat{W}(0)-\widehat{W}(k)\big]^2+\big[\beta\eta_n-\eta_cf'(\bar{n})\big]\big[\widehat{W}(0)-\widehat{W}(k)\big]-2\beta f'(\bar{n})}{\alpha\eta_n \bar{n}\big[\widehat{W}(0)-\widehat{W}(k)\big]}.
\end{equation}
Because we have assumed that our reaction function $f(n)$ satisfies $f'(\bar{n}) < 0$ and we have shown in Proposition \ref{prop:L_A-cos} that the eigenvalues of the graphon Laplacian satisfy $\widehat{W}(0)-\widehat{W}(k) > 0$, we can deduce that the threshold $\chi^*(k)$ in \eqref{eq:threshold-w} is positive for all nonzero integer wavenumbers $k$. 

In particular, we have shown that, give a perturbation proportional to $\cos\left( 2 \pi k x\right)$, the linearization matrix will satisfy 
$\operatorname{tr}(M_\chi(k))<0$ and $\det(M_\chi(k))<0$ when $\chi>\chi^*(k)$. Therefore $M_\chi(k)$ will have real eigenvalues of opposite signs for sufficiently large $\chi$, suggesting that we will expect the growth of perturbations with wavenumber $k$ without a temporally oscillatory behavior sufficient close to the onset of instability. Because this threshold $\chi^*(k)$ depends on the wavenumber under consideration, we can also look to characterize a threshold for pattern formation that holds over perturbations for all possible integer wavenumbers, allowing us to study the strength of chemotactic sensitivity $\chi$ required to produce spatial patterns for at least one wavenumber.

To do this, we define the overall instability threshold by
\begin{equation}\label{eq:threshold-wmin}
    \chi^* := \inf_{k\in\mathbb Z\setminus\{0\}}\chi^*(k).
\end{equation}
The homogeneous equilibrium is linearly stable against all spatial perturbations when $\chi<\chi^*$, and unstable whenever $\chi>\chi^*$. If this infimum is achieved for some finite positive wavenumber, we denote this critical wavenumber by $k^*$ and write $\chi^*=\chi^*(k^*)$.

However, it may be possible that the instability threshold $\chi^*(k)$ may be first achieved simultaneously by infinitely many wavenumbers simultaneously. In particular, there exists a second threshold $\chi^{**}$ on chemotactic instability that guarantees the instability of the uniform state for all sufficiently large wavenumbers $k$. When $\chi$ exceeds this threshold $\chi^{**}$, we see that are infinitely many wavenumbers $k \in \mathbb{Z}$ for which the perturbations proportional to $\cos\left(2 \pi k x\right)$ will become linearly unstable under the graphon model of equation~\eqref{eq:our-graphon-model}. We characterize this second threshold in the following proposition. 

\begin{proposition}
Assume the hypotheses of Proposition~\ref{prop:L_A-cos}. 
 Then the Fourier coefficients $\widehat{W}(k)$ defined in \eqref{eq:A-hat} satisfy $\widehat{W}(k)\to 0$ as $k\to \infty$. In addition, we see that the limit
\begin{equation}\label{eq:double-threshold-w}
    \chi^{**}:=\lim_{k\to\infty}\chi^*(k)
    =
    \frac{
    \frac12\eta_n\eta_c\widehat{W}(0)^2
    +\big[\beta\eta_n-\eta_c f'(\bar n)\big]\widehat{W}(0)
    -2\beta f'(\bar n)
    }{
    \alpha\eta_n\bar n\,\widehat{W}(0)
    }.
\end{equation}
exists. In particular, if $\chi>\chi^{**}$, then there exist infinitely many wavenumbers $k$ such that $\chi>\chi^*(k)$.
\end{proposition}

Under that assumption that $f'(\bar{n})<0$, the threshold quantity $\chi^{**}$ from Equation~\eqref{eq:double-threshold-w} is positive, so there exists sufficiently strong finite values of the chemotactic sensitivity $\chi$ such that infinitely many wavenumbers will become unstable under the graphon Keller--Segel model. This stands in contrast to the stability analysis for the PDE Keller--Segel model with equivalent reaction terms, for which the eigenvalues of the one-dimensional Laplacian operator $\partial_{xx}$ take the form $-4 \pi^2 k^2$ and prevent the possibility of destabilizing sufficiently large wavenumbers $k$. 

Since $\chi^*(k) \to \chi^{**}$ as $k \to \infty$, it follows that $\chi^* \leq \chi^{**}$. If $\chi^* < \chi^{**}$, the infimum $\chi^*$ is attained by a finite set of positive wavenumbers. A unique minimizer $k^*$ provides an isolated threshold required for the single-mode analysis below, whereas multiple minimizers necessitate a coupled amplitude system. If $\chi^* = \chi^{**}$, the infimum may not be attained at any finite wavenumber, and, at $\chi=\chi^*$, the maximal real part of the modal eigenvalues approaches zero as $k\to\infty$. Consequently, any $\chi > \chi^{**}$ destabilizes all sufficiently large wavenumbers. The subsequent weakly nonlinear analysis assumes $\chi^* < \chi^{**}$ with a unique minimizer $k^*$. 

\begin{Remark}
        Although there exist infinitely many integer wavenumbers $k$ that destabilize $M_\chi(k)$, the boundedness of $\widehat{W}(k)$ guarantees that the eigenvalues of $M_\chi(k)$ are a bounded function of the wavenumber $k$. In particular, we observe from the dispersion relation plotted in Figure \ref{fig:dispersion-relation} that a finite wavenumber $k^*$ has the maximal real part among all non-negative integer wavenumbers $k$ for the choice of parameters considered in that figure, so we expect the emergence of a spatial pattern with spatial profiles with a finite number of peaks across the spatial domain. This emergence of physically plausible spatial patterns and the convergence of numerical solutions to long-time equilibrium states suggests that the infinite family of unstable wavenumbers does not raise any issue with the qualitative behavior of the pattern-forming instability even for $\chi > \chi^{**}$.

        In contrast, the presence of infinitely many unstable wavenumbers has often served as an indication of the possibility of ill-posedness through finite-time blowup or a lack of continuous dependence on initial data in PDE models \citep{joseph1990short,bernoff2016biological,garcia2020illposedness,ellefsen2021equilibrium,yao2026pattern}. We do not see a similar issue with ill-posedness in our graphon Keller--Segel model, as we have been able to demonstrate global well-posedness for $L^\infty$ solutions to our model in Proposition \ref{prop:graphon_bounds} for non-negative initial densities of the population and chemical attractant. Notably, the form of eigenvalues taken for the graphon Laplacian operator as a bounded function of $k$ differs from the form of $-4 \pi^2 k^2$ taken for the local Laplacian $\partial_{xx}$ on our spatial domain, so our formulation of directed motion on graphons may be able to produce well-posed graphon models featuring negative diffusive effects that would produce physically unrealistic solutions for corresponding PDE models of directed motion. 
\end{Remark}

We can now provide further illustration of these conditions for the onset of chemotactic instability in our graphon model by considering the uniform band kernel on the unit circle:
\begin{equation}\label{eq:band-kernel}
    W(x,y)=\frac{1}{2\rho}\ind_{\{d(x,y)<\rho\}},
    \qquad
    d(x,y)=\min\{|x-y|,1-|x-y|\},
    \quad 0<\rho<\frac12.
\end{equation}
Writing $W(x,y)=W(y-x)$ with $W$ even and $1$-periodic, we obtain
\begin{equation} \label{eq:Wvaluesring}
    \widehat W(0)=1,
    \quad
    \widehat W(k)=\frac{\sin(2\pi k\rho)}{2\pi k\rho}
    =\operatorname{sinc}(2\pi k\rho)
    \quad(k\ne0),
\end{equation}
where $\operatorname{sinc}(z)=\sin(z)/z$ for $z\ne0$ and
$\operatorname{sinc}(0)=1$.
This normalized band kernel is a weighted interaction kernel and does not satisfy the edge-probability assumption $0\le W\le1$ used in Section~\ref{ssec:random-limit}. The normalization for finite-graph comparisons is discussed in Section~\ref{sec:simulation-graph}.
We can use these explicit expressions for the Fourier coefficients $ \widehat{W}(k)$ to provide the following analytical expressions for the instability thresholds $\chi^*(k)$ and $\chi^{**}$ for the graphon Keller--Segel model: 
\begin{subequations} \label{eq:chistaruniform}
    \begin{align}
        \chi^{*}&=\inf_{k\in\mathbb{Z}\backslash\{0\}}\left\{\frac{\frac{1}{2}\eta_n\eta_c\big[1-\operatorname{sinc}(2\pi k\rho)\big]^2+\big[\beta\eta_n-\eta_cf'(\bar{n})\big]\big[1-\operatorname{sinc}(2\pi k\rho)\big]-2\beta f'(\bar{n})}{\alpha\eta_n \bar{n}\big[1-\operatorname{sinc}(2\pi k\rho)\big]}\right\}\\
        \chi^{**}&=\frac{\frac{1}{2}\eta_n\eta_c+\big[\beta\eta_n-\eta_cf'(\bar{n})\big]-2\beta f'(\bar{n})}{\alpha\eta_n \bar{n}}.
    \end{align}
\end{subequations}
For the uniform band kernel \eqref{eq:band-kernel}, we numerically compute the minimal threshold $\chi^*$ and the second threshold $\chi^{**}$ as functions of the band width $\rho \in (0,1/2)$.  We note that the threshold $\chi^*$ will be piecewise smooth in $\rho$, with possible jumps in the derivative of this threshold occurring when the critical wavenumber $k^*$ changes as the radius $\rho$ increases. We note that $\chi^* < \chi^{**}$ for all $\rho \in (0,1/2)$, so there will always be a finite number of critical wavenumbers except for the complete graphon that arises in the case of $\rho = \frac{1}{2}$. 

\begin{Remark}
Notably, we find that, for $\rho = \frac{1}{2}$, $\operatorname{sinc}(2 \pi k \rho) = \operatorname{sinc}(2 \pi k) = 0$ for each nonzero integer wavenumber $k$, and therefore we can see from Equation \eqref{eq:chistaruniform} that $\chi^*(k) = \chi^{**}$ for each nonzero wavenumber $k$, and therefore we see that all wavenumber destabilize simultaneously as $\chi$ increases past $\chi^{**}$. This means that, for the case of the complete graphon, there is no characteristic spatial scale that emerges for patterned states in our graphon Keller--Segel model, and it will not be possible for us to analyze our instability through the weakly nonlinear analysis introduced in Section \ref{sec:graphon-weakly-general}.
\end{Remark}

We can also visualize properties of our condition pattern-forming instability for the case of the uniform band graphon. In Figure \ref{fig:thresholds-graph}, we plot the thresholds $\chi^*$ and $\chi^{**}$ presented in Equation \eqref{eq:chistaruniform} as a function of the band radius $\rho$, highlighting the piecewise smooth nature of the threshold $\chi^*$ and indicating the fact that $\chi^*$ increases to $\chi^{**}$ as $\rho \to \frac{1}{2}$. We also present the emergence of unstable wavenumbers for increasing chemotactic sensitivity $\chi$, plotting the dispersion relation for the case of the uniform band graphon in Figure \ref{fig:dispersion-relation}. The dispersion relation presents the largest real part for the eigenvalues of the linearization matrix $M_{\chi}(k)$ as a function of the wavenumber $k$, and we plot threshold examples of this dispersion relation for cases in which all wavenumbers are unstable (with $\chi < \chi^{**}$), when a single wavenumber crosses the instability threshold (with $\chi = \chi^*$), when finitely many wavenumbers have become unstable (for $\chi$ satisfying $\chi^* < \chi < \chi^{**}$), and once infinitely many wavenumbers have become unstable (for $\chi > \chi^{**}$). 

\begin{figure}[!ht]
    \centering
    \includegraphics[width=0.5\linewidth]{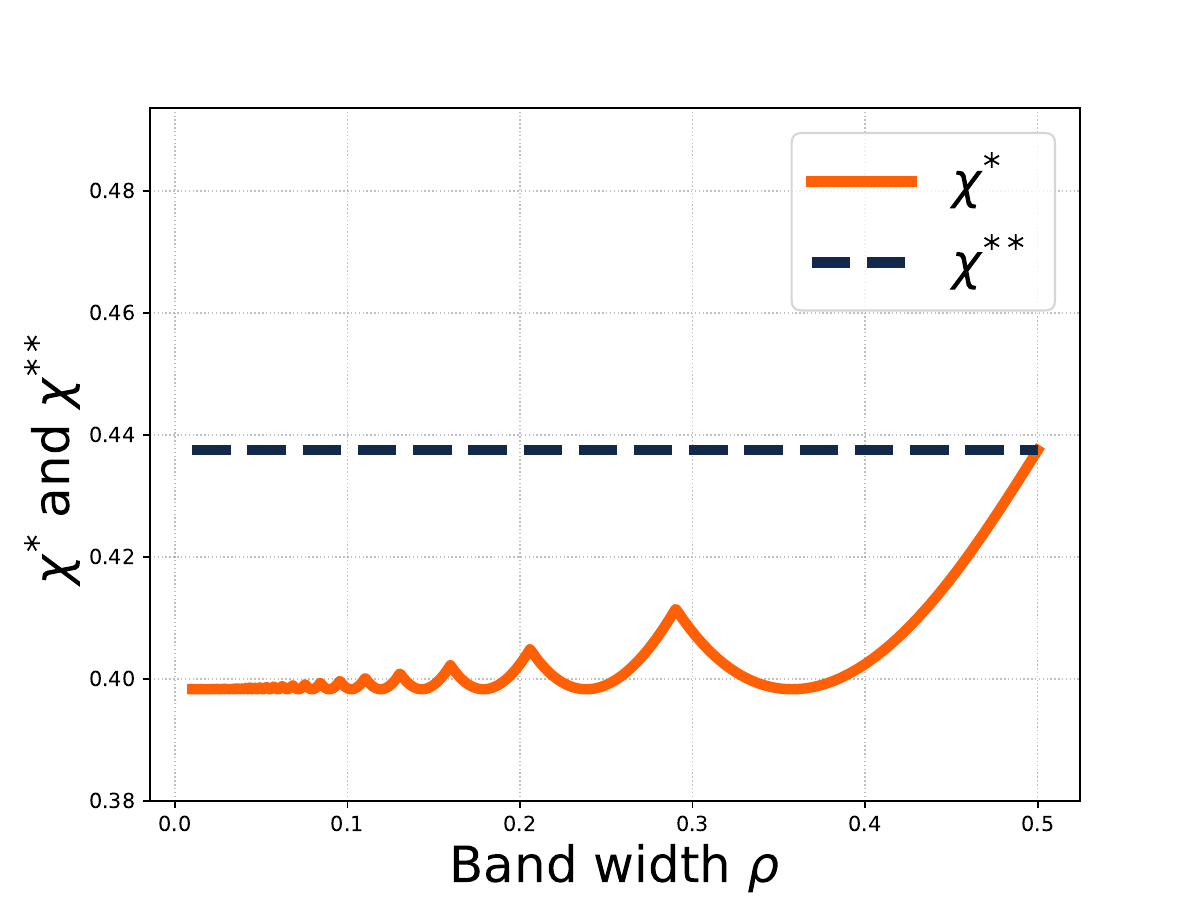}
    \caption{Thresholds $\chi^*$ and $\chi^{**}$ as functions of the band width $\rho$ for the uniform band kernel. The orange solid line represents $\chi^*$, while the blue dashed line represents $\chi^{**}$. The PDE parameters are $\eta_n=0.05$, $\eta_c=0.08$, $\alpha=0.2$, and $\beta=0.1$. The reaction term is given by $f(n)=rn(1-n/\bar{n})$ with $r=0.1$, $\bar{n}=8$, and $\bar{c}=16$.}
    \label{fig:thresholds-graph}
\end{figure}

\begin{figure}[!ht]
    \centering

    \begin{subfigure}[t]{0.46\linewidth}
        \centering
        \includegraphics[width=\linewidth]{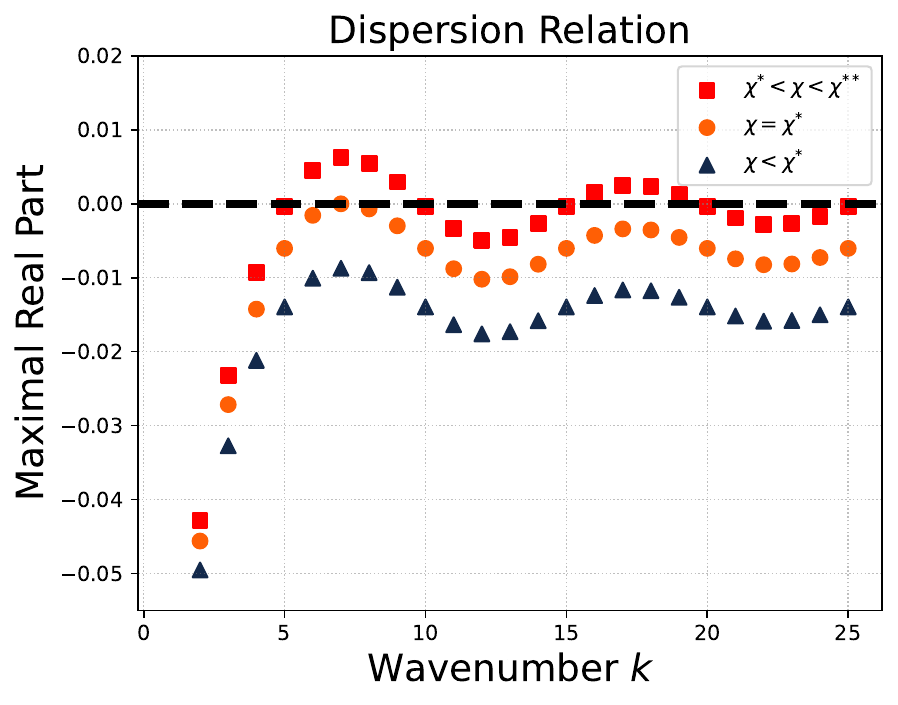}
    \end{subfigure}
    \hspace{0.05\linewidth}
    \begin{subfigure}[t]{0.46\linewidth}
        \centering
        \includegraphics[width=\linewidth]{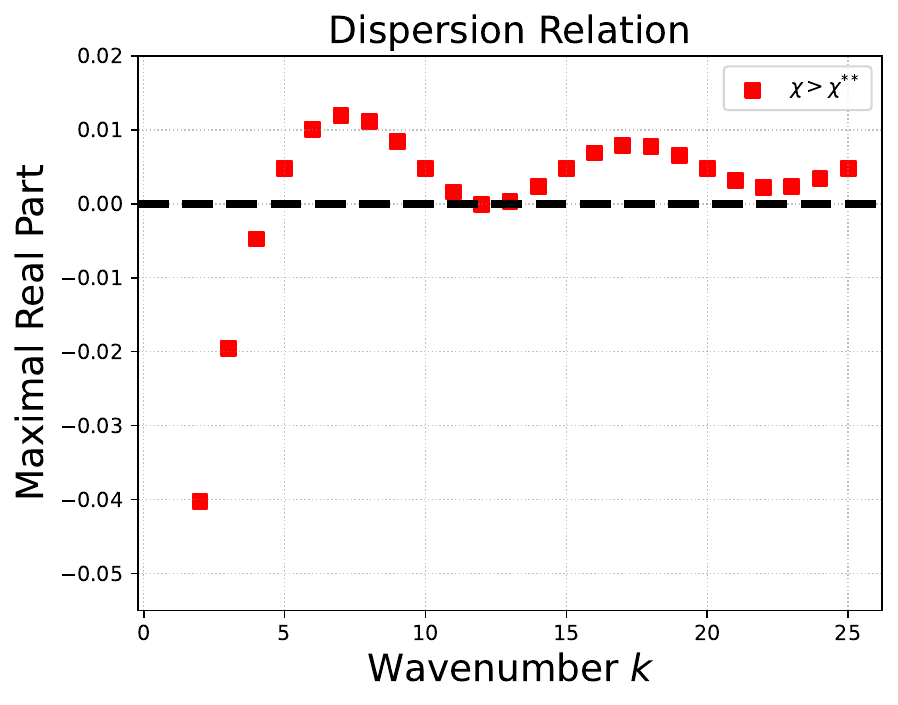}
    \end{subfigure}

    \caption{Dispersion relation for the uniform band kernel with width $\rho=0.1$. The left panel shows the case $\chi < \chi^{**}$, while the right panel shows the case $\chi > \chi^{**}$. We take $\chi = 0.35$, $0.3985$, $0.43$, and $0.47$. The PDE parameters are $\eta_n=0.05$, $\eta_c=0.08$, $\alpha=0.2$, and $\beta=0.1$. The reaction term is given by $f(n)=rn(1-n/\bar{n})$ with $r=0.1$, $\bar{n}=8$, and $\bar{c}=16$. The critical perturbation wavenumber $k^*=7$, the threshold is $\chi^*=\chi^*(7)\approx 0.3984$, while the second threshold is $\chi^{**}\approx 0.4375$.}
    \label{fig:dispersion-relation}
\end{figure}

%%%%%%%%%%%%%%%%%%%%%%%%%%%%%%%%%%%%%%%%%%%%%%%%%%%%%%%%%%
%%%%%%%%%%%%%%%%%%%%%%%%%%%%%%%%%%%%%%%%%%%%%%%%%%%%%%%%%%
\subsection{Weakly Nonlinear Analysis for the Graphon Model}
\label{sec:graphon-weakly-general}

In Section~\ref{sec:linear-stability-graphon}, we have studied the threshold of the bifurcation parameters $\chi$ crossing which our graphon model of Equation~\eqref{eq:our-graphon-model} becomes linearly unstable. To further understand the nature of the bifurcation and the ensuing pattern formation, we will now perform a formal weakly nonlinear analysis to describe the behavior of the emerging pattern in our graphon Keller--Segel model near the onset of instability. 

For this analysis, we will retain our assumption of a translation-invariant kernel $W(x,y)$ and consider solutions starting from an initial condition satisfying the reflection symmetry condition from Equation~\eqref{eq:symmetric-condition}. To perform our weakly nonlinear expansion of the graphon up to deriving an amplitude equation at cubic order, we will assume that our reaction function $f(n) \in C^4(\mathbb{R})$ near the equilibrium point $\bar{n}$. For this analysis, we will assume that there is a unique critical wavenumber $k^*$ at the threshold for instability $\chi^*$, so that there is a unique positive integer $k^*$
\begin{equation}
\begin{aligned}
    \chi^* = \chi^*(k^*) < \chi^{**}, 
    \text{ and }
    \chi^*(k) > \chi^* \text{ for all integers } k \ge 1 \text{ with } k \ne k^*.
\end{aligned}
\end{equation}
Combining our assumptions on the unique critical wavenumber $k^*$ and our reflection-symmetric solution, there will be a single unstable eigenfunction for our linearized system that is proportion $\cos(2\pi k^*x)$, which will guarantee that our linearization matrix satisfies $M_{\chi^*}(0)$ and $M_{\chi^*}(2k^*)$ are invertible. 

To study the behavior of our model near the onset of instability, we introduce the following small parameter $\varepsilon$ satisfying $0<\varepsilon \ll 1$ and we assume that the chemotactic sensitivity $\chi$ can be written as 
\begin{subequations}\label{eq:weakly-transformations}
\begin{equation}
\chi=\chi^*(k^*)+\varepsilon^2\tilde{\chi}
\end{equation} 
and we consider the following slow time variable $\tau$
\begin{equation}\quad \tau=\varepsilon^2 t
\end{equation}
\end{subequations}
which will allow us to describe the slow growth of the unstable spatial pattern when $\chi$ is sufficiently close to the critical threshold value $\chi^*$ for the onset of instability. Here, $\chi^*(k^*)$ is the threshold chemotactic sensitivity for pattern formation that can be obtained by Equation~\eqref{eq:threshold-wmin} and $\tilde{\chi}$ will be used a control parameter to describe the behavior of crossing through the threshold $\chi^*(k^*)$ as $\tilde{\chi}$ increases from negative values to positive values (allowing our weakly nonlinear stability analysis and the resulting amplitude equation to characterize the behavior of the bifurcation as $\chi$ passes through $\chi^*$). 

\begin{Remark}
Our choice of the scaling of the deviation from the threshold chemotactic sensitivity $\chi - \chi^*$ and the slow timescale $\tau$ proportion to $\varepsilon^2$ was chosen to reflect the periodic boundary conditions that we consider for our graphon model. It is possible to follow the approach highlighted by \cite{woolley2022boundary} to deduce the appropriate scaling of parameters required for our problem, although we have started directly with the scaling that is typically used to analysis instabilities for nonlocal pattern-formation problems with periodic boundary conditions \citep{topaz2006nonlocal,eftimie2009weakly,buttenschon2021non}.
\end{Remark}

To perform our weakly nonlinear analysis, we consider the following perturbation expansions for the profiles of population density and the concentration of the chemoattractant
\begin{equation}\label{eq:wnl-expansion}
    \begin{aligned}
        n(\tau,x)&=\bar{n}+\varepsilon n_1(\tau,x) + \varepsilon^2 n_2(\tau,x)+\varepsilon^3n_3(\tau,x)+\cdots\\
        c(\tau,x)&=\bar{c}+\varepsilon c_1(\tau,x) + \varepsilon^2 c_2(\tau,x)+\varepsilon^3 c_3(\tau,x)+\cdots,
    \end{aligned}
\end{equation}
where $(\bar{n},\bar{c})$ corresponds to the uniform equilibrium state and we consider solutions that depend only on the slow timescale $\tau = \varepsilon^2 t$ due to our assumption that $\chi$ is close to the threshold $\chi^*$ for the onset of pattern formation. By plugging this perturbation ansatz into Equation~\eqref{eq:our-graphon-model}, we show in Section \ref{sec:derivation_stability_analysis} of the appendix that the functions in our expansion will satisfy a hierarchy of integro-differential equations for each order of $\varepsilon$, with the problem at each order inheriting the periodic boundary conditions from the original graphon Keller--Segel model. Up to cubic order in $\varepsilon$, this hierarchy of equations will be given by
\begin{subequations}
\begin{alignat}{2}
    O(1) &: \:\:\: \hspace{10.6mm} \begin{bmatrix}
        0\\
        0
    \end{bmatrix}=&& \begin{bmatrix}
    f(\bar{n})\\
    \alpha \bar{n}-\beta \bar{c}
\end{bmatrix}, \label{eq:DM-graphon-order1} \\
O(\varepsilon) &: \:  -\mathcal{D}_{\chi^*}\begin{bmatrix}
        n_1\\
        c_1
    \end{bmatrix} =&& \begin{bmatrix}
        0\\
        0
    \end{bmatrix}, \label{eq:DM-graphon-orderep} \\
O(\varepsilon^2) &: \:     -\mathcal{D}_{\chi^*}\begin{bmatrix}
        n_2\\
        c_2
    \end{bmatrix} =&& \begin{bmatrix}
            \mathcal{H}_n^{(2)}[n_1,c_1]\\
            0
        \end{bmatrix}, \label{eq:DM-graphon-orderep2} \\
O(\varepsilon^3) &: \:   -\mathcal{D}_{\chi^*}\begin{bmatrix}
            n_3\\
            c_3
        \end{bmatrix} = && \,\,\begin{aligned}[t]
            &\,\,-\frac{\partial}{\partial \tau}\begin{bmatrix}
            n_1\\
            c_1
        \end{bmatrix}-\tilde{\chi}\begin{bmatrix}
            \frac{\eta_n}{2}\bar{n}\int_{0}^{1}W(y-x)[c_1(\tau,y)-c_1(\tau,x)]\,\mathrm{d}y\\
            0
        \end{bmatrix}\\
        &+\begin{bmatrix}
            \mathcal{H}_{n}^{(3)}[n_1,c_1,n_2,c_2]\\
            0
        \end{bmatrix},
        \end{aligned}
        \label{eq:DM-graphon-orderep3}
\end{alignat}
\end{subequations}
where the operator $\mathcal{D}_{\chi^*}$ appearing on the lefthand side corresponds to the linearization of the graphon model at the critical sensitivity of the chemoattractant to the individuals $\chi^*$ given by
\begin{equation}\label{eq:linearization-graphon}
    \mathcal{D}_{\chi^*}:=\begin{bmatrix}
        \frac{\eta_n}{2}\mathcal{L}_W+f'(\bar{n}) & -\frac{\eta_n \chi^* \bar{n}}{2}\mathcal{L}_W\\
        \alpha & \frac{\eta_c}{2}\mathcal{L}_W-\beta
    \end{bmatrix}
\end{equation}
and the nonlinear terms on the righthand sides of the problems at orders $\mc{O}(\varepsilon^2)$ and $\mc{O}(\varepsilon^3)$ are given by
\begin{equation}
    \begin{aligned}
        \mathcal{H}_n^{(2)}=&\,\,-\frac{\eta_n}{4}\chi^*\int_{0}^{1}W(y-x)\sigma_x^yn_1\delta_x^yc_1\,\mathrm{d}y+\frac{1}{2}f''(\bar{n})n_1^2,\\
        \mathcal{H}_n^{(3)}=&\,\,-\frac{\eta_n}{4}\chi^*\int_{0}^{1}W(y-x)\sigma_x^yn_2\delta_x^yc_1\,\mathrm{d}y-\frac{\eta_n}{4}\chi^*\int_{0}^{1}W(y-x)\sigma_x^yn_1\delta_x^yc_2\,\mathrm{d}y\\
        &\,\,+\frac{\eta_n}{24}(\chi^*)^3\bar{n}\int_{0}^{1}W(y-x)\left(\delta_x^yc_1\right)^3\,\mathrm{d}y+f''(\bar{n})n_1n_2+\frac{1}{6}f'''(\bar{n})n_1^3
    \end{aligned}
\end{equation}
with $\sigma_x^yn_i=n_i(\tau,y)+ n_i(\tau,x)$ and $\delta_x^yn_i=n_i(\tau,y)-n_i(\tau,x)$. Notably, the problem at order $\mc{O}(1)$ recovers the steady-state problem for the spatially homogeneous equilibrium state, the problem at order $\mc{O}(\varepsilon)$ describes the behavior of the linearized graphon system at the critical threshold $\chi = \chi^*$, and the problems at orders $\mc{O}(\varepsilon^2)$ and $\mc{O}(\varepsilon^3)$ describe the effects on the nonlinear movement and reaction terms on the emergent spatial patterns. The nonlinear terms on the righthand sides of the problems at order $\mc{O}(\varepsilon^2)$ and $\mc{O}(\varepsilon^3)$ only depend on solutions to the problem at lower orders, so we can look to successively solve the problem at each order by plugging in solutions from lower orders on the righthand side and solving the resulting inhomogeneous linear equations at the given order using the method of undetermined coefficients.

For the $\mc{O}(\varepsilon)$-problem, we are looking for a solution in the kernel of our linearized operator. This suggests that we consider a solution that is proportion to the eigenfunction associated with our critical eigenfunction $\cos(2\pi k^* x)$, so we take solution of the form
\begin{equation}\label{eq:first-order-eq}
    \begin{bmatrix}
        n_1(\tau,x)\\
        c_1(\tau,x)
    \end{bmatrix}
    =a(\tau)\mathbf{q}\cos(2\pi k^* x),
\end{equation}
where $a(\tau)$ is the amplitude of the sinusoidal profile of the $O(\varepsilon)$-solution and the vector $\mathbf{q} =(q_n,q_c)^\mathrm{t}$ with $q_n=1$ is an element of the kernel of the linearization matrix $M_{\chi^*}(k)$ when applied to functions proportional to sinusoidal functions of the form $\cos\left(2\pi k^* x\right)$, which is given by
\begin{equation*}
    M_{\chi^*}(k^*)=\begin{bmatrix}
        \frac{\eta_n}{2}\big(\widehat{W}(k^*)-\widehat{W}(0)\big)+f'(\bar{n}) & -\frac{\eta_n\chi^*\bar{n}}{2}\big(\widehat{W}(k^*)-\widehat{W}(0)\big)\\
        \alpha & \frac{\eta_c}{2}\big(\widehat{W}(k^*)-\widehat{W}(0)\big)-\beta
    \end{bmatrix}.
\end{equation*}
The equation $M_{\chi^*}(k^*)\mathbf{q}=\mathbf{0}$, together with the second row of the matrix $M_{\chi^*}(k^*)$, can be used to determine the expression for $q_c$ as 
\begin{equation}\label{eq:qc}
    q_c=\frac{2\alpha}{2\beta+\eta_c\big[\widehat{W}(0)-\widehat{W}(k^*)\big]}>0.
\end{equation}
We note that all components of the vector $\mathbf{q}=(q_n,q_c)$ have the same sign, which implies that the deviation of the spatial profiles of the populations and the chemical from the uniform coexistence state will be positively correlated.

Next, we turn to our problem at order $\mc{O}(\varepsilon^2)$. We can then plug our expression from Equation~\eqref{eq:first-order-eq} for $n_1(\tau,x)$ and $c_1(\tau,x)$ into the right-hand side of Equation~\eqref{eq:DM-graphon-orderep2} to see that our solutions at order $O(\varepsilon^2)$ satisfy an inhomogeneous system of ODEs for $n_2(\tau,x)$ and $c_2(\tau,x)$. This solution together with the identity $\cos^2(2\pi k^* x)=\frac{1}{2}[1+\cos(4\pi k^* x)]$ allow us to rewrite Equation~\eqref{eq:DM-graphon-orderep2} as
\begin{equation}
    -\mathcal{D}_{\chi^*}\begin{bmatrix}
            n_2\\
            c_2
        \end{bmatrix}=-\mathbf{s}_0a^2(\tau)-\mathbf{s}_2a^2(\tau)\cos(4k^*\pi x),
\end{equation}
where the coefficients $\mathbf{s}_0$ and $\mathbf{s}_2$ are given by
\begin{equation} \label{eq:second-order-RHS}
    \begin{aligned}
        \mathbf{s}_{0}=\begin{bmatrix}
        s_{0,n}\\
        s_{0,c}
    \end{bmatrix}&:=\begin{bmatrix}
        -\frac{1}{4}f''(\bar{n})q_n^2\\
        0
    \end{bmatrix},\\
    \mathbf{s}_{2}=\begin{bmatrix}
        s_{2,n}\\
        s_{2,c}
    \end{bmatrix}&:=\begin{bmatrix}
        -\frac{\eta_n \chi^*}{8}q_nq_c\big[\widehat{W}(0)-\widehat{W}(2k^*)\big]-\frac{1}{4}f''(\bar{n})q_n^2\\
        0
    \end{bmatrix}.
    \end{aligned}
\end{equation}
The complementary solution to this equation will be proportional to the solution at order $\mc{O}(\varepsilon)$, so the only new spatial behavior exhibited by the solution at this order will be described through a particular solution to the $\mc{O}(\varepsilon^2)$-problem. We guess a particular solution of $n_2$ and $c_2$ in the form
\begin{equation}\label{eq:second-order-eq}
    \begin{bmatrix}
        n_2(\tau,x)\\
        c_2(\tau,x)
    \end{bmatrix}=\mathbf{t}_0a^2(\tau)+\mathbf{t}_2a^2(\tau)\cos(4k^*\pi x),
\end{equation}
and we look to determine the corresponding coefficients for both the constant and sinusoidal term in this solution. Because we are assuming that the spatially uniform equilibrium is stable under the reaction dynamics and that $k^*$ is the only unstable wavenumber for our choice of diffusivity $\chi$ close to $\chi^*$ in \eqref{eq:threshold-wmin}, we are able to deduce that the matrices $M_{\chi^*}(0)$ and $M_{\chi^*}(2k^*)$ are invertible, and we can therefore see that a particular solution to the $O(\varepsilon^2)$ is given by the functions
\begin{equation}\label{eq:t-expressions}
    \begin{aligned}
        \mathbf{t}_0&=\begin{bmatrix}
            t_{0,n}\\
            t_{0,c}
        \end{bmatrix}:=\begin{bmatrix}
            f'(\bar{n}) & 0\\
            \alpha & -\beta
        \end{bmatrix}^{-1}\mathbf{s}_0,\\
        \mathbf{t}_2&=\begin{bmatrix}
            t_{2,n}\\
            t_{2,c}
        \end{bmatrix}:=\begin{bmatrix}
            \frac{\eta_n}{2}\big[\widehat{W}(2k^*)-\widehat{W}(0)\big]+f'(\bar{n}) & -\frac{\eta_n\chi^*\bar{n}}{2}\big[\widehat{W}(2k^*)-\widehat{W}(0)\big]\vspace{1.5pt}\\
            \alpha & \frac{\eta_c}{2}\big[\widehat{W}(2k^*)-\widehat{W}(0)\big]-\beta
        \end{bmatrix}^{-1}\mathbf{s}_2.
    \end{aligned}
\end{equation}
We can then use the solutions we have obtained for $(n_1(\tau,x), c_1(\tau,x))^{\mathrm{t}}$ and $(n_2(\tau,x),c_2(\tau,x))^{\mathrm{t}}$ to rewrite the right-hand side of our $O(\varepsilon^3)$ equation. In particular, we plug these expressions into Equation~\eqref{eq:DM-graphon-orderep3} and use the identity $\cos\left(2\pi k^* x \right)^3 = \frac{3}{4} \cos\left(2\pi k^*
 x\right) + \frac{1}{4} \cos\left(6\pi k^* x\right)$ to see that it is possible to write this equation in the form
\begin{equation} \label{eq:DM-graphon-ep3righthandexplicit}
- \mathcal{D}_{\chi^*} \begin{bmatrix}
    n_3\\
    c_3
\end{bmatrix} = \mathbf{H}_1 \cos\left(2\pi k^* x\right)  + \mathbf{H}_3 \cos\left(6\pi k^* x\right),
\end{equation}
where the entries vector $\mathbf{H}_1=(H_{1,n}, H_{1, c})^{\mathrm{t}}$ of coefficients for the term proportional to $\cos\left(2\pi k^* x\right)$ are given by
\begin{equation}\label{eq:H1-expressions}
    \begin{aligned}
        H_{1,n}=&\,\,-q_na'(\tau)+\frac{\eta_n}{2}\tilde{\chi}\bar{n}q_c\big[\widehat{W}(0)-\widehat{W}(k^*)\big]a(\tau)+\frac{\eta_n}{2}\chi^* q_ct_{0,n}\left[\widehat{W}(0)-\widehat{W}(k^*)\right]a^3(\tau)\\
        &\,\,+\frac{\eta_n}{8}\chi^* q_ct_{2,n}\left[\widehat{W}(0)+\widehat{W}(2k^*)-2\widehat{W}(k^*)\right]a^3(\tau)+\frac{\eta_n}{8}\chi^* q_nt_{2,c}\left[\widehat{W}(0)-\widehat{W}(2k^*)\right]a^3(\tau)\\
        &\,\,+\frac{\eta_n}{32}(\chi^*)^3\bar{n}q_c^3\left[4\widehat{W}(k^*)-\widehat{W}(2k^*)-3\widehat{W}(0)\right]a^3(\tau)+f''(\bar{n})q_n\left(t_{0,n}+\frac{1}{2}t_{2,n}\right)a^3(\tau)\\
        &\,\,+\frac{1}{8}f'''(\bar{n})q_n^3a^3(\tau),\\
        H_{1,c}=&\,\,-q_ca'(\tau).
    \end{aligned}
\end{equation}
To deduce the desired qualitative properties for our weakly nonlinear analysis, we do not need to use the coefficients of $\mathbf{H}_3$, the vector of coefficients for $\cos\left(6\pi k^* x\right)$. In particular, the terms proportional to $\cos(2 \pi k^* x)$ on the righthand side raise the possibility of a resonant solution due to the fact that solutions proportional to $\cos(2 \pi k^* x)$ form the solution to the associated linear homogeneous equation, so applying a solvability condition to the $\mc{O}(\varepsilon^3)$ problem will eventually allow us to derive the amplitude equation $a(\tau)$ describing the dependence of our solution to the $\mc{O}(\varepsilon)$ problem in terms of the slow time variable $\tau = \varepsilon^2 t$.

Because the right-hand side of Equation~\eqref{eq:DM-graphon-ep3righthandexplicit} contains terms proportional to the function $\cos\left(2\pi k^* x\right)$ that appears in the kernel of the linear operator $\mathcal{D}_{\chi^*}$, we need to ensure that this term does not produce a resonant solution. We can apply the Fredholm alternative to see that the term $\mathbf{H}_1 \cos\left(2\pi k^* x\right)$ must be orthogonal to the kernel of the adjoint operator $\mathcal{D}_{\chi^*}^{\dagger}$ given by
\begin{equation}
\mathcal{D}_{\chi^*}^{\dagger} = \begin{bmatrix}
\frac{\eta_n}{2} \mc{L}_W + f'(\bar n) & \alpha \\ - \frac{\eta_n \chi^* \bar n}{2} \mc{L}_W & \frac{\eta_c}{2} \mc{L}_W - \beta
\end{bmatrix}.
\end{equation}
Because this operator is evaluated at the critical chemotactic sensitivity $\chi^*$, the kernel of $\mathcal{D}_{\chi^*}^{\dagger}$ will take the form $\mathbf{p}\cos\left(2\pi k^* x\right)$ for some vector of constants $\mathbf{p} = (p_n,p_c)^{\mathrm{t}}$. Considering our solution space as $\left(L^2\left([0,1]\right)\right)^2$, we can then look to rule out the possibility of resonant solutions by showing that the term $\mathbf{H}_1 \cos\left(2\pi k^* x\right)$ is orthogonal to $\mathbf{p}\cos\left(2\pi k^* x\right) \in \ker\big(\mathcal{D}_{\chi^*}^{\dagger} \big)$ in the sense that
\begin{equation}
\int_0^1 \left[ \mathbf{H}_1 \cos\left(2\pi k^* x\right)  \right] \cdot \left[ \mathbf{p}  \cos\left(2\pi k^* x\right) \right] dx = 0 \Longrightarrow \frac{1}{2} \mathbf{H}_1 \cdot \mathbf{p} = 0.
\end{equation}
This allows us to deduce that we need the coefficient function $\mathbf{H}_1$ to be orthogonal to the vector $\mathbf{p} = (p_n,p_c)^{\mathrm{t}}$.

To proceed with our application of the Fredholm alternative, we need to find an explicit expression for the vector $\mathbf{p}$. One way that we can do this is to find a simplified expression for a representative element of $\mathcal{D}_{\chi^*}^{\dagger}$ by deriving an element of the kernel in the form $\mathbf{p_0} \cos\left( 2 \pi k^* x\right)$ with $\mathbf{p_0} = (1,p_{n,0})^{\mathrm{t}}$. We can then scale this representative element so that we have the following convenient property $\mathbf{p} \cdot \mathbf{q} = 1$ for our representations of the kernels of $\mathcal{D}_{\chi^*}$ and $\mathcal{D}_{\chi^*}^{\dagger}$. This procedure yields the following expressions for the elements $p_n$ and $p_c$ of $\mathbf{p}$ 
\begin{equation}\label{eq:pn-pc}
    \begin{aligned}
    p_n&=\frac{2\beta + \eta_c\big[\widehat{W}(0)-\widehat{W}(k^*)\big]}{\left(\eta_n+\eta_c\right)\big[\widehat{W}(0)-\widehat{W}(k^*)\big]+2\left(\beta-f'(\bar{n})\right)} \\
        p_c&=\frac{\Big\{\beta+\frac{1}{2}\eta_c\big[\widehat{W}(0)-\widehat{W}(k^*)\big]\Big\}\Big\{\eta_n\big[\widehat{W}(0)-\widehat{W}(k^*)\big]-2f'(\bar{n})\Big\}}{\alpha\left(\eta_n+\eta_c\right)\big[\widehat{W}(0)-\widehat{W}(k^*)\big]+2\alpha\left[\beta-f'(\bar{n})\right]},
    \end{aligned}
\end{equation}
and we can use these expressions to find the conditions required to satisfy the orthogonality condition $\mathbf{p} \cdot \mathbf{H}_1 = 0$.

Using the expressions for the components of $\mathbf{p}$ and $\mathcal{H}_1$ from Equations~\eqref{eq:pn-pc} and Equation~\eqref{eq:H1-expressions}, we can see that the condition $\mathbf{p} \cdot \mathbf{H}_1 = 0$ results in the following ODE that must be satisfied by the amplitude function $a(\tau)$ for the solution to the $\mc{O}(\varepsilon)$ problem:
\begin{equation}\label{eq:amplitude-eq}
    \frac{da(\tau)}{d\tau}=\kappa\tilde{\chi}a+\gamma a^3,
\end{equation}
where 
\begin{equation}\label{eq:wnl-coffcients}
    \begin{aligned}
        \kappa=&\,\,\frac{\eta_n}{2}\bar{n}p_nq_c\big[\widehat{W}(0)-\widehat{W}(k^*)\big],\\
        \gamma=&\,\,\frac{\eta_n}{2}\chi^* p_nq_ct_{0,n}\left[\widehat{W}(0)-\widehat{W}(k^*)\right]\\
        &\,\,+\frac{\eta_n}{8}\chi^* p_nq_ct_{2,n}\left[\widehat{W}(0)+\widehat{W}(2k^*)-2\widehat{W}(k^*)\right]+\frac{\eta_n}{8}\chi^* p_nq_nt_{2,c}\left[\widehat{W}(0)-\widehat{W}(2k^*)\right]\\
        &\,\,+\frac{\eta_n}{32}(\chi^*)^3\bar{n}p_nq_c^3\left[4\widehat{W}(k^*)-\widehat{W}(2k^*)-3\widehat{W}(0)\right]+f''(\bar{n})p_nq_n\left(t_{0,n}+\frac{1}{2}t_{2,n}\right)\\
        &\,\,+\frac{1}{8}f'''(\bar{n})p_nq_n^3.
    \end{aligned}
\end{equation}
This is a cubic Stuart-Landau equation, and we refer to this differential equation as the cubic amplitude equation for our pattern-forming instability. We will be able to use this amplitude equation to characterize the form of our bifurcation and to calculate approximate expressions for our spatial pattern close to the onset of instability. 

We can use information from our linearization and the previous steps of our weakly nonlinear analysis to show that the coefficient $\kappa$ is always positive. We know from Equation \eqref{eq:qc} that $q_c$ will be positive since all parameters are assumed to be positive. In addition, we can use the fact that the eigenvalues of the graphon Laplacian satisfy $\widehat{W}(k) - \widehat{W}(0) < 0$ and our assumption that $\beta > 0$, and our assumption that $f'(\bar n) < 0$ to deduce that $p_n$ and $p_c$ in \eqref{eq:pn-pc} are also positive. This allows us to deduce that the coefficient $\kappa$ in the amplitude Equation~\eqref{eq:amplitude-eq} is always positive. 

The lengthy expression for the cubic coefficient $\gamma$ is more difficult to assess, and the sign of this term does appear to be clear to characterize by inspection of the analytical expression in \eqref{eq:wnl-coffcients}. To determine the sign of $\gamma$, we will instead turn to a numerical approach to explore the qualitative behavior of the amplitude equation for sample reaction functions $f(n)$ and graphon kernels $W(x-y)$. We present this exploration in Section \ref{sec:graphon-weakly-alternative}.

The sign of $\gamma$ determines the local criticality of the bifurcation described by the cubic amplitude equation. Since $\kappa>0$, if $\gamma<0$ and $\tilde{\chi}>0$, the trivial state $a=0$ becomes unstable and
\begin{equation*}
    a_\pm = \pm\sqrt{-\frac{\kappa\tilde{\chi}}{\gamma}}
\end{equation*}
emerge as locally stable equilibria. This corresponds to a supercritical bifurcation, where the critical mode saturates at an amplitude of order $O(\varepsilon)$ near onset. Conversely, if $\gamma>0$, the bifurcation is subcritical: unstable nonzero equilibria exist for $\tilde{\chi}<0$, while $a=0$ remains locally stable below the threshold.

\begin{Remark}
Our approach is based on the assumption that there is a unique critical wavenumber $k^*$ satisfying $\chi^* = \chi^*(k^*)$, which allows us to study the change in the amplitude of this single unstable wavenumber and to assume that both eigenvalues of the linearization matrix $M_{\chi^*}(k)$ will have negative real part for all other integer wavenumbers $k$. If $\chi^* < \chi^{**}$ but there are multiple critical wavenumbers $k$ satisfying $\chi^*(k) = \chi^*$, then we would have to incorporate each of the unstable wavenumbers into our weakly nonlinear analysis and derive a system of amplitude equations to describe the emerging amplitudes for each unstable wavenumber. This possibility of multiple critical wavenumbers was addressed by \cite{bramburger2023pattern} for the case of the Swift-Hohenberg equation of graphons, who use a multi-dimensional center manifold reduction to study the stability of emerging patterns when multiple wavenumbers undergo Turing instability at the same parameter threshold.
\end{Remark}

\begin{remark}
In the case of a subcritical pitchfork bifurcation, the cubic amplitude equation does not make a prediction for a stable, nonzero equilibrium amplitude for any value of $\gamma$, and we instead see that the amplitude $a(\tau)$ from Equation \eqref{eq:amplitude-eq} would be expected to go off to infinity for $\gamma > 0$ and a nonzero initial amplitude. To make a more refined asymptotic approximation of the amplitude of emerging patterns, we could perform a multiple-timescale approximation in our weakly nonlinear expansion by considering two slow timescales $T_1 = \varepsilon^2 t$ and $T_2 = \varepsilon^4 t$. Such an expansion produces a quintic Stuart-Landau equation that features two stable nonzero equilibria for sufficiently large $\tilde{\gamma}$, and these predicted equilibrium amplitudes and weakly nonlinear expansion of our spatial patterns have been used to show good agreement between asymptotic expansions and numerical simulations of PDE models sufficiently close to the onset of the pattern-forming instability \citep{short2010nonlinear,tulumello2014cross,giunta2021pattern,farivar2025intraguild}.
\end{remark}

\subsection{Weakly Nonlinear Analysis for Different Reactions and Kernels}
\label{sec:graphon-weakly-alternative}

In this section, we apply our weakly nonlinear expansion to explore the pattern-forming bifurcation for example graphon adjacency functions and sample reaction functions of $f(n)$. We first characterize the form of pitchfork bifurcation for a uniform band kernel and a logistic reaction function, before exploring how changing the kernel to a small-world graphon or allowing a population growth function $f(n)$ that is locally convex at the equilibrium population size $\bar n$. Beyond characterizing the direction of the pitchfork bifurcation, we will also explore how the choice of reaction function may impact the total population size and the total chemical concentration achieved in patterned states close to the onset of equilibrium.

We start our analysis by considering the case of the logistic reaction function, which is given by
\begin{equation}
    f(n) = rn\left(1 - \frac{n}{\bar{n}}\right).
\end{equation}
Near the homogeneous equilibrium $n(t,x) \equiv \bar{n}$ and $c(t,x) \equiv \bar{c}$, we have
\begin{equation*}
    f'(\bar{n})=-r<0,\quad f''(\bar{n})=-\frac{2r}{\bar{n}},\quad f'''(\bar{n})=0.
\end{equation*}

We will now use these expressions for the derivatives of $f(n)$ at the equilibrium value $\bar n$ to determine the predictions of weakly nonlinear analysis. For the uniform band graphon, we use these properties of the reactions terms, the values of the Fourier coefficients $\widehat{W}(k)$ from Equation \eqref{eq:Wvaluesring}, and the threshold $\chi^*$ from Equation \eqref{eq:chistaruniform} to determine the values of the constant $\gamma$ and other appropriate quantities in our expansion. Correspondingly, for the case of the Watts--Strogatz graphon with adjacency function given by
\begin{equation*}
W(x,y)=
\begin{cases}
1-p, & d(x,y)\le \rho\\[2pt]
p, & d(x,y)>\rho,
\end{cases}
\end{equation*}
we can calculate that the corresponding Fourier coefficients are given by
\begin{subequations}
\begin{equation}
\widehat{W}(0) = p + 2 \rho \left( 1 - 2p\right) 
\end{equation}
and
\begin{equation}
\widehat{W}(k) = 2 \rho \left(1-2p\right) \operatorname{sinc}\left( 2 \pi k \rho\right)
\end{equation}
\end{subequations}
for nonzero integer wavenumbers $k$, which means that we can calculate the threshold for the onset of instability by plugging these expressions for $\widehat{W}(0)$ and $\widehat{W}(k)$ into our expression for $\chi^*(k)$ from Equation \eqref{eq:threshold-w}.

We then use these expressions and the definition of the cubic coefficient $\gamma$ from Equation \eqref{eq:wnl-coffcients} to characterize the nature of the pattern-forming instability for the case of the logistic reaction function and either the uniform band graph or the Watts--Strogatz graphon. We display the resulting values of $\gamma$ in Figure \ref{fig:bifurcation-type}, observing that $\gamma < 0$ when the band radius $\rho$ is sufficiently small for the uniform band kernel and when the rewiring probability is sufficiently small for the Watts--Strogatz graphon. This indicates that the pattern-forming instability corresponds to a supercritical pitchfork bifurcation for the majority of the range of possible band radii $\rho \in (0,1/2)$ or rewiring probabilities $p \in (0,1/2)$. However, we also see that $\gamma > 0$ for $\rho$ sufficiently close to $\frac{1}{2}$ for the uniform band graphon and for $p$ sufficiently close to $\frac{1}{2}$ for the Watts--Strogatz graphon, indicating that subcritical pitchfork bifurcations are possible when long-range connections are sufficiently abundant between patches in our metapopulation. As the uniform band graphon and the Watts-Strogatz graphon both approach a complete graphon in the respective limits as $\rho \to \frac{1}{2}$ and as $p \to \frac{1}{2}$, so we see that subcritical bifurcations and corresponding suggested potential bistability of the uniform state and patterned states appear to arise for sufficiently global connectivity among patches.

\begin{figure}[!ht]
    \centering
    \begin{subfigure}[t]{0.48\linewidth}
        \centering
        \includegraphics[width=\linewidth]{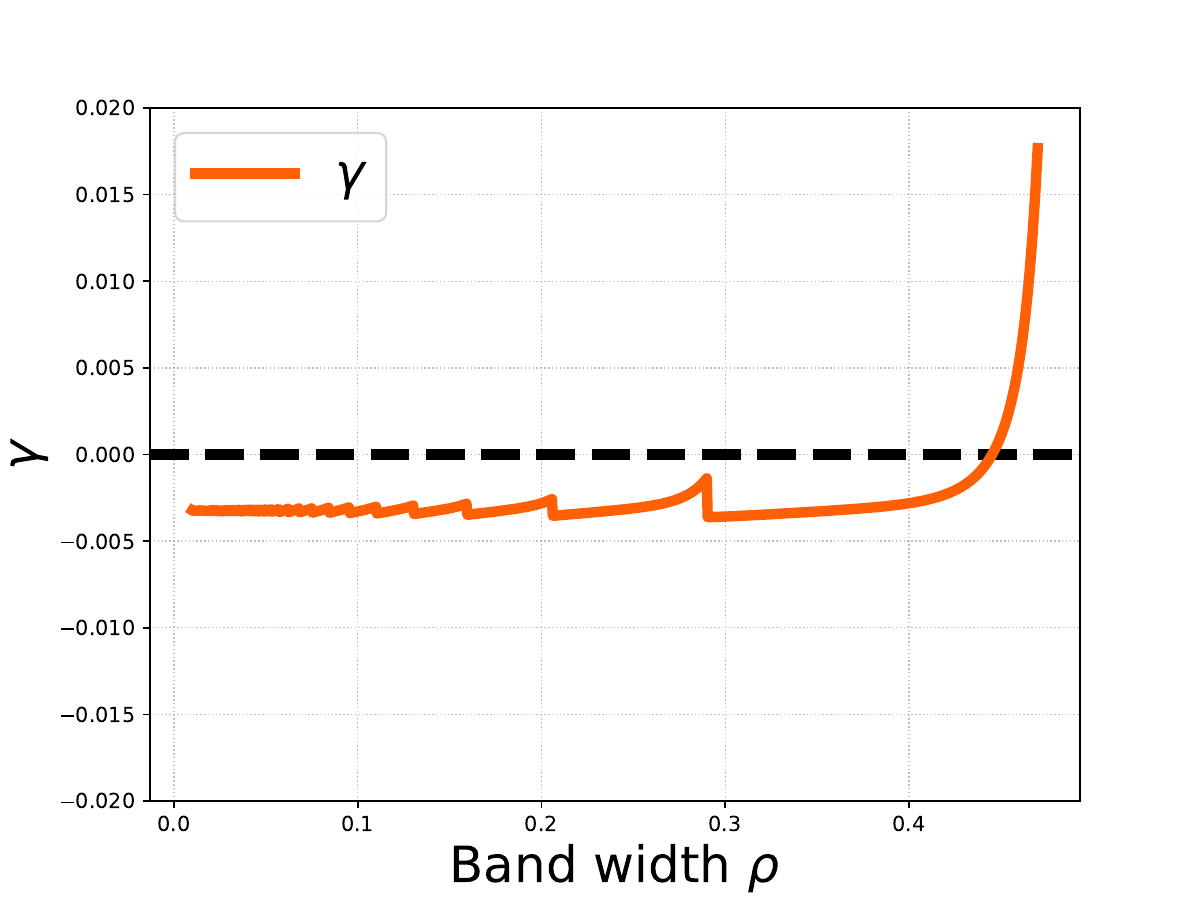}
        \caption{Uniform band graphon kernel}
    \end{subfigure}
    %\hspace{0.05\linewidth}
    \begin{subfigure}[t]{0.48\linewidth}
        \centering
        \includegraphics[width=\linewidth]{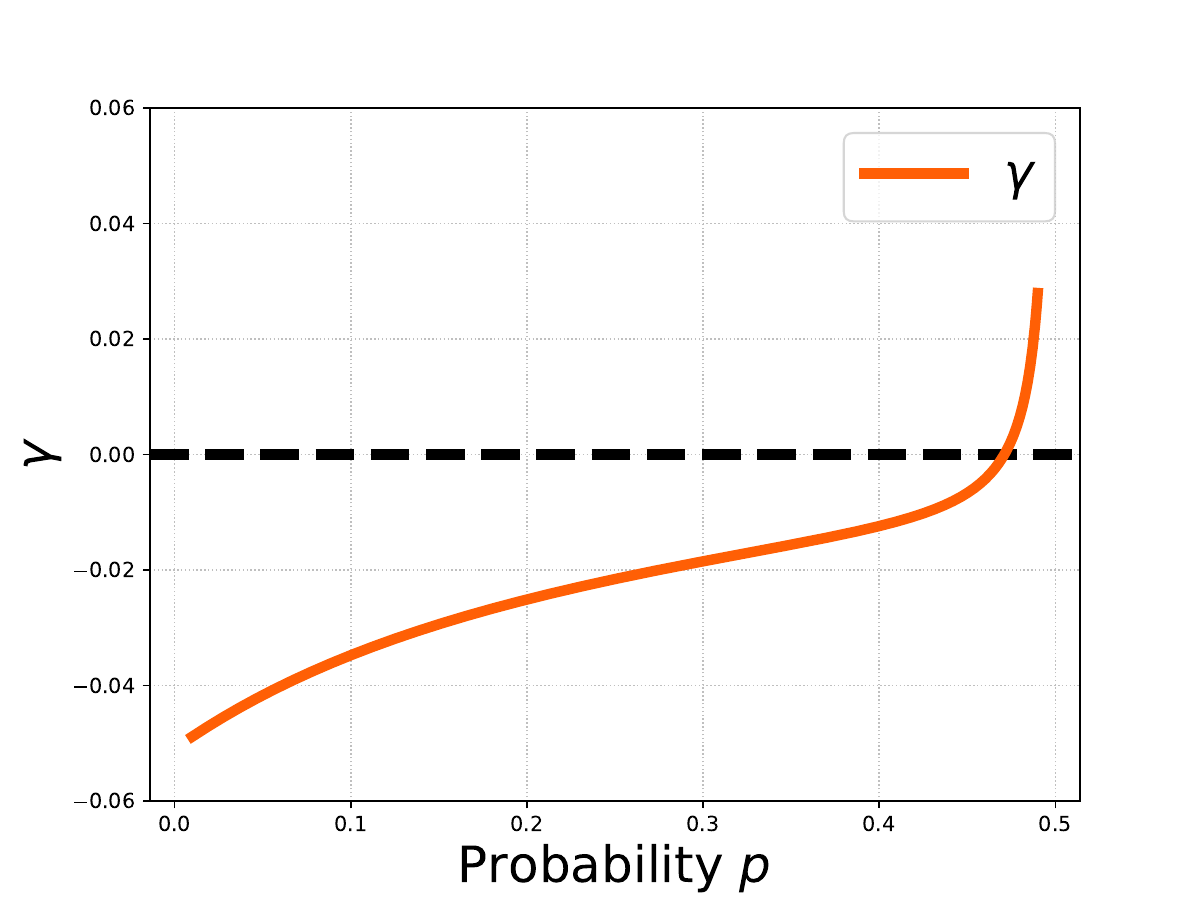}
        \caption{Small-World graphon kernel}
    \end{subfigure}
    \caption{Characterization of the form of pitchfork bifurcation for the graphon model in the cases of the uniform band kernel (left panel) and the small-world kernel (right panel). Each panel presents the value of the cubic-order coefficient $\gamma$ in the amplitude Equation~\eqref{eq:amplitude-eq} as a function of the band width $\rho$ in the left panel and the rewiring probability $p$ in the right panel. The weakly nonlinear analysis is performed for the logistic reaction function $f(n)=rn(1-n/\bar{n})$ with reaction parameters given by $\eta_n=0.05$, $\eta_c=0.08$, $\alpha=0.2$, $\beta=0.1$, $r=0.1$, $\bar{n}=8$, and $\bar{c}=16$. For each $\rho \in (0,1/2)$ or $p \in (0,1/2)$, the threshold $\chi^*$ is determined by Equation~\eqref{eq:threshold-wmin}, and the dashed lines correspond to the threshold value of $\gamma = 0$ separating the regimes of a supercritical pitchfork bifurcation (corresponding to $\gamma < 0$) and a subcritical pitchfork bifurcation (corresponding to $\gamma > 0$). The discontinuities of $\gamma$ in the left panel correspond to discrete changes in the critical wavenumber with the change of $\rho$.}
    \label{fig:bifurcation-type}
\end{figure}

We can also use our expansion obtained from weakly nonlinear analysis to obtain a heuristic approximation for our spatially patterned profiles of the population and chemical attractant. From solving the problems at orders $\mc{O}(\varepsilon)$ and $\mc{O}\left(\varepsilon^2\right)$, we can use our expressions from Equations \eqref{eq:first-order-eq} and \eqref{eq:second-order-eq} to write our approximate solution obtained from weakly nonlinear analysis in terms of the following expansion 
\begin{subequations} \label{eq:expanded-solution-tau}
\begin{align}
n(\tau,x) &= \bar n + \varepsilon  q_n a(\tau) \cos\left( 2 \pi k^* x\right) + \varepsilon^2 a(\tau)^2 \left[t_{0,n} + t_{2,n} \cos\left( 4 \pi k^* x \right) \right] + \mc{O}\left(\varepsilon^3\right) \\
c(\tau,x) &= \frac{\alpha \bar n}{\beta} + \varepsilon  q_c a(\tau) \cos\left( 2 \pi k^* x\right) + \varepsilon^2 a(\tau)^2 \left[t_{0,c} + t_{2,c} \cos\left( 4 \pi k^* x \right) \right] + \mc{O}\left(\varepsilon^3\right),
\end{align}
\end{subequations}
where $a(\tau)$ is the solution to our amplitude equation from Equation \eqref{eq:amplitude-eq}. In particular, this approximations suggests that a guess for the limiting spatial profiles of population and chemical attractant as $\tau \to \infty$ to is given by one of two possible asymptotic approximations for the steady-state densities of our model
\begin{subequations} \label{eq:expanded-solution-eq}
\begin{align}
\lim_{\tau \to \infty} n(\tau,x) &= \bar n + \varepsilon  q_n a_{\pm} \cos\left( 2 \pi k^* x\right) + \varepsilon^2 a_{\pm}^2 \left[t_{0,n} + t_{2,n} \cos\left( 4 \pi k^* x \right) \right] + \mc{O}\left(\varepsilon^3\right) \\
\lim_{\tau \to \infty} c(\tau,x) &= \frac{\alpha \bar n}{\beta} + \varepsilon  q_c a_{\pm} \cos\left( 2 \pi k^* x\right) + \varepsilon^2 a_{\pm}^2 \left[t_{0,c} + t_{2,c} \cos\left( 4 \pi k^* x \right) \right] + \mc{O}\left(\varepsilon^3\right),
\end{align}
\end{subequations}
where we note that this expansion for the steady state will depend on which of the two equilibria $a_-$ or $a_+$ are chosen for the amplitude equation. 

These expanded solutions and expanded steady-state densities will serve as a benchmark for numerical solutions that we will study in Section \ref{sec:numericalsimulations}. In addition, the form of these expansions will allow us to predict another key qualitative property of emerging patterns, the total population size $\int_0^{1}n(t,x) dx$ and the total chemical concentration $\int_0^1 c(t,x) dx$ across the spatial domain. In particular, we can integrate both sides of our expansions in Equation \eqref{eq:expanded-solution-tau} to write the total masses of our approximations of $n(\tau,x)$ and $c(\tau,x)$ as 
\begin{equation}
    \begin{aligned}
        \int_{0}^{1}n(\tau,x)\,dx&= \bar{n}+\varepsilon^2a^2(\tau)t_{0,n}+O(\varepsilon^3),\\
        \int_{0}^{1}c(\tau,x)\,dx&= \bar{c}+\varepsilon^2a^2(\tau)t_{0,c}+O(\varepsilon^3),
    \end{aligned}
\end{equation}
and we see that the terms proportional to $\cos(2 \pi k^*x)$ and $\cos(4 \pi k^*x)$ make no contribution to the mass of solutions because they integrate to $0$ across the spatial domain. As $\bar n$ and $\bar c$ are the equilibrium masses of the population and chemical attractant in the spatially uniform state, we see that the signs of the terms $t_{0,n}$ and $t_{0,c}$ will determine whether our emergent patterns states featured increases or decreases in total population size and total chemical attractant relative to the values achieved in the absence of spatial patterns. 

We can calculate the components of $\mathbf{t}_0$ by combining the expressions from Equations \eqref{eq:second-order-RHS} and \eqref{eq:second-order-eq}. This allows us to that, for an arbitary reaction function $f(n)$, the two components are given by
\begin{equation}\label{eq:t_0-expressions}
    \begin{aligned}
        t_{0,n}&=-\frac{f''(\bar{n})}{4f'(\bar{n})}q_n^2=-\frac{f''(\bar{n})}{4f'(\bar{n})} \\%=-\frac{1}{2\bar{n}},\\
        t_{0,c}&=-\frac{\alpha f''(\bar{n})}{4\beta f'(\bar{n})}q_n^2=-\frac{\alpha f''(\bar{n})}{4\beta f'(\bar{n})},%=-\frac{\alpha}{2\beta \bar{n}}.
    \end{aligned}
\end{equation}
where we used our convention that $q_n = 1$ to simplify the expression for these coefficients. Because we assume that $f'(\bar n) < 0$ to guarantee of the spatially uniform equilibrium under the reaction dynamics, we can deduce that the sign of $t_{0,n}$ and $t_{0,c}$ will be determined by the sign of the second derivative $f''(\bar n)$ of the reaction function at the equilibrium population size $\bar n$. Notably, we find that the mass of population and chemical attractant will either jointly increase or jointly decrease, consistent with the spatial correlation of the population and chemoattractant profiles due to the form of the $\mc{O}(1)$ solution and matching intuition with the underlying positive feedback between chemoattractant production and biased motion of individuals for the chemoattractant.

We also deduce from Equation \eqref{eq:t_0-expressions} that $t_{0,n}$ and  $t_{0,c}$ will be positive when $f''(\bar n) < 0$, while these quantities will both be negative when $f''(\bar n) > 0$. This means that whether the masses of patterned solutions exceed those achieved in the spatially uniform state depends on whether function describing the net growth of the population is locally convex or locally concave at the equilibrium population size $\bar n$. In particular, for the case of logistic growth $f(n) = r n \left( 1 - \frac{n}{\bar n} \right)$, we can calculate that
\begin{subequations}
\begin{align}
f'(\bar n) &= -r < 0  \\
f''(\bar n) &= - \frac{2r}{\bar n} < 0,
\end{align}
\end{subequations}
so the logistic reaction function is globally concave and we can deduce that $t_{0,n}, t_{0,c} < 0$ for the logistic reaction function, resulting in spatially patterned states with decreased overall mass of population and chemoattractant relative to the masses achieved in the spatially uniform equilibrium. 

\begin{Remark}
While weakly nonlinear analysis provides a heuristic prediction of the mass of solutions close to the onset of instability, it is actually possible to obtain stronger results on determining the long-time decay in mass of solutions in our graphon Keller--Segel model for a class of reaction functions including logistic growth. In particular, for reaction functions $f(n)$ that are concave for all $n > 0$ and feature a single positive stable equilibrium, we can show that the long-time mass of the solutions cannot exceed the mass of the uniform equilibrium $(\bar n, \bar c)$ and that any spatially patterned equilibrium features strictly less mass than that of the uniform state. We present these results on the decay of mass in Section \ref{sec:graphon-concave-mass-decay}.
\end{Remark}

In addition to the logistic function, it is possible to show that several other well-known functions $f(n)$ for the demographic dynamics of a single population are globally concave, including the $\vartheta$-logistic function and a model with Holling type II functional response that we explore further in Section \ref{sec:graphon-concave-mass-decay}. In addition, the Gompertz growth model $f(n) = rn \log\left(\bar n/n\right)$ is known to feature an accelerating growth rate at low populations but settles to a decelerating concave population growth rate at the carrying capacity $\bar n$\citep{gompertz1825nature,west2016evolutionary,west2019cellular,asadi2023generalized,wang2024gompertz}. Because we do not know of a famous population growth function $f(n)$ that is locally convex at the equilibrium population size, we create a new reaction function that will allow for the possibility of increased population size and chemical concentration in patterned states. 

In particular, we consider the following family of reproduction functions
\begin{equation}
    f_\theta(n) = re^{\theta(\bar{n} - n)} n \left(1-\frac{n}{\bar{n}}\right),
\end{equation}
where we have multiplied the typical logistic reaction function by an exponential weight function $e^{\theta(\bar{n} - n)}$, where the parameter $\theta$ describes the exponential decline of the effective intrinsic growth rate $r$ as the population size increases. Due to the form of this equation, we refer to this reaction function as the exponentially modified logistic reaction, and we note that the reaction function $f_{\theta}(n)$ reduces to the original logistic model when $\theta = 0$. We also observe that increasing the value of $\theta$ can impact the sign of terms appearing in our weakly nonlinear stability analysis, as we can calculate that
\begin{equation*}
    f'_\theta(\bar{n})=-r<0,\quad f''_\theta(\bar{n})=2r\left(\theta-\frac{1}{\bar{n}}\right),\quad f'''_\theta(\bar{n})=3r\theta\left(\frac{2}{\bar{n}}-\theta\right).
\end{equation*}
The form of these derivatives allows us to see that the equilibrium $\bar n$ is always stable and we see that the signs of $f''_{\theta}(\bar n)$ and $f'''_{\theta}(\bar n)$ both change sign as $\theta$ increase. 
When $\theta > 1/\bar{n}$, we have $f_\theta''(\bar{n})=2r\left(\theta-1/\bar{n}\right)>0$, which allows us to deduce that $t_{0,n}$ and $t_{0,c}$ will be positive and the masses of our patterned solutions will exceed the masses achieved in the spatially uniform state when the chemotactic sensitivity is close to the threshold for instability. We further explore the effects of the exponentially modified logistic function in numerical simulations in Section \ref{sec:numericalsimulations}, showing that these predictions on the increase in mass of solutions also hold far from the onset of patterns and highlighting a potentially surprising robustness for the predictions of weakly nonlinear stability analysis beyond is guaranteed region of validity. 

We can also use the form of the derivatives calculated above to determine whether the pattern-forming bifurcation is subcritical or supercritical for a given graphon kernel and value of the parameter $\theta$ under the exponentially modified logistic reaction function. In particular, we show in  Figure~\ref{fig:bifurcation-type-exponential} the numerically computed value of the cubic-order coefficient $\gamma$ in the amplitude Equation~\eqref{eq:amplitude-eq} for the uniform band kernel and small-world kernel and different values of $\theta$. We see that increasing the sensitivity parameter $\theta$ can result in subcritical bifurcations at band radii $\rho$ or rewiring probabilities $p$ for which a supercritical pitchfork bifurcation is predicted for the case of the logistic reaction function. In particular, this means that the exponentially modified logistic reaction functions appears to be more conducive to bistability between uniform and patterned states than the standard logistic model, and that more global coupling through a longer-range neighborhood band or a greater probability of long-range connections can work in concert with a decelerating net growth rate to promote resilience in spatial patterns when the strength of chemotactic sensitivity decreases. 

\begin{figure}[!ht]
    \centering
    \begin{subfigure}[t]{0.48\linewidth}
        \centering
        \includegraphics[width=\linewidth]{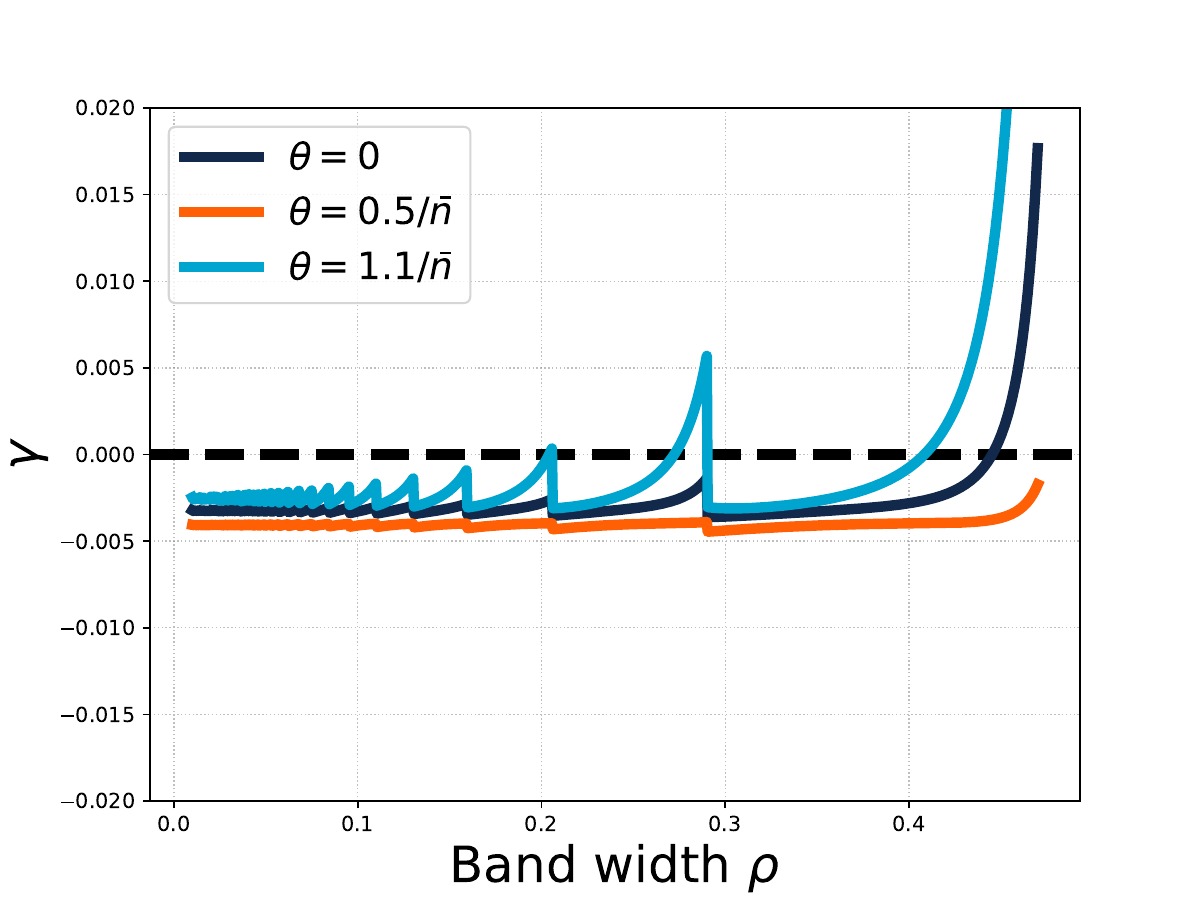}
        \caption{Uniform band graphon kernel}
    \end{subfigure}
    %\hspace{0.05\linewidth}
    \begin{subfigure}[t]{0.48\linewidth}
        \centering
        \includegraphics[width=\linewidth]{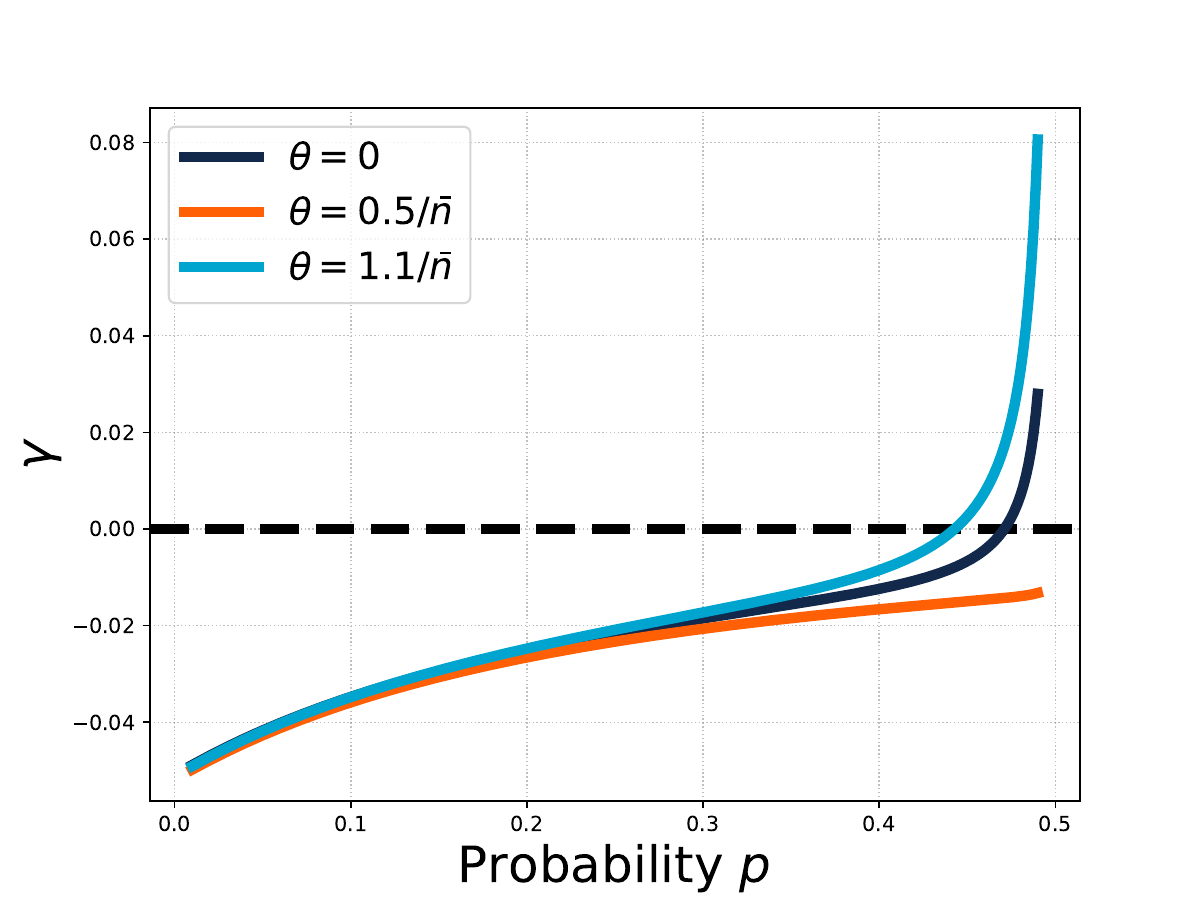}
        \caption{Small world graphon kernel}
    \end{subfigure}
    \caption{Predictions of the type of pattern-forming bifurcation for the uniform band kernel (left panel) and the small-world kernel (right panel). The panels show the cubic-order coefficient $\gamma$ in the amplitude Equation~\eqref{eq:amplitude-eq} as a function of the band width $\rho$ in the left panel and the probability $p$ in the right panel. The PDE parameters are $\eta_n=0.05$, $\eta_c=0.08$, $\alpha=0.2$, and $\beta=0.1$. The reaction term $f(n)$ is chosen to be the exponential-modified logistic function $re^{\theta(\bar{n}-n)}n(1-n/\bar{n})$, with $r=0.1$ and $\bar{n}=8$, and we set $\bar{c}=16$. The parameter $\theta$ is taken to be $0$, $0.5/\bar{n}$, and $1.1/\bar{n}$. We take $\rho\in(0,1/2)$ for the uniform band kernel and $p\in(0,1/2)$ for the small-world kernel, with $\chi^*$ determined by Equation~\eqref{eq:threshold-wmin}.}
    \label{fig:bifurcation-type-exponential}
\end{figure}

\begin{Remark}
    For the exponential-modified logistic function $f_\theta(n)$, the results of the linear stability analysis remain the same as those obtained with the logistic function $f(n) = r n \left( 1 - n / \bar n\right)$, since $f_\theta'(\bar{n}) = f'(\bar{n}) = -r$. This is a consequence of the choice to have the exponential weight take the form $e^{\theta (n- \bar n)}$, which simplifies to unity when $n = \bar n$. Although the form of the exponentially modified function is potentially somewhat artificial due to the choice of multiplying both the birth and death rate components of the logistic function by the same density-dependent exponential weight, the form of our reaction function allows us to keep the stability properties consistent with those of the logistic function while also allowing the possibility of local convexity. 
    
    Namely, by considering a reaction term of the form $f(n) = r g(n) n \left[ 1 - \frac{n}{\bar n} \right]$, we see that the 
    \[ f''(\bar n) = - 2 r \left[ \frac{g(\bar n)}{\bar n} - g'(\bar n) \right],\]
so picking a function $g(n)$ such that $g(\bar n) = 1$ and $g'(\bar n)$ is sufficiently negative will allow us to achieve a locally convex reaction function. In particular, the choice $g(n) = e^{\theta (\bar n - n)}$ with $\theta > \frac{1}{\bar n}$ will allow us to achieve local convexity at $n = \bar n$. 
\end{Remark}

\subsection{Global Mass Deficit of Steady States for Concave Reactions}
\label{sec:graphon-concave-mass-decay}

So far, we have show in Section \ref{sec:graphon-weakly-alternative} that weakly nonlinear analysis allows us to heuristically predict that the total population size and chemical concentration will be lower than in spatially uniform states provided that the reaction function $f(n)$ is locally concave at the equilibrium value $\bar n$. It turns out that we can prove that this decrease in solution mass for patterned states will hold whenever the reaction function $f(n)$ is strictly concave for non-negative population sizes $n$. In Lemma \ref{lem:logistic-mass-bounds}, we use Jensen's inequality and a comparison principle to show that the long-time total population and total chemical concentration cannot exceed $\bar n$ and $\bar c$. For steady-state solutions, we show in Proposition \ref{prop:logistic-total-mass-deficit} that the total mass of solutions will be less than the values achieved in the spatially uniform state unless $n(t,x) = \bar n$ and $c(t,x) = \bar c$ almost everywhere on the domain. For the case of the logistic reaction function, these results can be made more explicit due to having an explicit expression for the upper bound on solution mass and due to an explicit bound on the mass of the steady-state solution in terms of the variance of the steady-state density. 

\begin{lemma}[Mass identities and comparison]
\label{lem:logistic-mass-bounds}
Let $I=[0,1]$ and $W:I^2\to[0,\infty)$ be measurable and symmetric, with $\operatorname*{ess\,sup}_{x\in I}\int_I W(x,y)\,dy<\infty$. Assume $\eta_n,\eta_c\ge0$, $\alpha,\beta>0$, and $f:\mathbb R\to\mathbb R$ is locally Lipschitz. Let $U=(n,c)\in C^1([0,\infty);L^\infty(I)^2)$ be a nonnegative solution of $\partial_tU=\mathcal F_W[U]$, and define $M_n(t):=\int_I n(t,x)\,dx$ and $M_c(t):=\int_I c(t,x)\,dx$. Then
\begin{equation}
\label{eq:graphon-total-mass-identities}
    \frac{dM_n}{dt}=\int_I f\bigl(n(t,x)\bigr)\,dx,
    \qquad
    \frac{dM_c}{dt}=\alpha M_n(t)-\beta M_c(t).
\end{equation}
Suppose further that $f$ is concave on $[0,\infty)$ with $f(0)=0$ and a unique positive zero $\bar n$. The solution $z$ of
\begin{equation*}
    z'=f(z),\qquad z(0)=M_n(0)
\end{equation*}
exists globally and, for all $t\ge0$, satisfies
\begin{equation}
\label{eq:concave-mass-comparison}
    0\le M_n(t)\le z(t)\le\max\{M_n(0),\bar n\}.
\end{equation}
In addition, we may estimate that
\begin{equation}
\label{eq:chemical-mass-comparison}
    0\le M_c(t)\le e^{-\beta t}M_c(0)
    +\alpha\int_0^t e^{-\beta(t-s)}z(s)\,ds.
\end{equation}
Notating that $\bar c:=\alpha\bar n/\beta$, we have
\begin{equation}
\label{eq:logistic-mass-limsup}
    \limsup_{t\to\infty}M_n(t)\le\bar n,
    \qquad
    \limsup_{t\to\infty}M_c(t)\le\bar c.
\end{equation}
In particular, if $M_n(0)\le\bar n$, then $M_n(t)\le\bar n$ for all $t\ge0$.

The same conclusions apply to nonnegative global solutions of the discrete system~\eqref{eq:KS-discrete} under symmetric nonnegative weights, replacing spatial integrals with the normalized sums $M_n^{(N)}(t):=\frac1N\sum_{k=1}^N n_k^{(N)}(t)$ and $M_c^{(N)}(t):=\frac1N\sum_{k=1}^N c_k^{(N)}(t)$.
\end{lemma}

\begin{proof}
Since $W \in L^1(I^2)$ and the solutions are bounded, Fubini's theorem applies. By the symmetry of $W$ and $S(a,b)+S(b,a)=1$, exchanging $x$ and $y$ cancels the transport terms, yielding~\eqref{eq:graphon-total-mass-identities}.

If $f$ is concave on $[0,\infty)$, Jensen's inequality gives
\begin{equation}
\label{eq:concave-mass-differential-inequality}
    M_n'(t)=\int_I f\bigl(n(t,x)\bigr)\,dx
    \le f\left(\int_I n(t,x)\,dx\right)=f(M_n(t)).
\end{equation}
The root conditions on $f$ ensure $z$ is global, bounded between $z(0)$ and $\bar n$, and converges monotonically to $\bar n$ (with $z\equiv 0$ if $z(0)=0$). Standard comparison yields~\eqref{eq:concave-mass-comparison} and the invariant bound $M_n\le\bar n$. Variation of constants in the $M_c$-equation gives~\eqref{eq:chemical-mass-comparison}.

Since $\limsup_{t\to\infty}z(t)\le\bar n$, the inequality $M_n(t)\le z(t)$ gives $\limsup_{t\to\infty}M_n(t)\le\bar n$. For any $\varepsilon>0$, choose $t_0$ such that $M_n(t)\le\bar n+\varepsilon$ for all $t\ge t_0$. Then,
\begin{align*}
    M_c(t)\le e^{-\beta(t-t_0)}M_c(t_0)
    +\frac{\alpha}{\beta}(\bar n+\varepsilon)
    \bigl(1-e^{-\beta(t-t_0)}\bigr),
    \qquad t\ge t_0.
\end{align*}
Sending $t\to\infty$ and $\varepsilon\downarrow0$ proves~\eqref{eq:logistic-mass-limsup}. The discrete case follows identically via normalized sums.
\end{proof}

\begin{proposition}[Steady-state mass deficit]
\label{prop:logistic-total-mass-deficit}
Assume the hypotheses of Lemma~\ref{lem:logistic-mass-bounds}, and let $f:\mathbb R\to\mathbb R$ be locally Lipschitz and concave on $[0,\infty)$, with $f(0)=0$ and unique positive zero $\bar n$. Let $U_*=(n_*,c_*)\in L^\infty(I)^2$ be a nonnegative steady state, and define $M_{n,*}:=\int_I n_*(x)\,dx$ and $M_{c,*}:=\int_I c_*(x)\,dx$. If $M_{n,*}>0$, then
\begin{equation*}
    M_{n,*}\le\bar n,
    \qquad
    M_{c,*}=\frac{\alpha}{\beta}M_{n,*}\le\bar c,
    \qquad \bar c:=\frac{\alpha}{\beta}\bar n.
\end{equation*}
If $f$ is strictly concave on $[0,\infty)$, equality in either bound holds if and only if $U_*\equiv(\bar n,\bar c)$ a.e., rendering the bounds strict for any nonconstant steady state.

For logistic growth $f(n)=rn(1-n/\bar n)$ with $r>0$, the deficit satisfies the exact variance identity
\begin{equation}
\label{eq:logistic-deficit-identity}
    \bar n-M_{n,*}
    =\frac{1}{M_{n,*}}
    \int_I\bigl(n_*(x)-M_{n,*}\bigr)^2\,dx.
\end{equation}
The same conclusions apply to~\eqref{eq:KS-discrete} under symmetric nonnegative weights and normalized masses.
\end{proposition}

\begin{proof}
At steady state, the mass identities and Jensen's inequality yield
\begin{equation*}
    0=\int_I f(n_*(x))\,dx\le f(M_{n,*}),
    \qquad M_{c,*}=\frac{\alpha}{\beta}M_{n,*}.
\end{equation*}
Since $f(s)<0$ for $s>\bar n$, both mass bounds immediately follow.

If $f$ is strictly concave, equality in either bound requires $M_{n,*}=\bar n$, forcing equality in Jensen's inequality and thus $n_*\equiv\bar n$ a.e. Setting $\mathfrak{c}:=c_*-\bar c$, the steady-state chemical equation reduces to $\frac{\eta_c}{2}\mathcal L_W\mathfrak{c}-\beta\mathfrak{c}=0$. Multiplying by $\mathfrak{c}$ and integrating gives
\begin{equation*}
    -\frac{\eta_c}{4}\int_I\int_I
    W(x,y)\bigl(\mathfrak{c}(y)-\mathfrak{c}(x)\bigr)^2\,dy\,dx
    -\beta\int_I\mathfrak{c}(x)^2\,dx=0.
\end{equation*}
Since $W,\eta_c\ge0$ and $\beta>0$, both terms are nonpositive, forcing $\mathfrak{c}=0$ a.e., whence $c_*\equiv\bar c$. The converse is immediate.

For logistic growth, expanding $\int_I n^2\,dx=M_n^2+\int_I(n-M_n)^2\,dx$ in the mass identity yields
\begin{equation}
\label{eq:logistic-mean-var-identity}
    \frac{dM_n}{dt}
    =rM_n(t)\left(1-\frac{M_n(t)}{\bar n}\right)
    -\frac{r}{\bar n}\int_I\bigl(n(t,x)-M_n(t)\bigr)^2\,dx.
\end{equation}
At steady state, dividing by $(r/\bar n)M_{n,*}>0$ yields \eqref{eq:logistic-deficit-identity}. The discrete case follows identically via normalized sums.
\end{proof}

\paragraph{Examples.}
The following reaction terms are defined for $n\ge0$ and admit locally Lipschitz extensions to $\mathbb R$.

For logistic growth $f(n)=rn(1-n/\bar n)$, the comparison function is explicitly
\begin{equation}
\label{eq:logistic-mass-comparison}
    z(t)=\frac{\bar n M_n(0)}
    {M_n(0)+(\bar n-M_n(0))e^{-rt}},
\end{equation}
with $z\equiv0$ if $M_n(0)=0$.

Another family of population models featuring an unstable extinction equilibrium $n = 0$ and a stable, positive carrying-capacity equilibrium $\bar n$ is the theta-logistic model (also written as the $\vartheta$-logistic model), which generalizes logistic growth to incorporate different possible shapes for the term characterizing density-dependent regulation \citep{gilpin1973global,ross2009note}. The reaction function for the $\vartheta$-logistic model is given by $f(n)=rn\bigl(1-(n/\bar n)^{\vartheta}\bigr)$, recovering the standard logistic model for the case of the $\vartheta = 1$. We can show that $f''(n)=-r\vartheta(1+\vartheta)\bar n^{-\vartheta}n^{\vartheta-1}<0$ for $n>0$ provided that $r,\bar n,\vartheta>0$ . Therefore the $\vartheta$-logistic reaction function is strictly concave on $[0,\infty)$, and therefore the results of Lemma \ref{lem:logistic-mass-bounds} and Proposition \ref{prop:logistic-total-mass-deficit} also apply to this generalized form of logistic growth. 

Another way to describe one-species population dynamics with a unique positive and stable equilibrium point $\bar{n}$ is to describe a population with constant per-capita death rate and a saturating per-capita birth rate corresponding to Holling type-II functional response \citep{holling1966functional,britton2003essential,martinez2013vegetation,hermsen2022emergent}. Such a reaction term can be written in the form
\[f(n)=n\left(\frac{a}{1+a h_0 n}-\delta\right),\] where $a>\delta>0$ and $h_0>0$. Here the constant $h_0$ is the handling time (distinct from the generic per-capita growth function $h(n)$) and the parameter $a$ (sometimes called the attack rate) decreases the per-capita birth rate at low population density. Direct differentiation yields the unique positive zero $\bar n=\displaystyle\frac{a-\delta}{a h_0\delta}$ and $f''(n)=-\displaystyle\frac{2a^2h_0}{(1+a h_0 n)^3}<0$ for $n\ge0$.

Consequently, both families are strictly concave on $[0,\infty)$ and admit unique positive zeros, satisfying the comparison bounds and the strict steady-state mass deficit property.

%%%%%%%%%%%%%%%%%%%%%%%%%%%%%%%%%%%%%%%%%%%%%%%%%%%%%%%%%%
%%%%%%%%%%%%%%%%%%%%%%%%%%%%%%%%%%%%%%%%%%%%%%%%%%%%%%%%%%
\section{Numerical Simulations}
\label{sec:numericalsimulations}

In this section, we perform numerical simulations of our Keller--Segel models on both graph and graphon metapopulations. We start in Section \ref{sec:simulation-graph} by performing numerical simulations of the ODE Keller--Segel system on graphs from Equations~\eqref{eq:KS_graph_n} and~\eqref{eq:KS_graph_c} for the case of a uniform band network, showing the emergence of non-uniform states for sufficiently large chemotaxis sensitivity $\chi$ and showing that the long-time behavior approaches the sinusoidal patterns achieved in the graphon model when we consider a metapopulation graph with large enough patch number $N$. We then consider numerical simulations of the graphon model from Equations~\eqref{eq:KS_graphon_n} and~\eqref{eq:KS_graphon_c} in Section~\ref{sec:simulation-graphon}, demonstrating good agreement with the qualitative behavior predicted from the linear and weakly nonlinear stability analyses for our integro-differential equation model. Furthermore, we use numerical simulations to explore the total population size and total chemoattractant concentration for the case of large $\chi$, showing that some predictions from weakly nonlinear analysis continue to hold far beyond the onset of pattern-forming instability. 

\subsection{Simulations for the Graph Model}\label{sec:simulation-graph}

In this section, we present results from the numerical simulations of the ODE system from Equations~\eqref{eq:KS_graph_n} and~\eqref{eq:KS_graph_c}, exploring the pattern-forming behavior for the graph metapopulation in the case of logistic population dynamics $f(n)=rn(1-n/\bar{n})$ and a family of ring graphs. In particular, we consider ring graphs in which patches are connected to the nearest 10 percent of patches on each of the ring, corresponding to the family of graphs that corresponds to a uniform band graphon with radius $\rho = 0.1$ in the limit of infinite patch size. For each of the networks presented in this section, we perform numerical simulations of equations~\eqref{eq:KS_graph_n} and~\eqref{eq:KS_graph_c} using the built-in Runge-Kutta solver \textit{ode45} in MATLAB. 

We start by simulating the graph Keller--Segel dynamics for a network with $N = 20$ patches (with connectivity corresponding to a ring graph with two neighbors on each side), and we display numerical trajectories of the population size $n_i(t)$ and chemoattractant $c_i(t)$ for five selected patches in Figure \ref{fig:graph-simulations}. For this finite graph, the linear instability threshold is $\min_{\{k:\lambda_k<0\}}\chi_k^*$, where $\chi_k^*$ is given by~\eqref{eq:ODE-threshold} and $\lambda_k$ are the eigenvalues of the scaled graph Laplacian $L=(A-D)/N$. In particular, by starting from initial conditions with a small perturbation from the uniform equilibrium state with $n_i = \bar{n}$ and $c_i = \frac{\alpha}{\beta} \bar{n}$, we are able to see that $n_i(t)$ and $c_i(t)$ move from a near-equilibrium value and settle to different equilibrium levels after several thousand time units.

\begin{figure}[htbp]
    \centering

    \begin{subfigure}[b]{0.48\textwidth}
        \includegraphics[width=\linewidth]{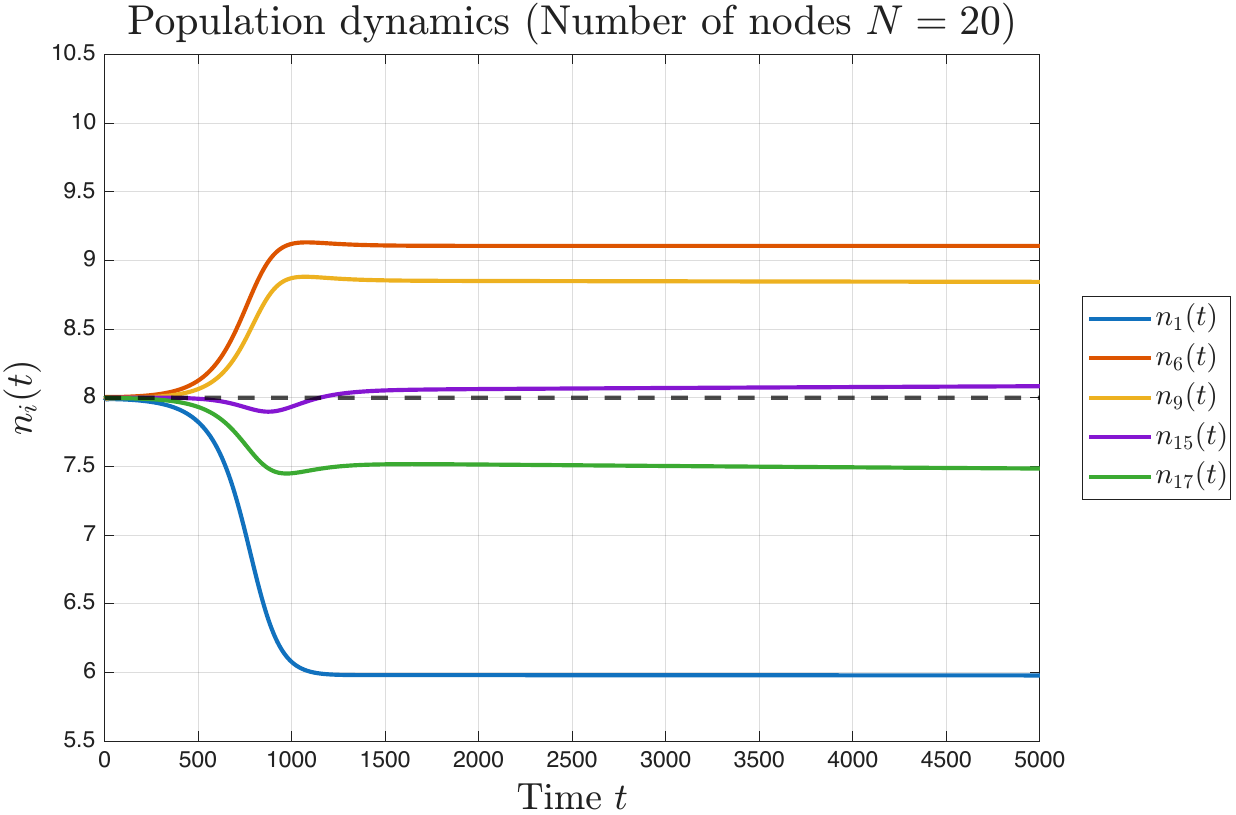}
    \end{subfigure}
    \hfill
    \begin{subfigure}[b]{0.48\textwidth}
        \includegraphics[width=\linewidth]{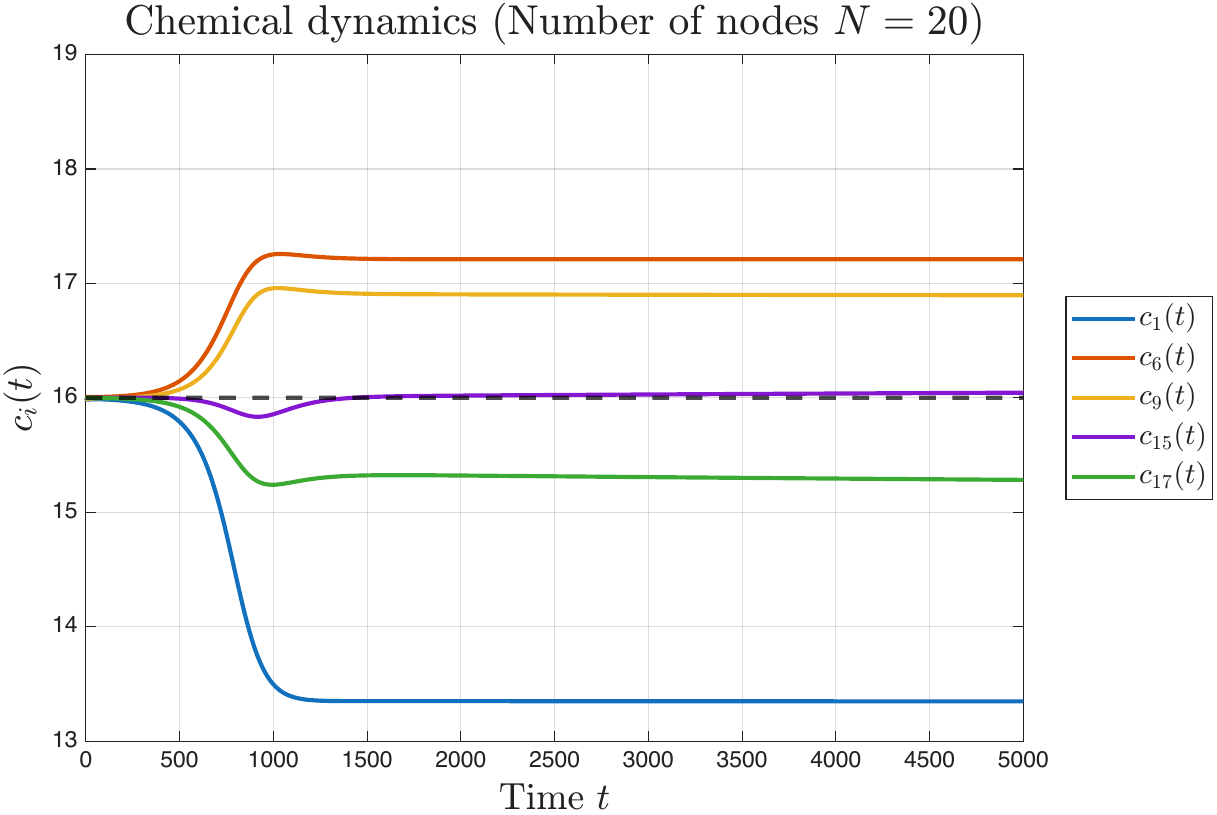}
    \end{subfigure}\vspace{-24pt}
    
    \vspace{2em}
    
    \caption{Numerical simulations of the graph Keller--Segel model from  Equations~\eqref{eq:KS_graph_n} and~\eqref{eq:KS_graph_c} for the case of 20 patches and a ring graph network connecting each patch to the closest two neighboring patches on each side (corresponding to the nearest 10\% on nodes on each side). The parameters are set to $\eta_n=0.05$, $\eta_c=0.08$, $\alpha=0.2$, $\beta=0.1$, $r=0.1$, $\bar{n}=8$, and $\bar{c}=16$. The horizontal axis represents time, and the vertical axis represents the solution value.}
    \label{fig:graph-simulations}
\end{figure}

Next, we show the states achieved at time $t=10{,}000$ under simulations of our graph metapopulation models starting from near-uniform initial states, varying the number of nodes in our metapopulation but fixing the proportion of neighboring patches to the closest 10 percent on nodes on each side. We display the resulting states achieved in Figure \ref{fig:graph-simulations-snapshots} for the cases of 20, 100, and 200 patches, and we see that a relatively irregular pattern of population density and chemoattractant concentration is achieved for the simulations with 20 nodes. For our simulations with 100 and 200 nodes, we see that the resulting patterns begin to appear more like a sinusoidal pattern as predicting for the limiting graphon model, and we see that the seven peaks appearing in the simulated patterns resembles the number of peaks predicted through linear stability analysis of the graphon model for the uniform band kernel with radius $\rho = 0.1$.

\begin{figure}[htbp]
    \centering

    \begin{subfigure}[b]{0.46\textwidth}
        \includegraphics[width=\linewidth]{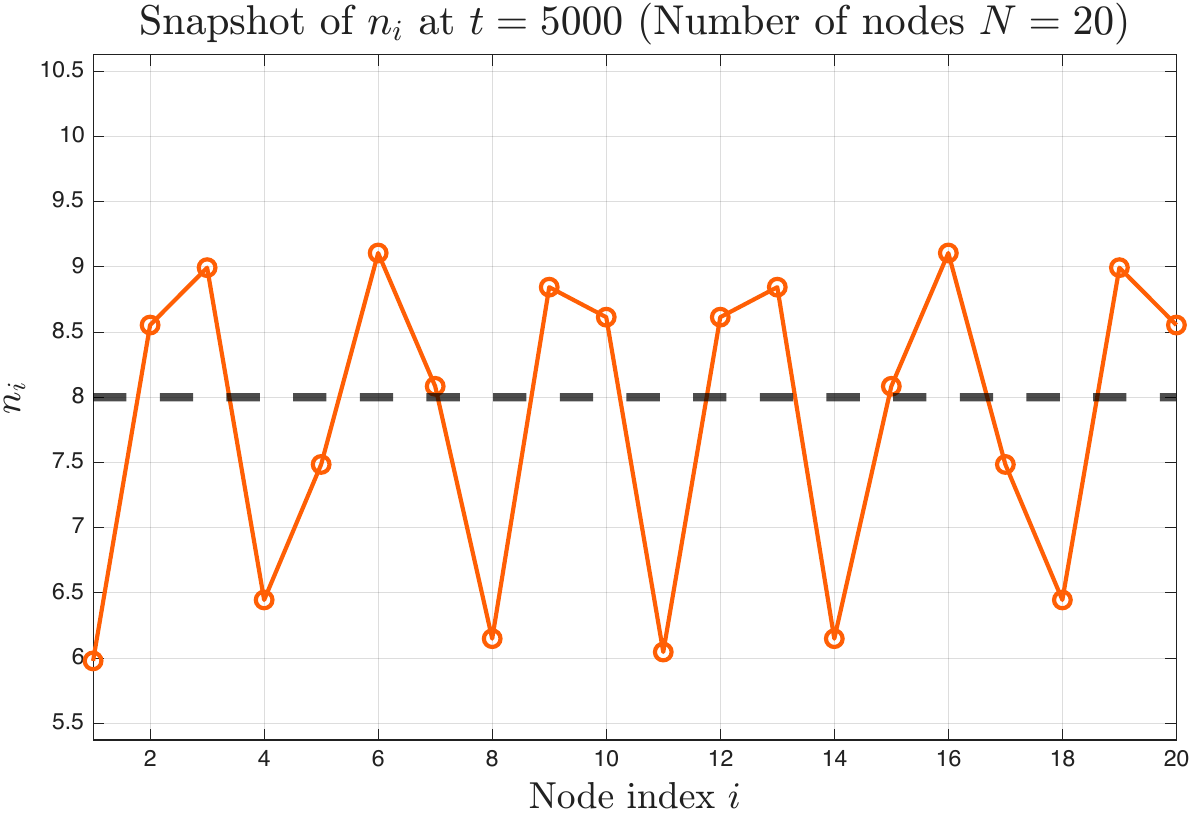}
    \end{subfigure}
    \hfill
    \begin{subfigure}[b]{0.46\textwidth}
        \includegraphics[width=\linewidth]{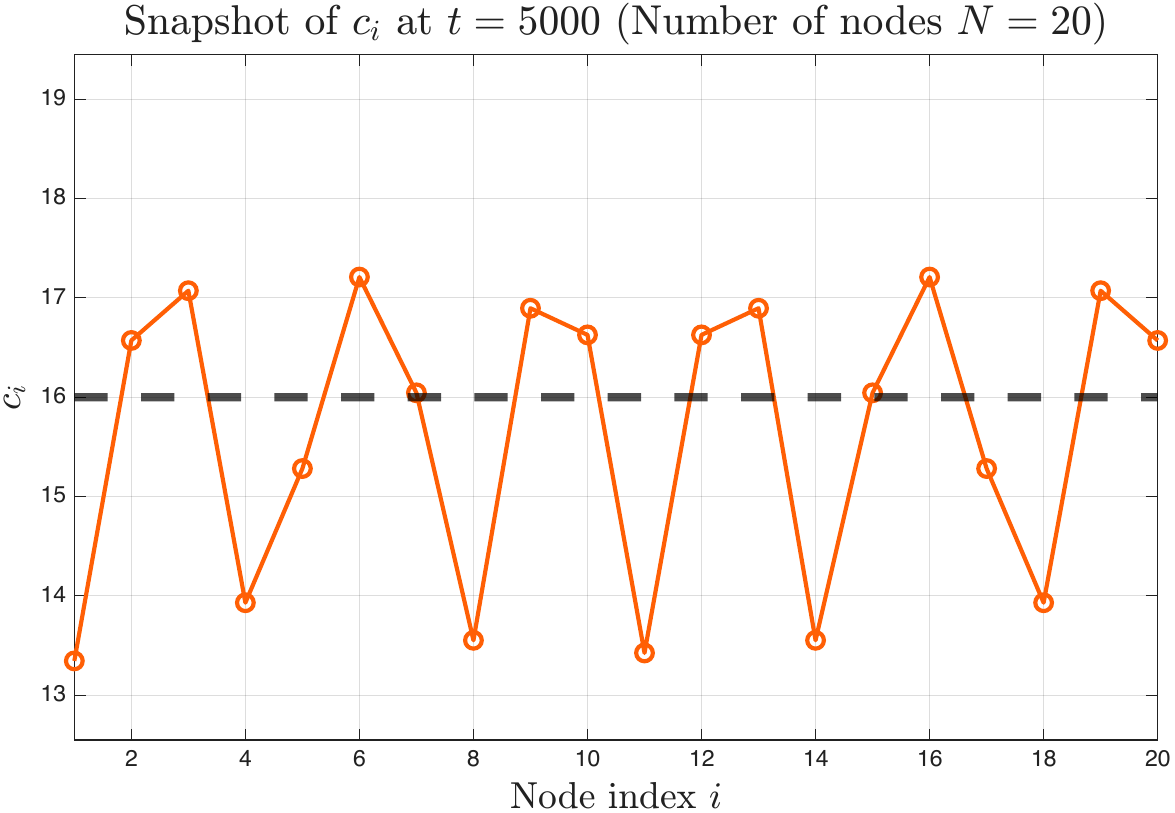}
    \end{subfigure}\vspace{-24pt}
    \begin{subfigure}[b]{\textwidth}
        \centering
        \vspace{2em}
        \caption*{(a) Node number $N=20$}
    \end{subfigure}

    \vspace{1em}

    \begin{subfigure}[b]{0.46\textwidth}
        \includegraphics[width=\linewidth]{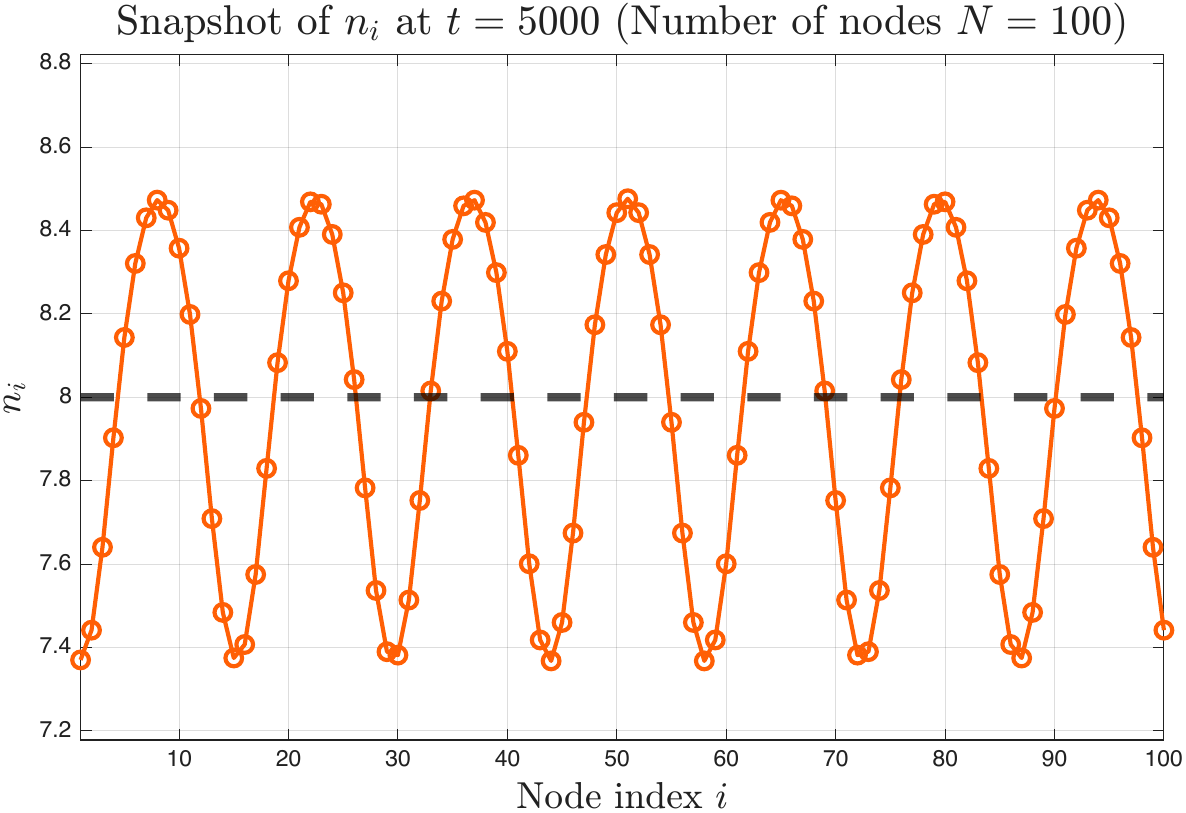}
    \end{subfigure}
    \hfill
    \begin{subfigure}[b]{0.46\textwidth}
        \includegraphics[width=\linewidth]{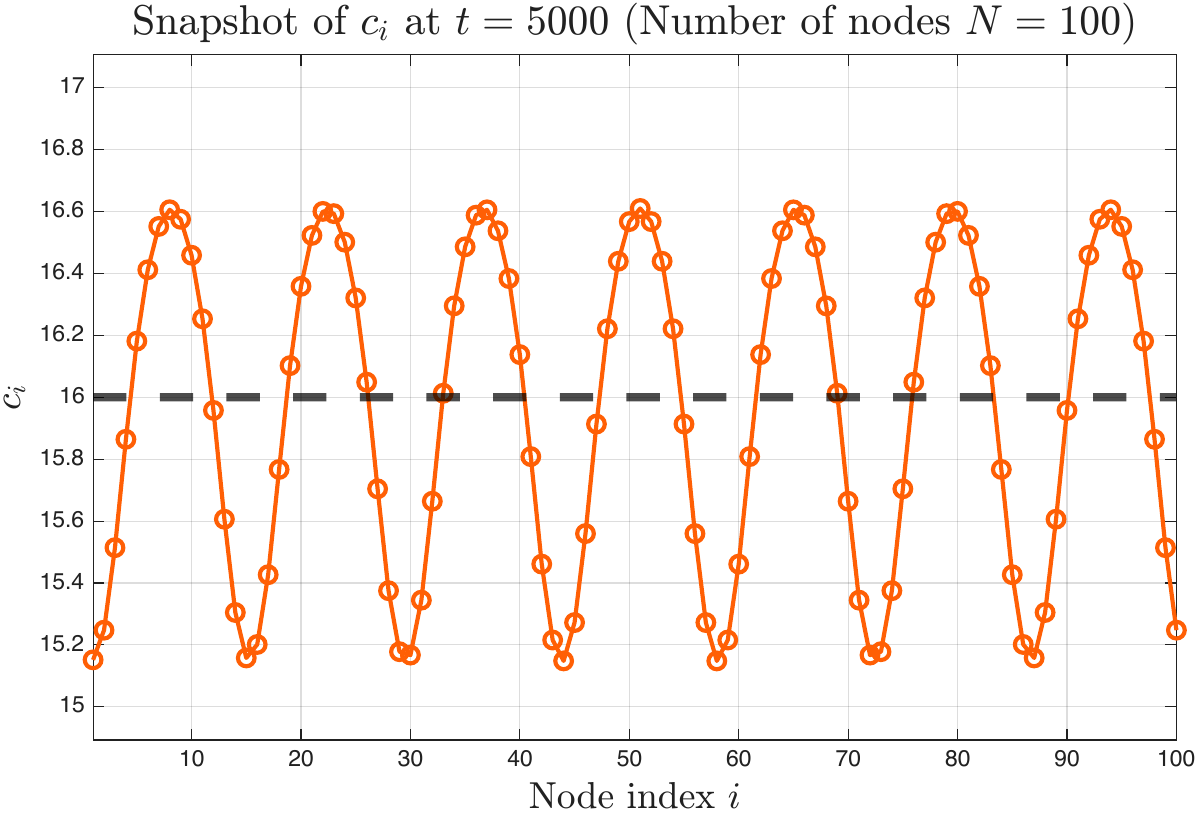}
    \end{subfigure}\vspace{-24pt}
    \begin{subfigure}[b]{\textwidth}
        \centering
        \vspace{2em}
        \caption*{(b) Node number $N=100$}
    \end{subfigure}

    \vspace{1em}

    \begin{subfigure}[b]{0.46\textwidth}
        \includegraphics[width=\linewidth]{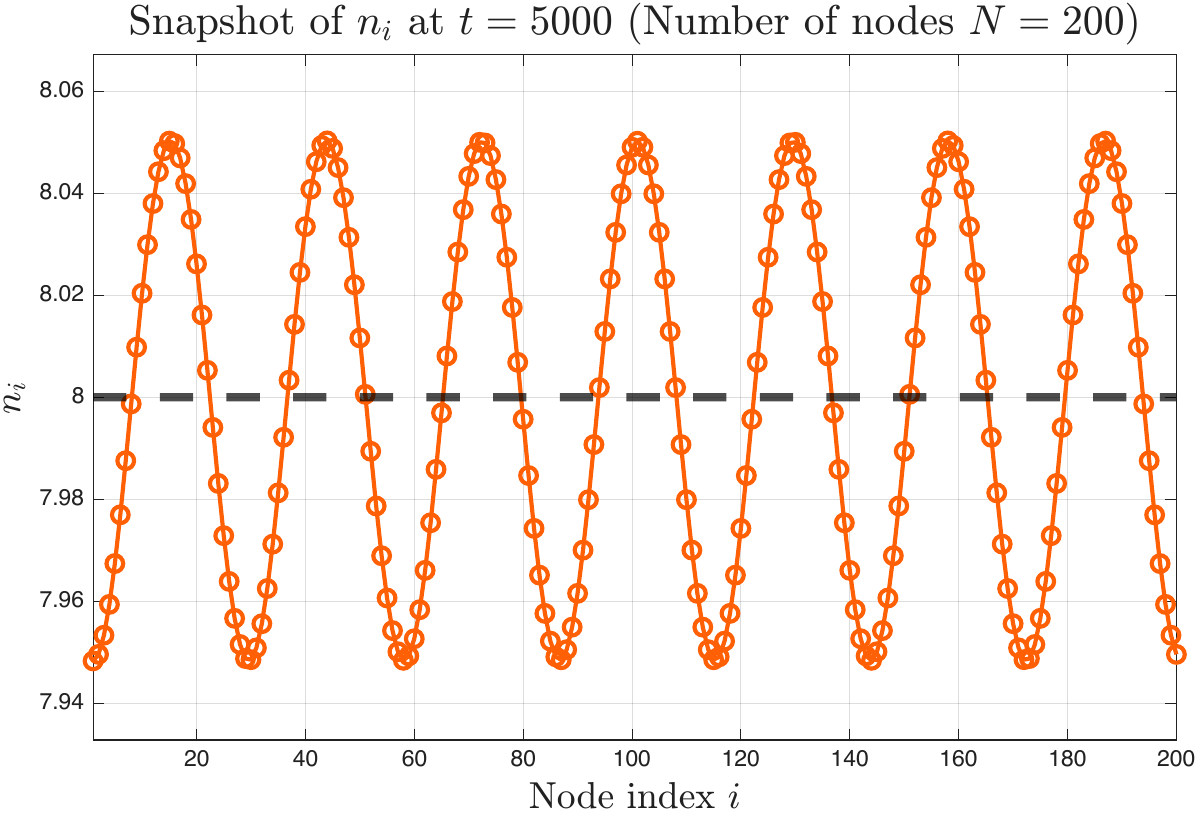}
    \end{subfigure}
    \hfill
    \begin{subfigure}[b]{0.46\textwidth}
        \includegraphics[width=\linewidth]{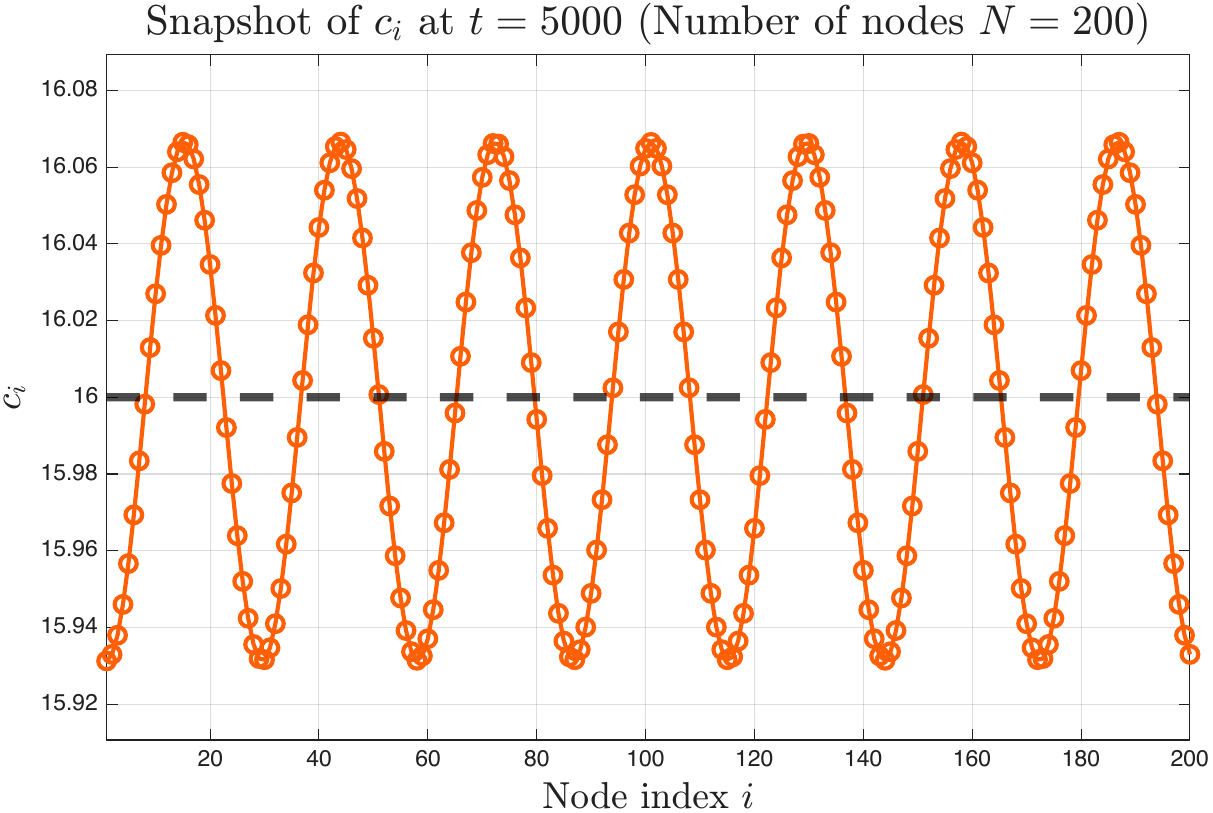}
    \end{subfigure}\vspace{-24pt}
    \begin{subfigure}[b]{\textwidth}
        \centering
        \vspace{2em}
        \caption*{(c) Node number $N=200$}
    \end{subfigure}

    \vspace{0em}
    
    \caption{Snapshots of the numerical solutions for the graph dynamics of Equations~\eqref{eq:KS_graph_n} and~\eqref{eq:KS_graph_c} on metapopulation networks with different numbers of patches at time $t = 5,000$ time units. Each panel describes the case of a uniform band graph with patches connected to the nearest 10 percent of patches to each side for the cases of a domain with $N = 20$ patches (a), $N = 100$ patches (b), and $N = 200$ patches (c). The reaction and movement parameters are set to $\eta_n=0.05$, $\eta_c=0.08$, $\alpha=0.2$, $\beta=0.1$, $r=0.1$, $\bar{n}=8$, and $\bar{c}=16$. The horizontal axis represents node index $i$, and the vertical axis represents the population size $n_i(t)$ (left panels) and the chemical concentration $c_i(t)$ (right panels).  }
    \label{fig:graph-simulations-snapshots}
\end{figure}

In particular, this suggests that, at least close to the onset of pattern-forming instability, the long-time behavior of models on large graph metapopulations may agree with the predicted long-time behavior of the graphon models that describe the infinite-node limit for these families of finite graphs. Notably, the convergence results in Section \ref{sec:graphon-limits} proved different properties of the convergence of solutions to the graphon model for the graph models on large families of deterministic or random graphs over a fixed finite time interval $[0,T]$, so our theoretical results on convergence do not guarantee agreement of the asymptotic behavior for the graphon model and that of the graph models for sufficiently many nodes. Furthermore, the similar forms of the threshold conditions on $\chi$ for the graph and graphon models from Equations~\eqref{eq:ODE-threshold} and~\eqref{eq:threshold-w} suggest that we expect the initial pattern-forming behavior to agree in the cases of graphons and large enough graphs, but the longer-time numerical agreement seen in our simulations suggests a goal of future work to prove long-time agreement of solutions. \cite{bramburger2023pattern} have demonstrated agreement of long-time behavior through convergence of center manifold dynamics between graph and graphon Swift-Hohenberg models, so it may be possible to demonstrate a similar approach for the case of metapopulation Keller--Segel models.

While we have primarily focused on simulations of the graph metapopulation model for a family of ring graphs that serve as a discretized approximation of the uniform band graphon with radius $\rho = 0.1$, we can also perform numerical simulations on other metapopulations with a small number of patches. In Figure~\ref{fig:graph-simulations-limit-cycles}, we show several samples of numerical trajectories $n_i(t)$ and $c_i(t)$ in our graph Keller--Segel model for a metapopulation with $N = 9$ patches and the case of a ring graph with patches connected only to their nearest neighbor on the left and right. Presenting the trajectories of five example nodes, we see that the population size and chemoattractant concentration appears to oscillate in time and may potentially suggest the presence of long-time convergence to a limit-cycle solution. While our prior simulations resulted in the emergence stationary patterns in Figure~\ref{fig:graph-simulations} for the case of metapopulation graphs that converge to the uniform band graphon, we see that different qualitative behaviors appear to be possible when more general connectivity patterns are allowed in the discrete graph model. This indicates that understanding the dynamics of our graph Keller-Segel model on arbitrary finite graphs may yield additional insight into possible behaviors of directed motion on metapopulations that does not appear in our analytically tractable setting of the graphon limit and distance-dependent connectivity, opening up possible directions for future work to explore a more careful exploration of nonlinear dynamics for small metapopulation networks. 

\begin{figure}[htbp]
    \centering

    \begin{subfigure}[b]{0.48\textwidth}
        \includegraphics[width=\linewidth]{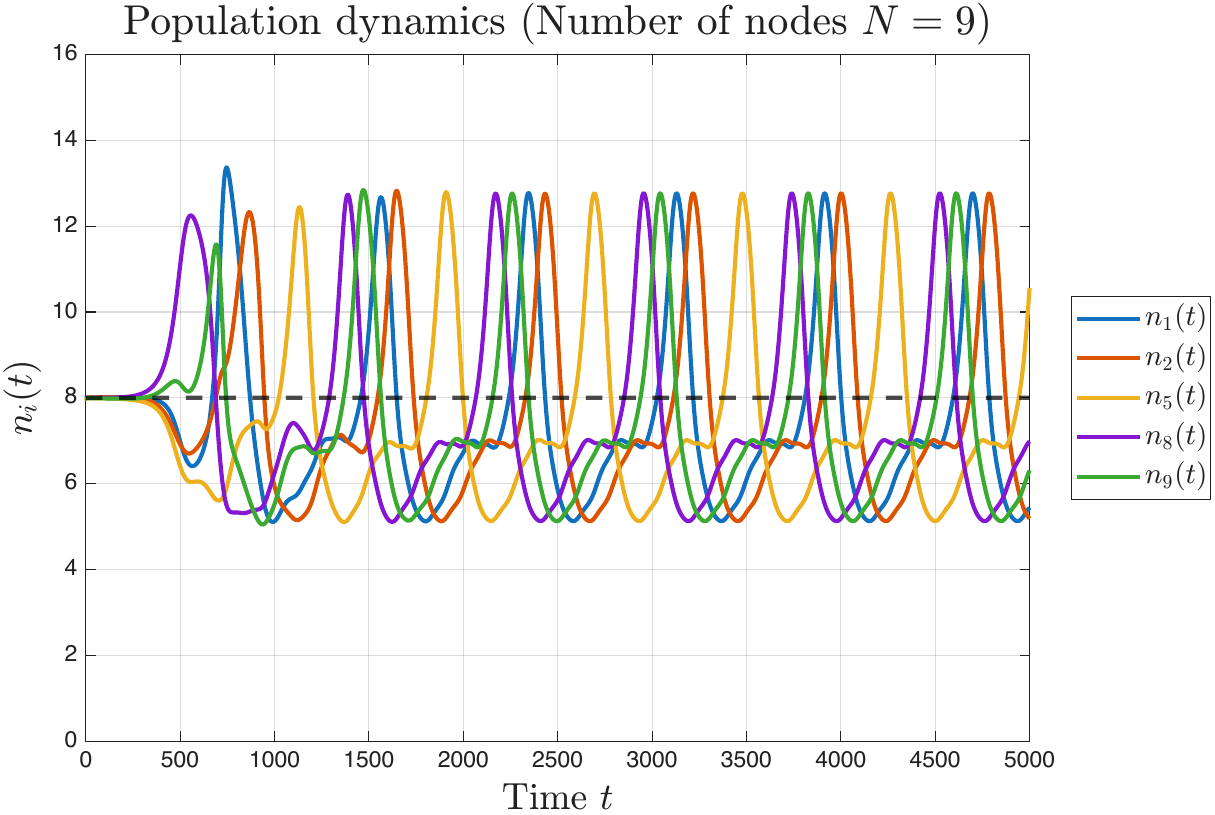}
    \end{subfigure}
    \hfill
    \begin{subfigure}[b]{0.48\textwidth}
        \includegraphics[width=\linewidth]{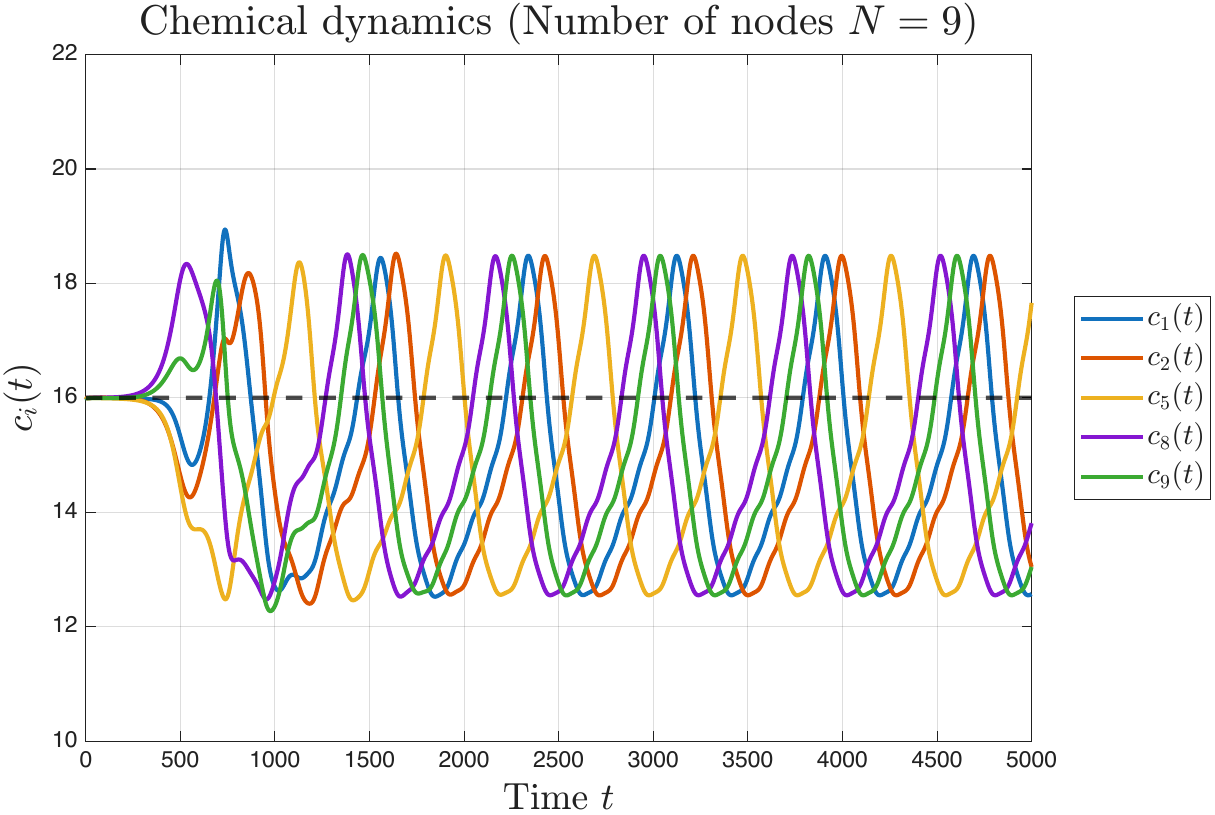}
    \end{subfigure}
    
    \caption{Numerical simulations of the graph dynamics from  equations~\eqref{eq:KS_graph_n} and~\eqref{eq:KS_graph_c} for the case of a ring graph with 9 nodes and patches connected to the nearest neighbor on each side. The curves describe the time-dependent behavior of numerical solutions for the population size $n_i(t)$ (left) and the chemoattractant concentration $c_i(t)$ (right) at patch $i$ for $i \in \{1,2,5,8,9\}$, each plotted as a function of time. The parameters are set to $\eta_n=0.05$, $\eta_c=0.08$, $\alpha=0.2$, $\beta=0.1$, $r=0.1$, $\bar{n}=8$, and $\bar{c}=16$.}
    \label{fig:graph-simulations-limit-cycles}
\end{figure}

\subsection{Simulations for the Graphon Model}\label{sec:simulation-graphon}

This section presents numerical simulations of the graphon Keller--Segel model~\eqref{eq:our-graphon-model} using the normalized band kernel~\eqref{eq:band-kernel}. Unless otherwise stated, we discretize the periodic domain $[0,1]$ with $N=128$ uniform grid points $x_j=j/N$ for $j=0,\ldots,N-1$, and set $\rho=0.1$. The nonlocal terms are computed pseudospectrally via exact Fourier multipliers $\widehat W(k)=\operatorname{sinc}(2\pi\rho k)$ with $\widehat W(0)=1$, avoiding the discretization error inherent in pointwise quadrature $W(x_i,x_j)/N$.

Writing $W(x,y)=J(y-x)$, the chemical diffusion operator $\mathcal L_Wc=J*c-c$ is evaluated using the Fast Fourier Transform (FFT). For the population transport, we first evaluate the non-separable preference term in physical space by defining the matrix $Q$ with entries
\begin{equation*}
    Q_{ij}
    =n_jS(c_i,c_j)-n_iS(c_j,c_i)
    =\frac{n_j e^{\chi c_i}-n_i e^{\chi c_j}}
           {e^{\chi c_i}+e^{\chi c_j}}.
\end{equation*}
The integral at $x_i$ is then approximated pseudospectrally by computing the row-wise discrete convolution. Specifically, for each $i$, we apply the FFT to the $i$-th row of $Q$, multiply by $\widehat W(k)$, apply the inverse FFT, and extract the $i$-th component. We integrate the resulting spatially discretized ODE system using MATLAB's adaptive solver \textit{ode15s}, with relative and absolute tolerances $\mathrm{RelTol}=10^{-8}$ and $\mathrm{AbsTol}=10^{-10}$, respectively.

For our simulations, we use initial conditions based on small random perturbations from the spatially uniform equilibrium, define our initial condition by
\begin{equation}
    \begin{aligned}
        n(0,x)&=\bar{n}+\zeta\mathcal{R}(x),\\
        c(0,x)&=\bar{c}+\zeta\mathcal{R}(x),
    \end{aligned}
\end{equation}
where $\zeta=0.01$. The discrete perturbation $\mathcal{R}$ is constructed to satisfy the reflection symmetry assumed in Section~\ref{sec:graph-model-stability}. For even $N$, we independently sample 
\begin{equation}
    \mathcal{R}(x_j)\sim \operatorname{Unif}[-1,1],\quad j= 0,...,N/2.
\end{equation}
and determine the remaining values by reflection according to
\begin{equation}
    \mathcal R(x_j)=\mathcal R(x_{(-j)\bmod N}),\quad j=N/2+1,...,N-1.
\end{equation}
Each figure in this section shows the result of a single numerical simulation with one realization of the random perturbation, except for simulations reported in Figure~\ref{fig:w-parameters}, where the quantities are averaged over the results obtained from ten simulations with different random initial conditions. To guarantee reproducibility of the simulations, we fix the random seeds in MATLAB so that the same realizations of the random distributions are generated each time. The simulation code is available in the GitHub repository linked in statement on code and data availability.

\begin{Remark}
    The condition $\mathcal{R}(x) = \mathcal{R}(1-x)$ can also be understood as symmetry of $\mathcal{R}$ with respect to $x=1/2$. We impose this symmetry so that the $\sin(2\pi kx)$ mode in the eigenfunction pair $\{\sin(2\pi kx), \cos(2\pi kx)\}$ of $\mathcal{L}_W$, corresponding to the eigenvalue $\widehat{W}(k) - \widehat{W}(0)$, is eliminated, allowing us to compare our numerical simulations with the results of weakly nonlinear analysis based on the case of a single unstable eigenfunction near the onset of pattern-forming instability. 
\end{Remark}

For the parameter values in Figure~\ref{fig:dispersion-relation}, the continuum threshold is $\chi^*\approx0.3984580819$, with a critical wavenumber $k^*=7$. The near-onset simulations in Figure~\ref{fig:simulation-graphon} use $\chi=\chi^*+\varepsilon^2$ with $\varepsilon=0.01$. The other parameters are taken to be $\eta_n=0.05$, $\eta_c=0.08$, $\alpha=0.2$, $\beta=0.1$, $r=0.1$, with the homogeneous equilibrium $(\bar n,\bar c)=(8,16)$. Here $r$, $\bar n$ are parameters appearing in the population reaction term for both the logistic and exponentially modified logistic growth functions. For the exponentially modified logistic reaction function, we additionally take $\theta = 1.1/\bar n$. Other choices of $\chi$, $\varepsilon$, or reaction parameters are specified in the corresponding figure captions.

We first use our numerical simulations to see the emergence of spatial patterns starting from a small perturbation of the uniform state, exploring the time-dependent behavior of the pattern-forming dynamics in Figure \ref{fig:simulation-graphon}. We study the dynamics of the spatial profiles of $n(t,x)$ and $c(t,x)$ and the corresponding total quantities $\int_0^1 n(t,x) dx$ and $\int_0^1 c(t,x) dx$ across the spatial domain, presenting these quantities for both the case of the logistic reaction function (\ref{fig:simulation-graphon}a) and the exponential Allee-logistic reaction function (\ref{fig:simulation-graphon}b). We use the same parameters used to generate the dispersion relation in Figure~\ref{fig:dispersion-relation} and perform our simulation for the chemotactic sensitivity $\chi=\chi^*+\varepsilon^2$ with $\varepsilon=0.01$, showing that the numerical simulations for both reaction functions produce sinusoidal patterns with seven peaks. This number of peaks is consistent with the theoretical prediction of a dominant wavenumber of $k^* = 7$ shown in Figure~\ref{fig:dispersion-relation}, and the correlation of peaks in the population and chemoattractant densities is consistent with agreement in the signs of the coefficients $q_n$ and $q_c$ for the solution of the $\mc{O}(\varepsilon)$ problem in our weakly nonlinear analysis. Furthermore, we seek that, for the case of the logistic reaction function, the total population size and total chemoattractant concentration decrease in time before settling to levels below those achieved in the spatially uniform state, consistent with the predictions of the signs of the coefficients $t_{0,n}$ and $t_{0,c}$ from the $\mc{O}\left(\varepsilon^2\right)$-solution under weakly nonlinear analysis and with the guarantee from Proposition \ref{prop:logistic-total-mass-deficit} of decreased population and chemoattractant mass for patterned equilibrium states. For the case of the exponential Allee-logistic model, we see that the total population and chemoattractant increase in time before equilibrating to a level above that achieved in the uniform state, again agreeing with the predictions for weakly nonlinear analysis for our choice of the parameter $\theta$ characterizing the shape of the exponential modification term.

\begin{figure}[htbp]
    \centering

    %==================== First group ====================

    \begin{subfigure}{0.4\linewidth}
        \centering
        \includegraphics[width=\linewidth]{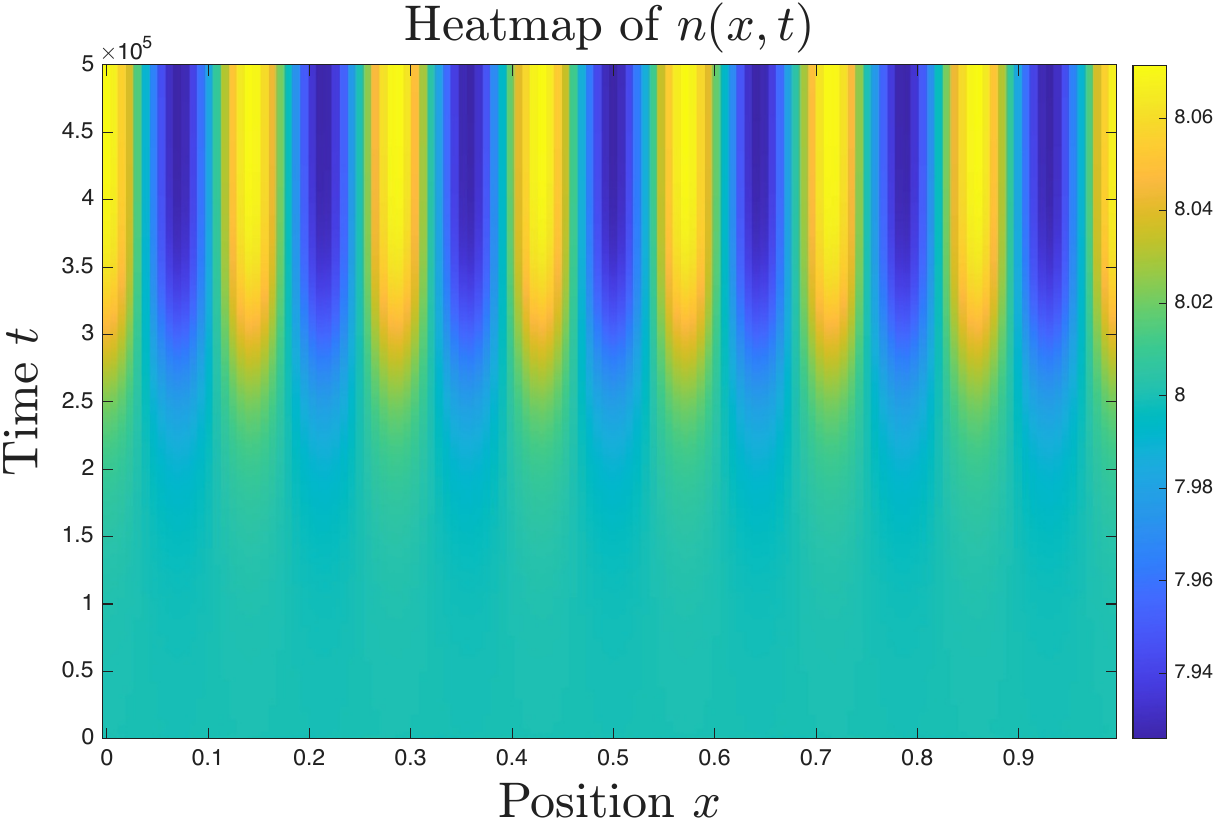}
    \end{subfigure}
    \begin{subfigure}{0.4\linewidth}
        \centering
        \includegraphics[width=\linewidth]{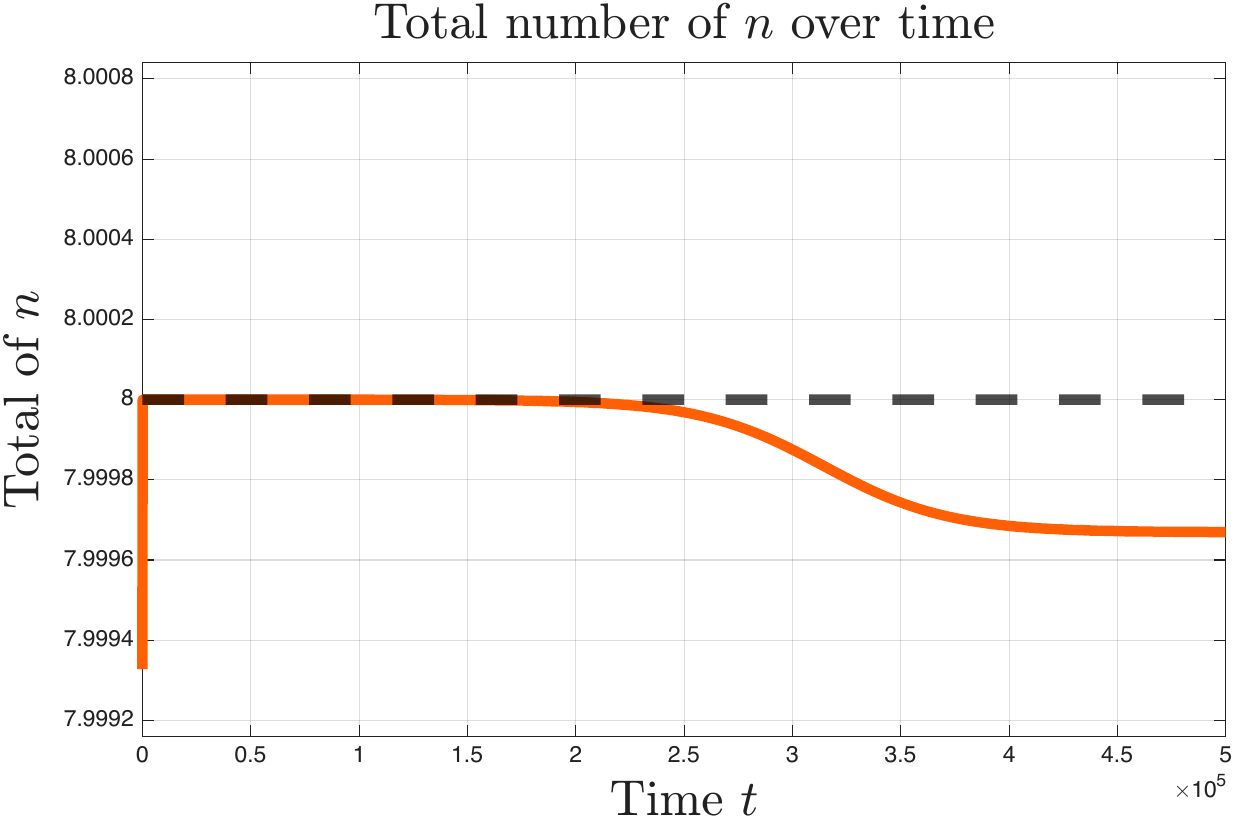}
    \end{subfigure}

    \vspace{0.3em}

    \begin{subfigure}{0.4\linewidth}
        \centering
        \includegraphics[width=\linewidth]{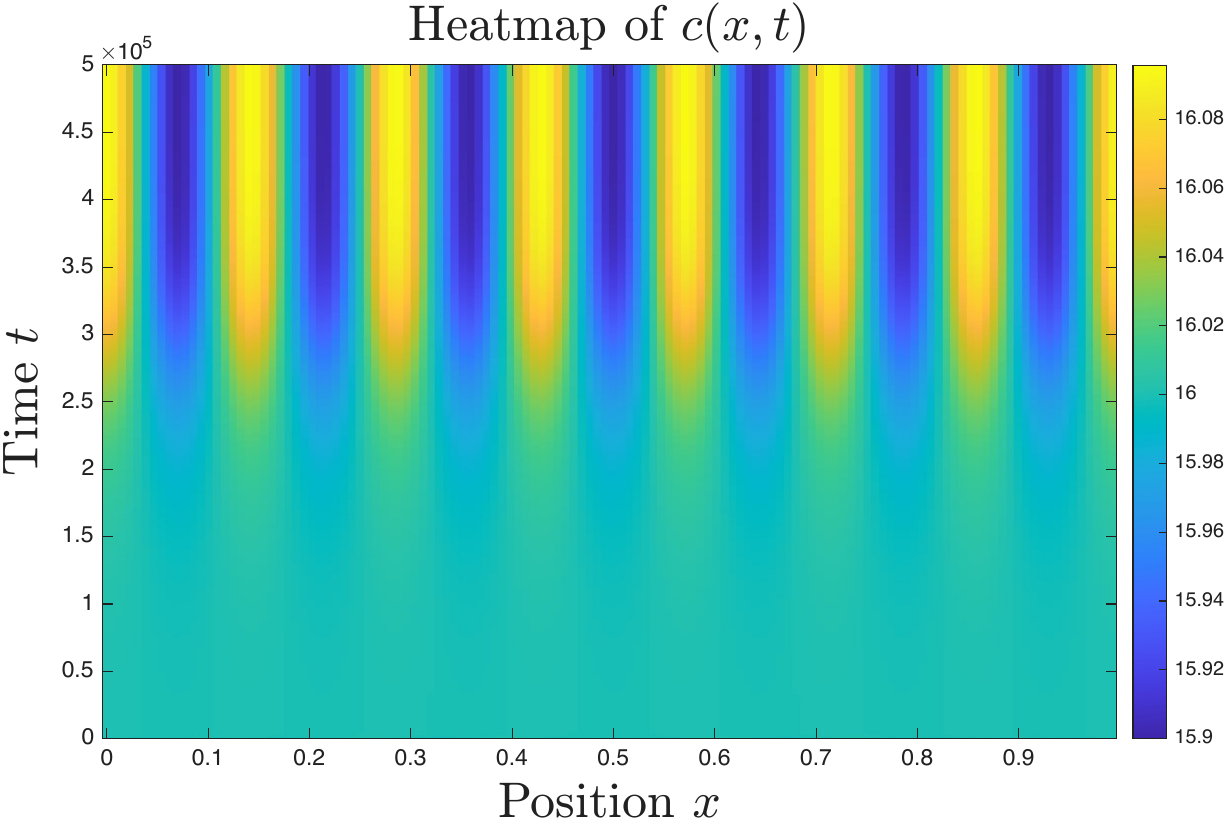}
    \end{subfigure}
    \begin{subfigure}{0.4\linewidth}
        \centering
        \includegraphics[width=\linewidth]{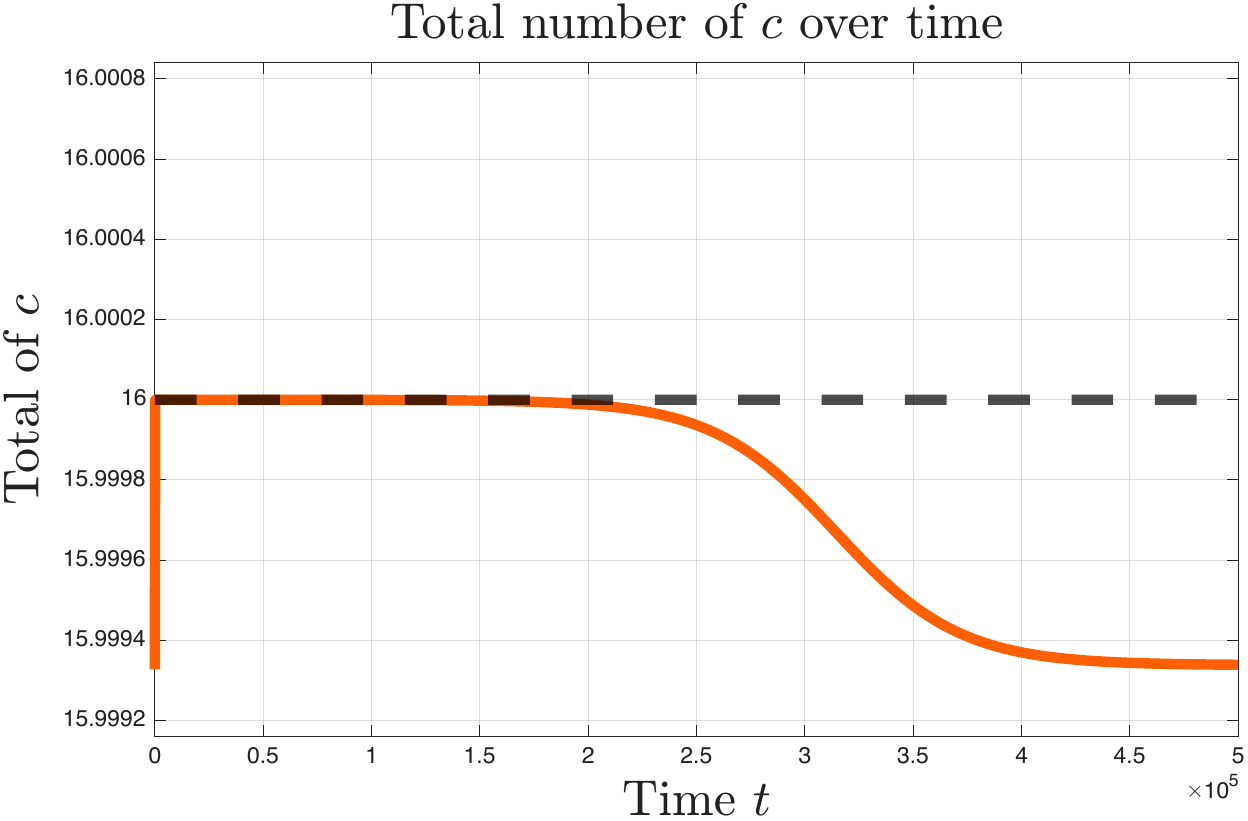}
    \end{subfigure}

    \caption*{(a) Numerical simulation results for the logistic reaction function.}

    \vspace{0.3em}

    %==================== Second group ====================

    \begin{subfigure}{0.4\linewidth}
        \centering
        \includegraphics[width=\linewidth]{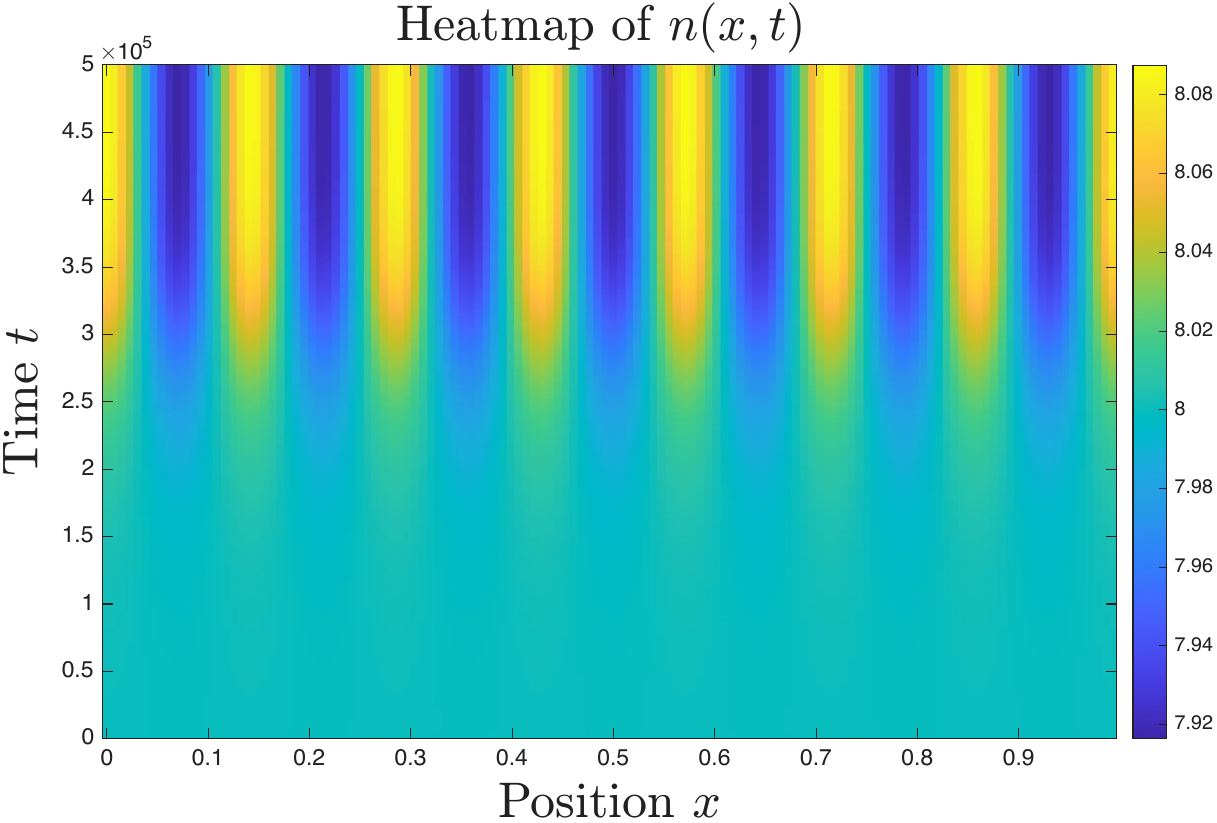}
    \end{subfigure}
    \begin{subfigure}{0.4\linewidth}
        \centering
        \includegraphics[width=\linewidth]{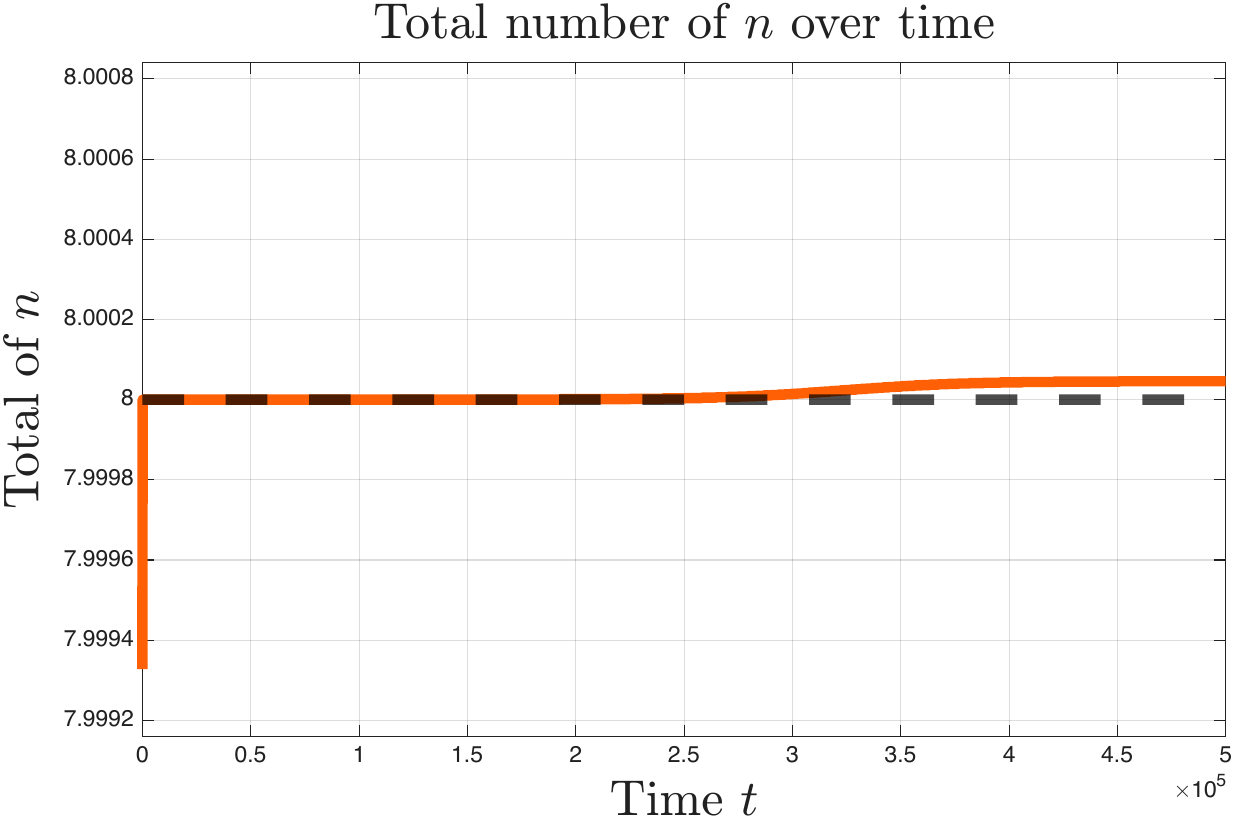}
    \end{subfigure}

    \vspace{0.3em}

    \begin{subfigure}{0.4\linewidth}
        \centering
        \includegraphics[width=\linewidth]{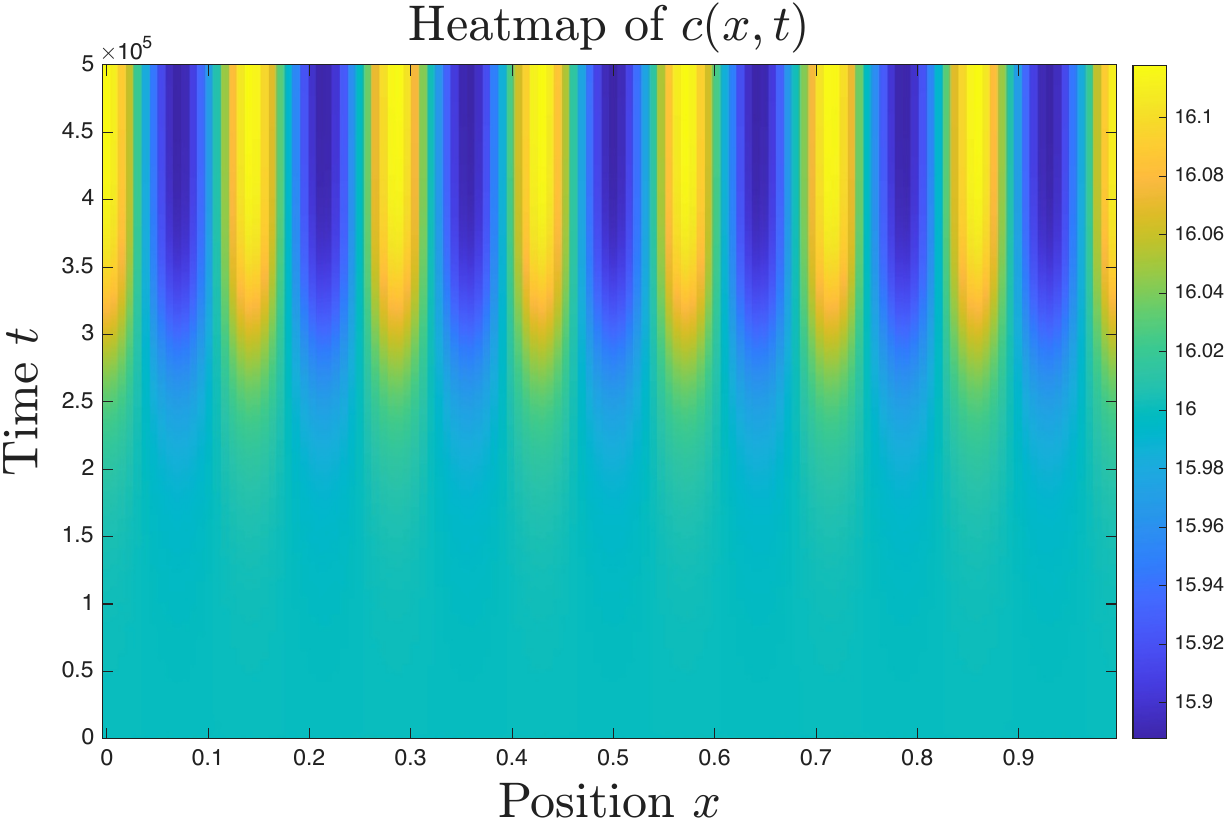}
    \end{subfigure}
    \begin{subfigure}{0.4\linewidth}
        \centering
        \includegraphics[width=\linewidth]{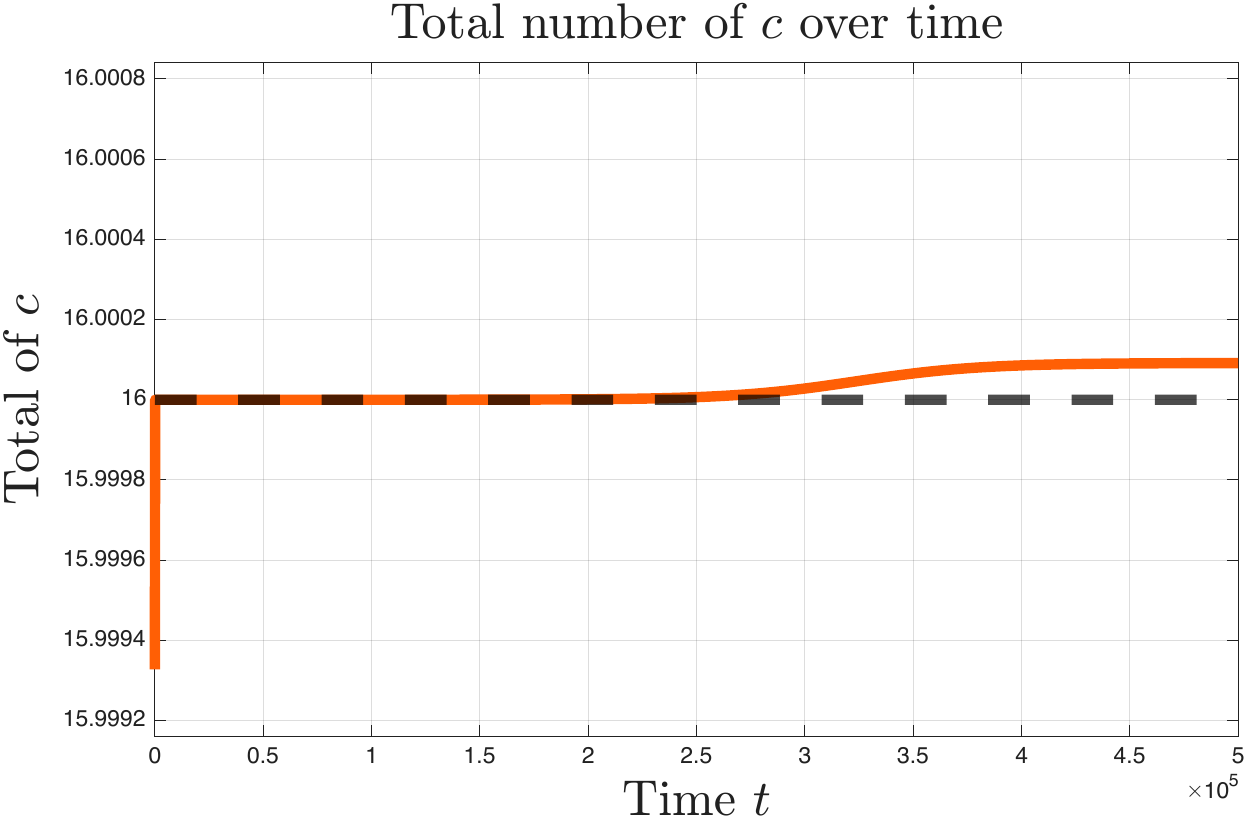}
    \end{subfigure}

    \caption*{(b) Numerical simulation results for the exponential logistic reaction function with $\theta=1.1/\bar{n}$.}

    \caption{
  Numerical simulation results are shown with the PDE parameters taken as $\eta_n=0.05$, $\eta_c=0.08$, $\alpha=0.2$, $\beta=0.1$, $r=0.1$, and $\chi=\chi^*+\varepsilon^2\approx0.3985$. The homogeneous equilibria are $\bar{n}=8$ and $\bar{c}=16$. In the heatmap, the horizontal axis represents the spatial domain $[0,1]$, while the vertical axis denotes time, increasing from bottom to top.
    }

    \label{fig:simulation-graphon}
\end{figure}

We can also use our numerical simulations of the graphon model to validate the analytical approximation of the patterned steady states provided by the weakly nonlinear analysis from Section \ref{sec:graphon-weakly-general}. We present this comparison in Figure \ref{fig:sol-wna}, showing the numerical solutions achieved for the graphon model up to time $t = 500{,}000$ starting from a random initial profile close to the uniform equilibrium state, as well as the asymptotic approximation for the long-time limit of $n(\tau,x)$ and $c(\tau,x)$ from weakly nonlinear analysis from Equation~\eqref{eq:expanded-solution-eq} (taking the appropriate equilibrium amplitude value $a_+$ or $a_-$ based on which choice produces closer qualitative agreement between the resulting solutions). We provide this comparison for values of the chemotactic sensitivity $\chi = \chi^* + \varepsilon^2$ for three values of $\varepsilon$, showing that agreement between the numerical solution and analytical prediction increases as we decrease $\varepsilon$. While we see minor deviations between the numerical and asymptotic solutions for $\varepsilon = 0.1$, we see that the numerical and asymptotic curves have strong agreement for $\varepsilon = 0.05$ and $\varepsilon = 0.01$, suggesting that the asymptotic approximations from weakly nonlinear analysis provide an accurate description of the spatial profiles achieved by the population and chemoattractant for sensitivity $\chi$ close to the onset of our pattern-forming instability.

\begin{figure}[!htbp]
    \centering
    \begin{subfigure}[b]{0.48\textwidth}
       \includegraphics[width=\linewidth]{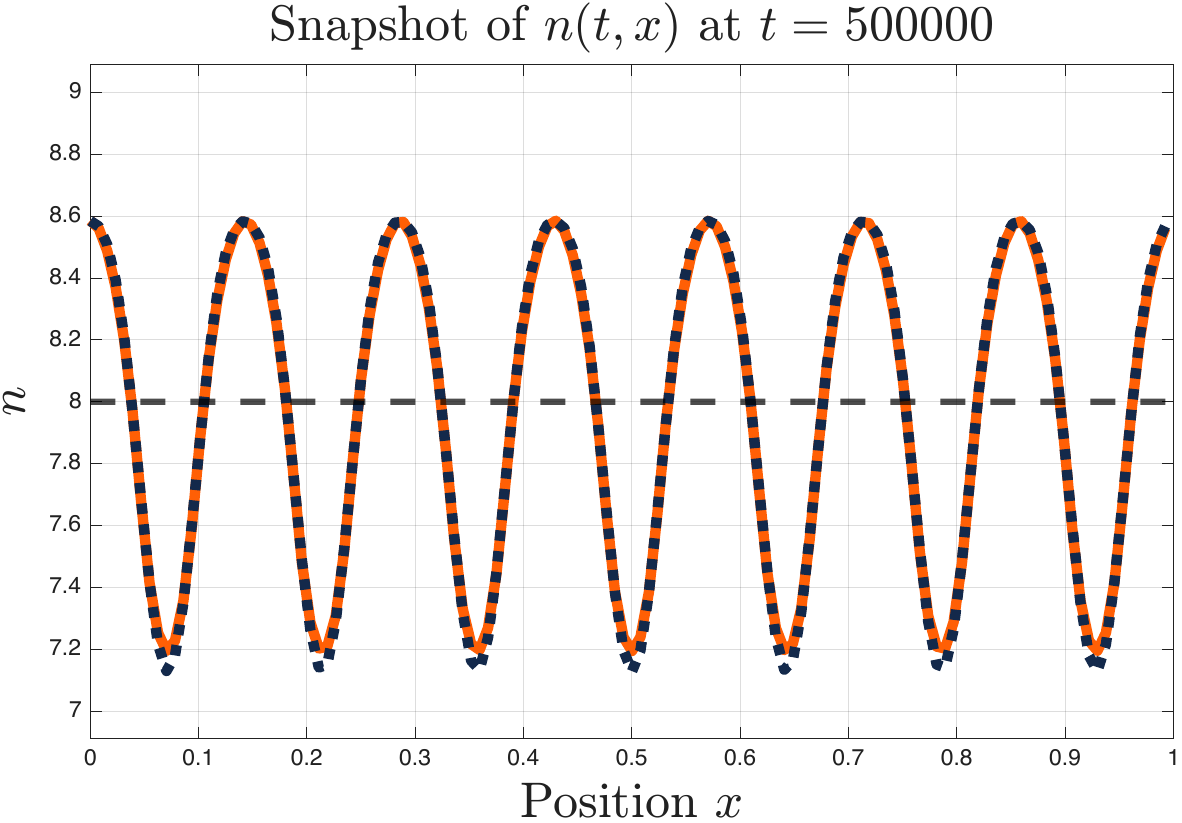}
    \end{subfigure}
    \hfill
    \begin{subfigure}[b]{0.48\textwidth}
        \includegraphics[width=\linewidth]{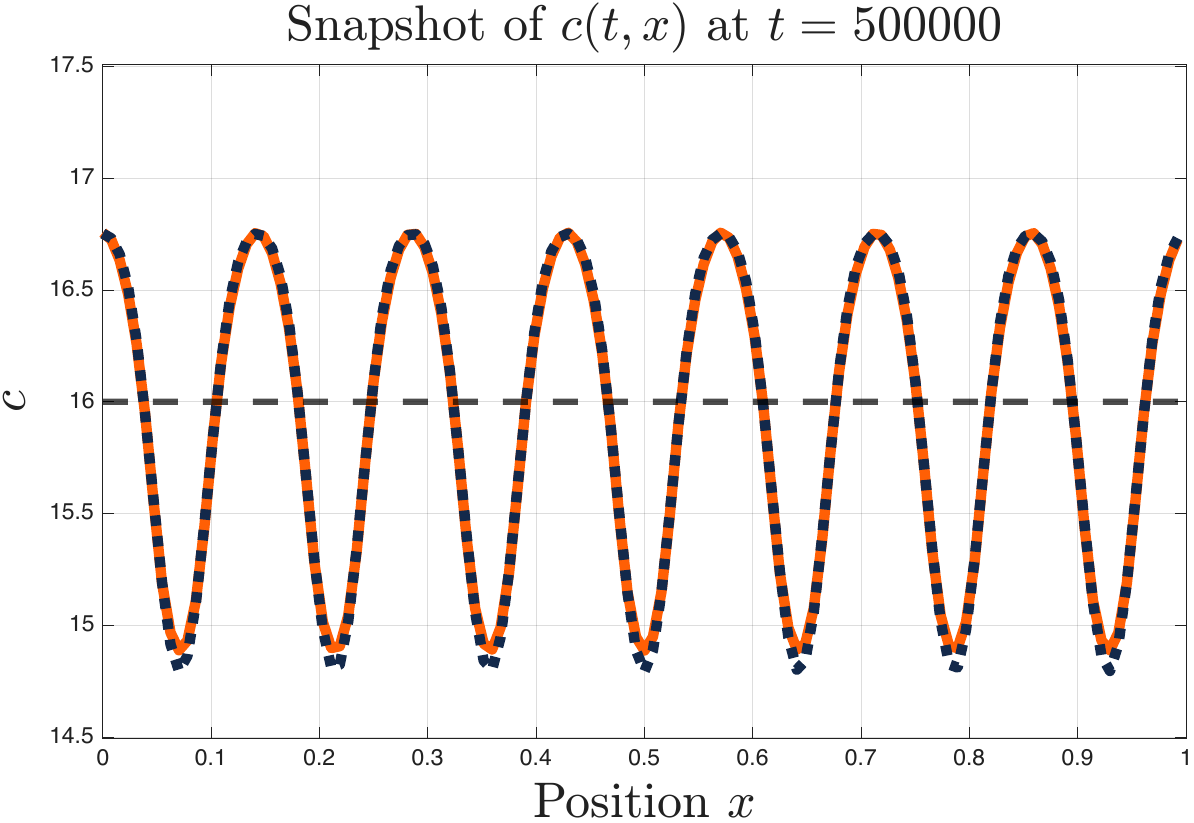}
    \end{subfigure}\vspace{-24pt}
    \begin{subfigure}[b]{\textwidth}
        \centering
        \vspace{2em}
        \caption*{(a) The results when $\varepsilon = 0.1$ at $t=500{,}000$}
    \end{subfigure}

    \vspace{0.5em}

    \begin{subfigure}[b]{0.48\textwidth}
       \includegraphics[width=\linewidth]{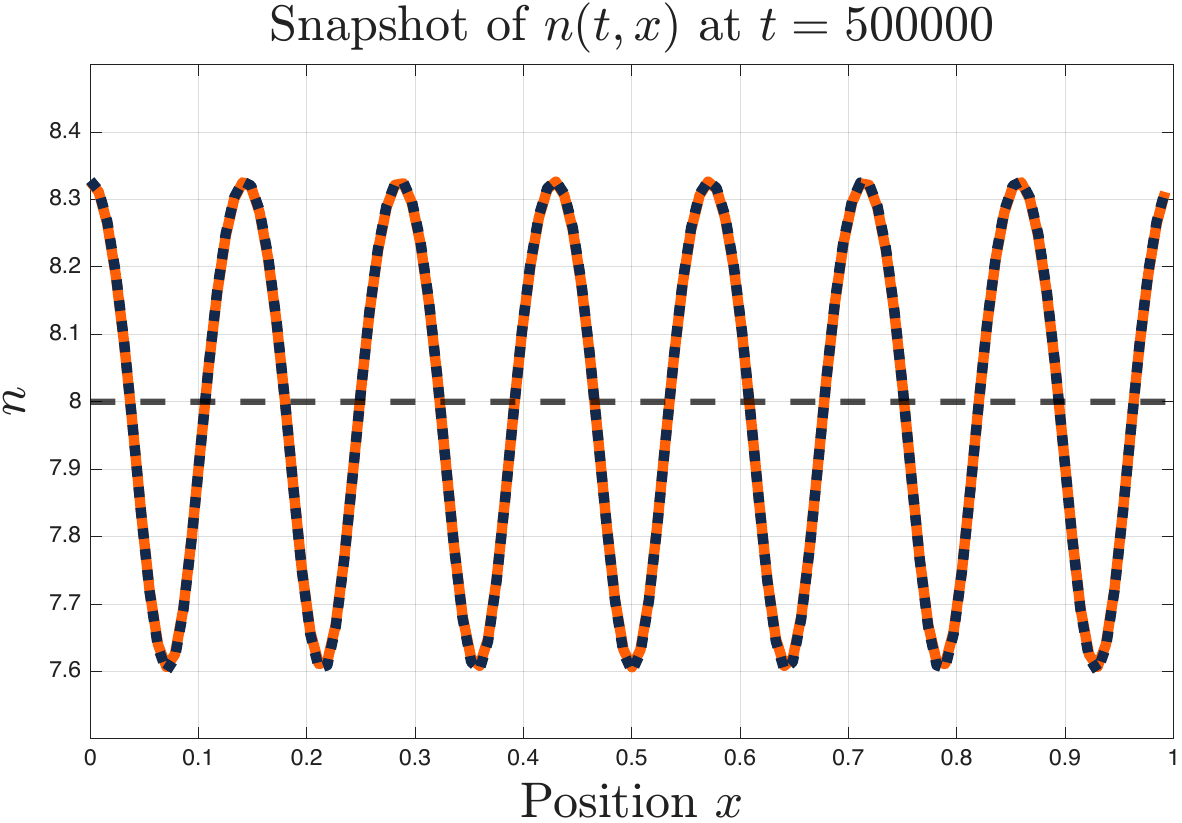}
    \end{subfigure}
    \hfill
    \begin{subfigure}[b]{0.48\textwidth}
        \includegraphics[width=\linewidth]{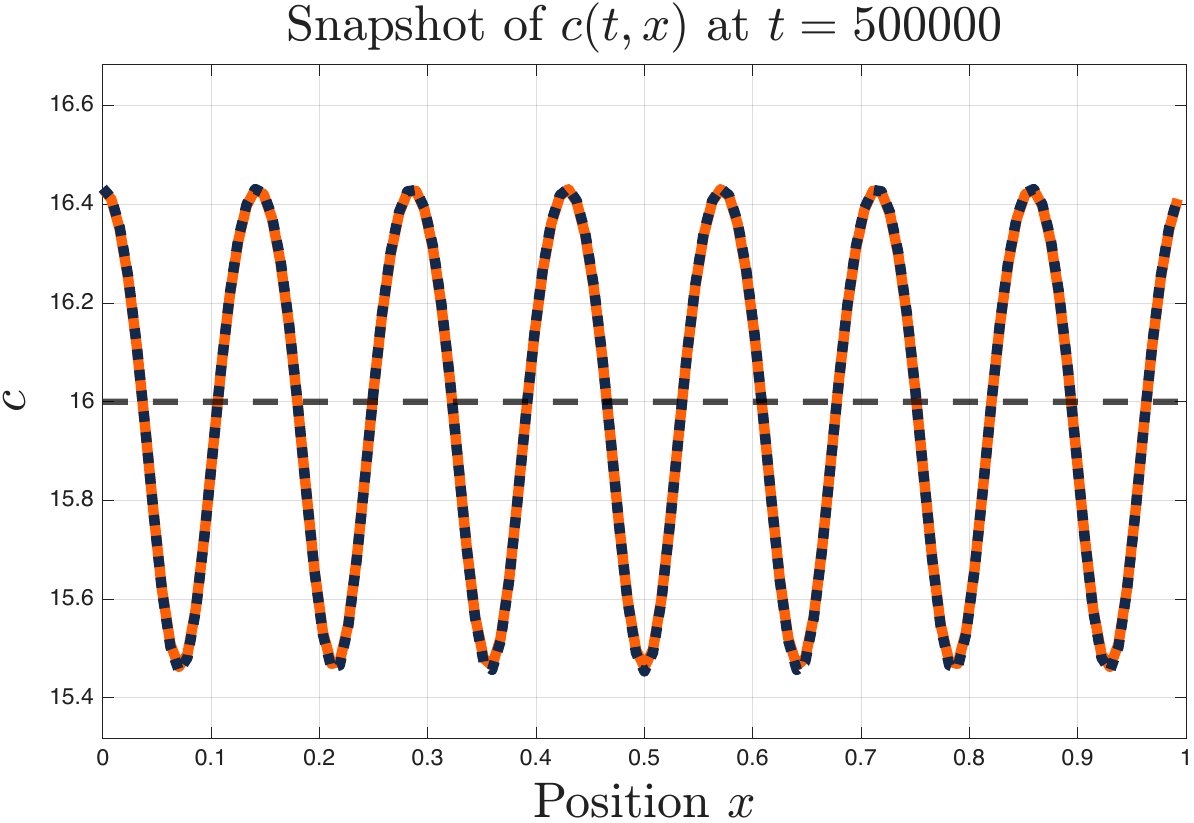}
    \end{subfigure}\vspace{-24pt}
    \begin{subfigure}[b]{\textwidth}
        \centering
        \vspace{2em}
        \caption*{(b) The results when $\varepsilon = 0.05$ at $t=500{,}000$}
    \end{subfigure}

    \vspace{0.5em}

   \begin{subfigure}[b]{0.48\textwidth}
        \includegraphics[width=\linewidth]{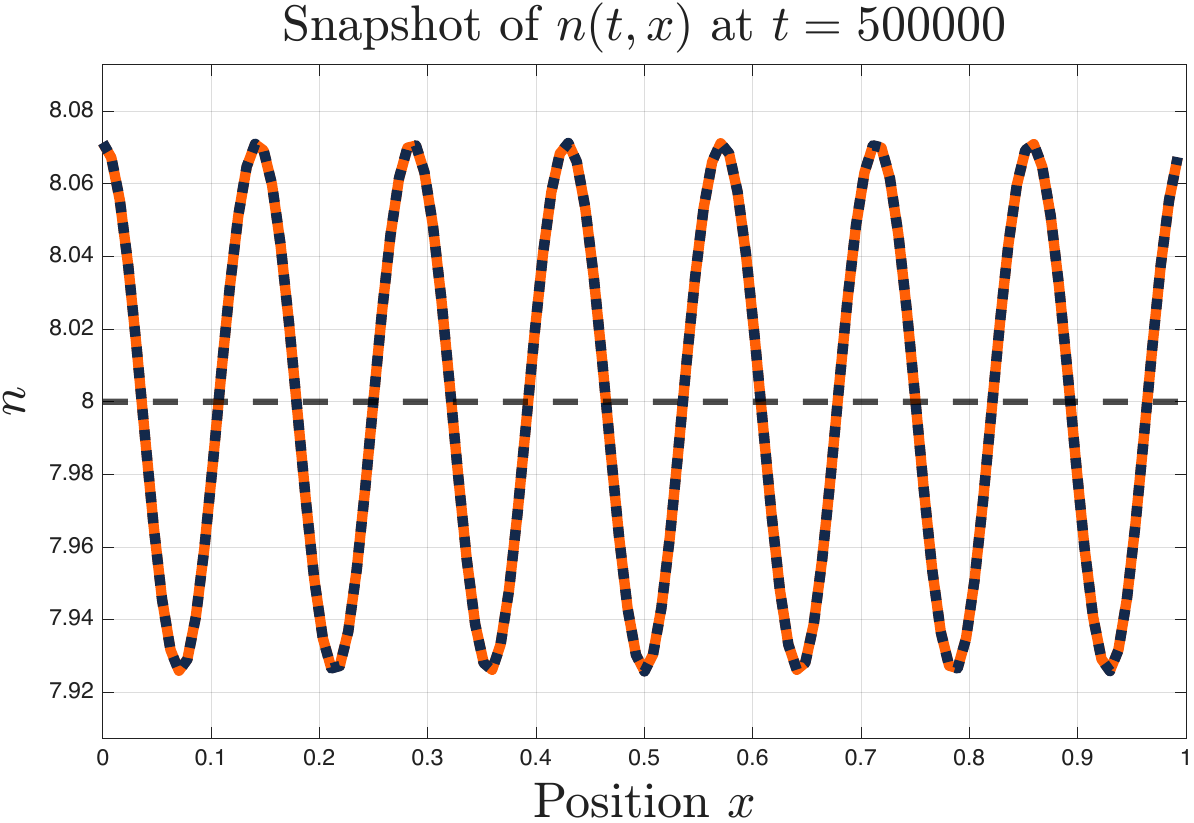}
    \end{subfigure}
    \hfill
    \begin{subfigure}[b]{0.48\textwidth}
        \includegraphics[width=\linewidth]{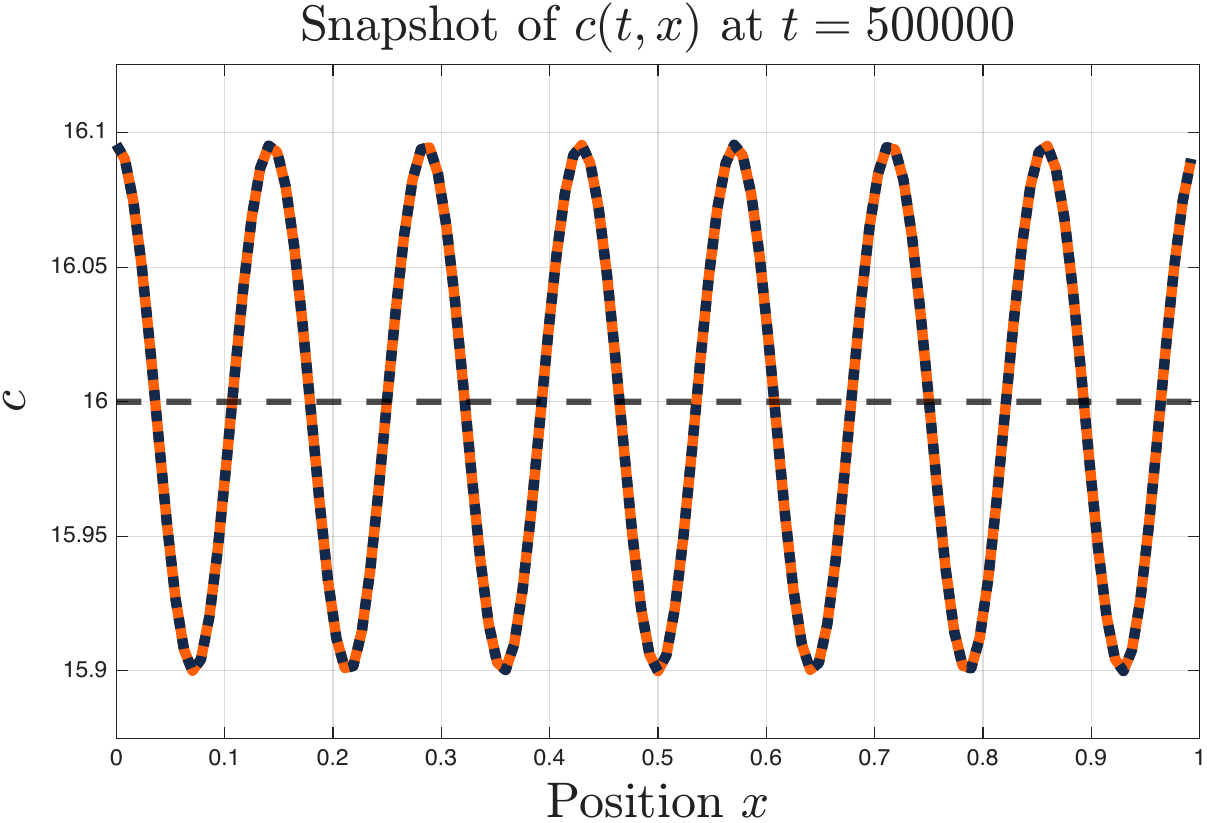}
    \end{subfigure}\vspace{-24pt}
    \begin{subfigure}[b]{\textwidth}
        \centering
        \vspace{2em}
        \caption*{(c) The results when $\varepsilon = 0.01$ at $t=500{,}000$}
    \end{subfigure}

    \vspace{0.5em}

    \caption{Snapshots of numerical solutions for $n(t,x)$ and $c(t,x)$ at time $t=500{,}000$ and the corresponding analytical approximation for the steady-state population and chemical profiles. Here the orange lines represent the numerical solutions of our graphon model at $t = 500{,}000$, while the dark blue dashed lines represent the approximations of the solutions obtained by the weakly nonlinear analysis. We set $\chi=\chi^*(k^*)+\varepsilon^2$ with critical wavenumber $k^*=7$ for different values of $\varepsilon$, while using the same values of parameters and the same random initial condition as in Figure~\ref{fig:simulation-graphon}.}\label{fig:sol-wna}
\end{figure}

We also explore the spatial profiles of emergent patterns in Figure \ref{fig:snapshots}, plotting snapshots of $n(t,x)$ and $c(t,x)$ for the logistic reaction function and uniform band graphon after a large time $t$ for different values of the chemotactic sensitivity parameter $\chi$. In particular, we see that, for the case of $\chi$ slightly above the instability threshold $\chi^*$ (with $\chi=\chi^*+\varepsilon^2$ and $\varepsilon=0.01$), the spatial profiles roughly resemble a sinusoidal pattern with small amplitude with the oscillations in density centered around the equilibrium densities achieved in the spatially uniform state. As chemotactic sensitivity $\chi$ increases past the threshold $\chi^{**}$ upon which infinitely many wavenumbers become unstable, we see that the peaks of the pattern appear to flatten and we see the overall mass of solutions appears to have decreased below the uniform equilibrium value. When the chemotactic instability further increases substantially beyond the threshold $\chi^{**}$, we see that the resulting patterns begin to resemble mesa patterns (also referred to as plateau patterns) featuring a relatively flat region at the high point of population density or chemoattractant concentration. Such mesa patterns are often seen in reaction-diffusion models \citep{kolokolnikov2007self,mckay2012stability} and have also been demonstrated in chemotaxis models that feature globally existing solutions \citep{hillen2009user}. As PDE models of chemotaxis and related phenomena have often been shown to produce spike-like patterns in the limit of strong chemotactic sensitivity \citep{potapov2005metastability,hillen2007classification,kolokolnikov2014basic,buttenschoen2020cops,kong2022stability,kong2023existence,kong2025boundary}, this qualitative behavior appears to be a notable feature of the nonlocal form of our model of chemotactic motion and its formulation in terms of biased random walks on graphon metapopulations. 

\begin{figure}[htbp]
    \centering
    
    \begin{subfigure}{0.48\linewidth}
        \centering
        \includegraphics[width=\linewidth]{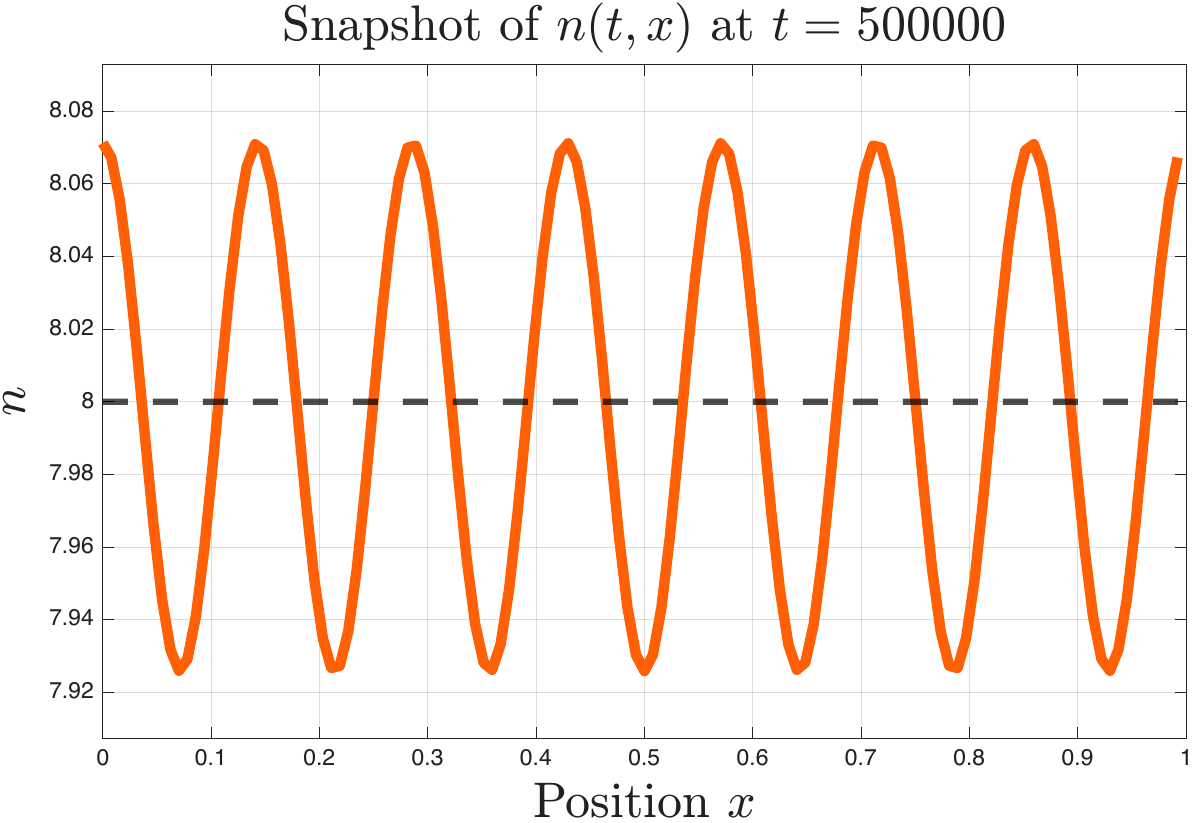}
        \caption{Snapshot of $n(t,x)$ with $\chi=\chi^{*}+\varepsilon^2$}
    \end{subfigure}
    \begin{subfigure}{0.48\linewidth}
        \centering
        \includegraphics[width=\linewidth]{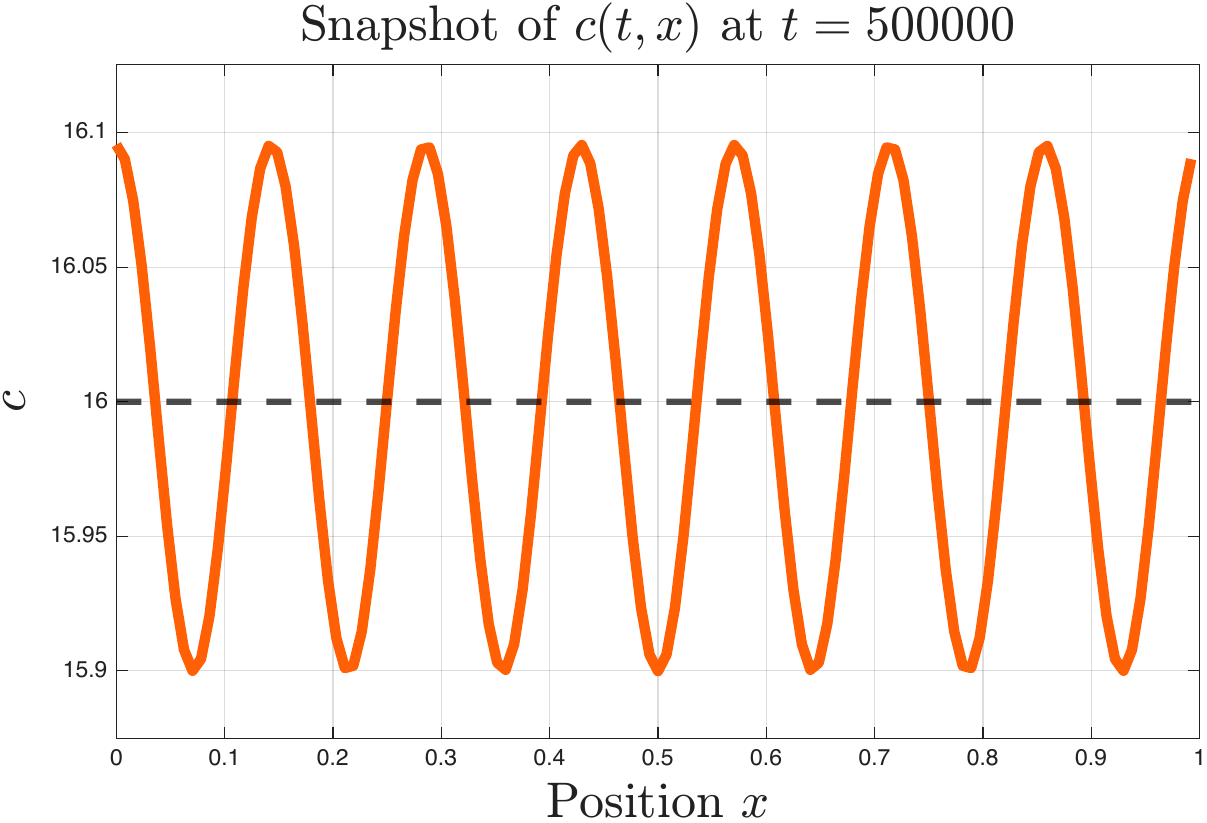}
        \caption{Snapshot of $c(t,x)$ with $\chi=\chi^{*}+\varepsilon^2$}
    \end{subfigure}

    \vspace{0.5em}

    \begin{subfigure}{0.48\linewidth}
        \centering
        \includegraphics[width=\linewidth]{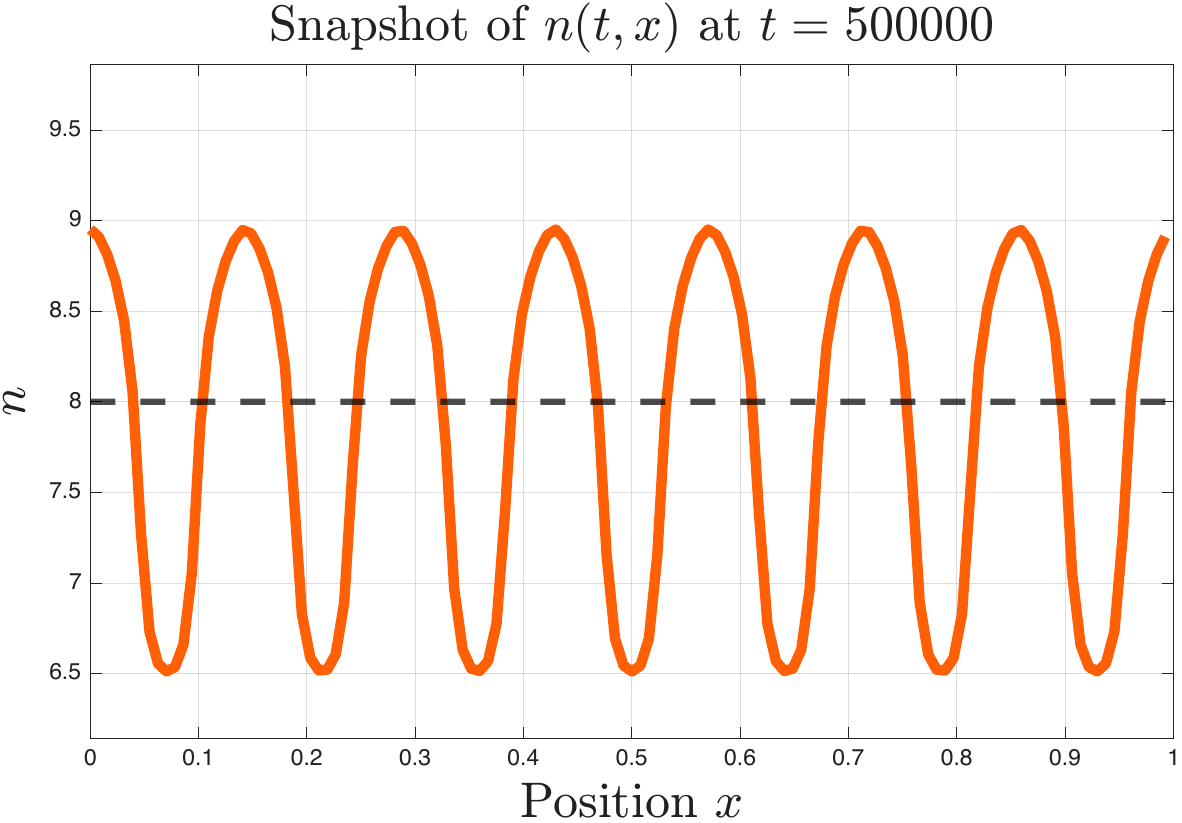}
        \caption{Snapshot of $n(t,x)$ with $\chi=\chi^{**}+\varepsilon^2$}
    \end{subfigure}
    \begin{subfigure}{0.48\linewidth}
        \centering
        \includegraphics[width=\linewidth]{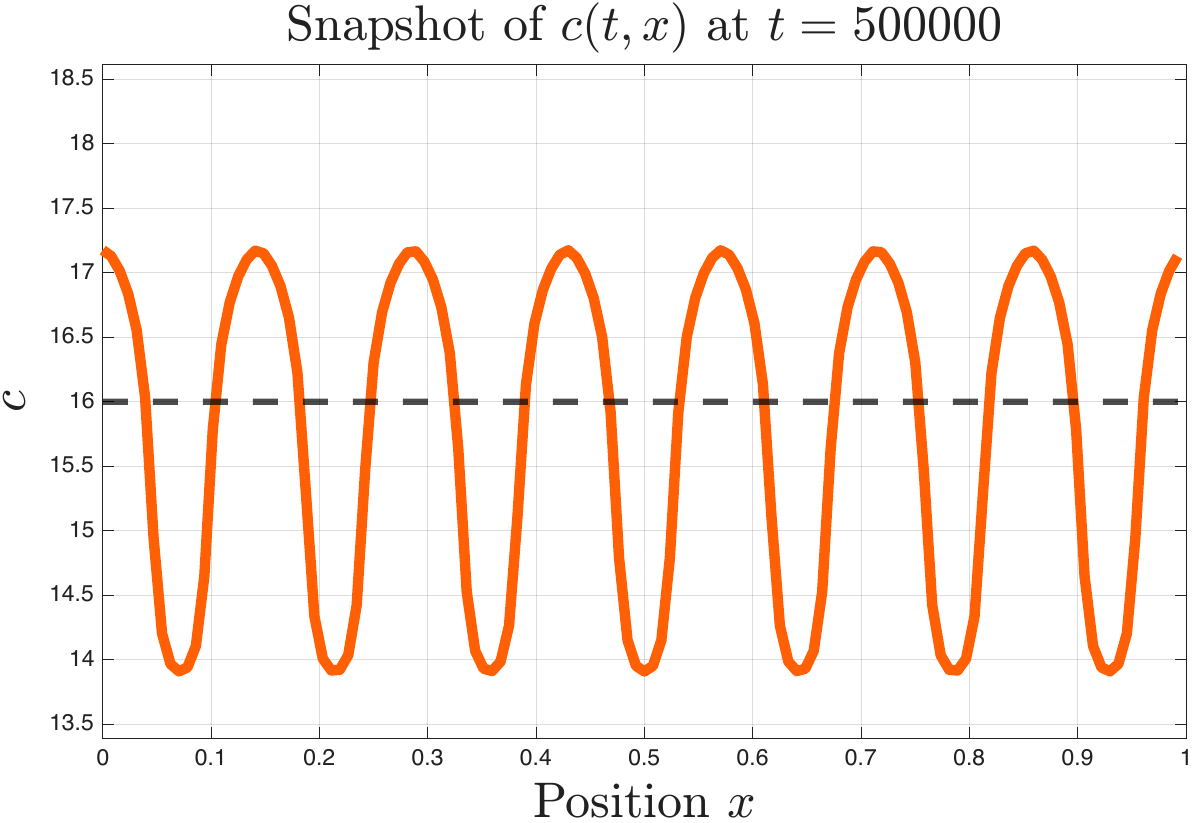}
        \caption{Snapshot of $c(t,x)$ with $\chi=\chi^{**}+\varepsilon^2$}
    \end{subfigure}

    \vspace{0.5em}

    \begin{subfigure}{0.48\linewidth}
        \centering
        \includegraphics[width=\linewidth]{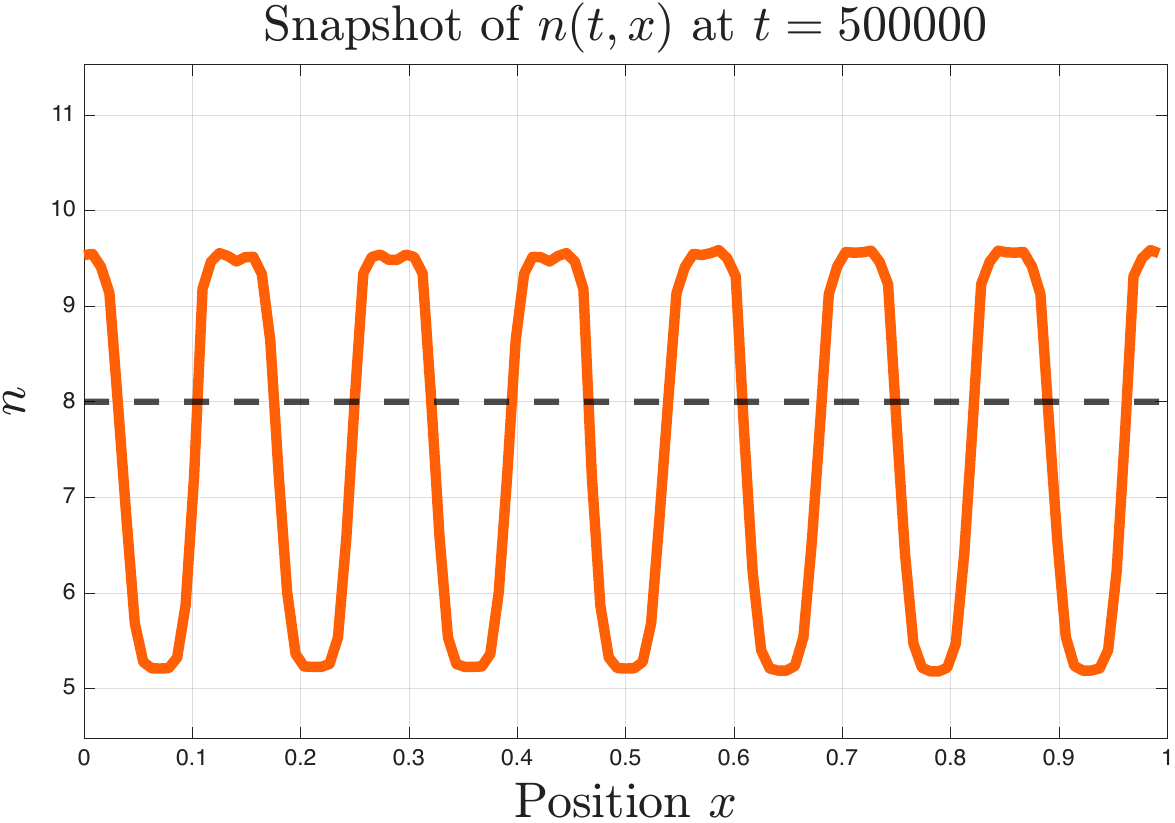}
        \caption{Snapshot of $n(t,x)$ with $\chi=1.8\chi^{**}$}
    \end{subfigure}
    \begin{subfigure}{0.48\linewidth}
        \centering
        \includegraphics[width=\linewidth]{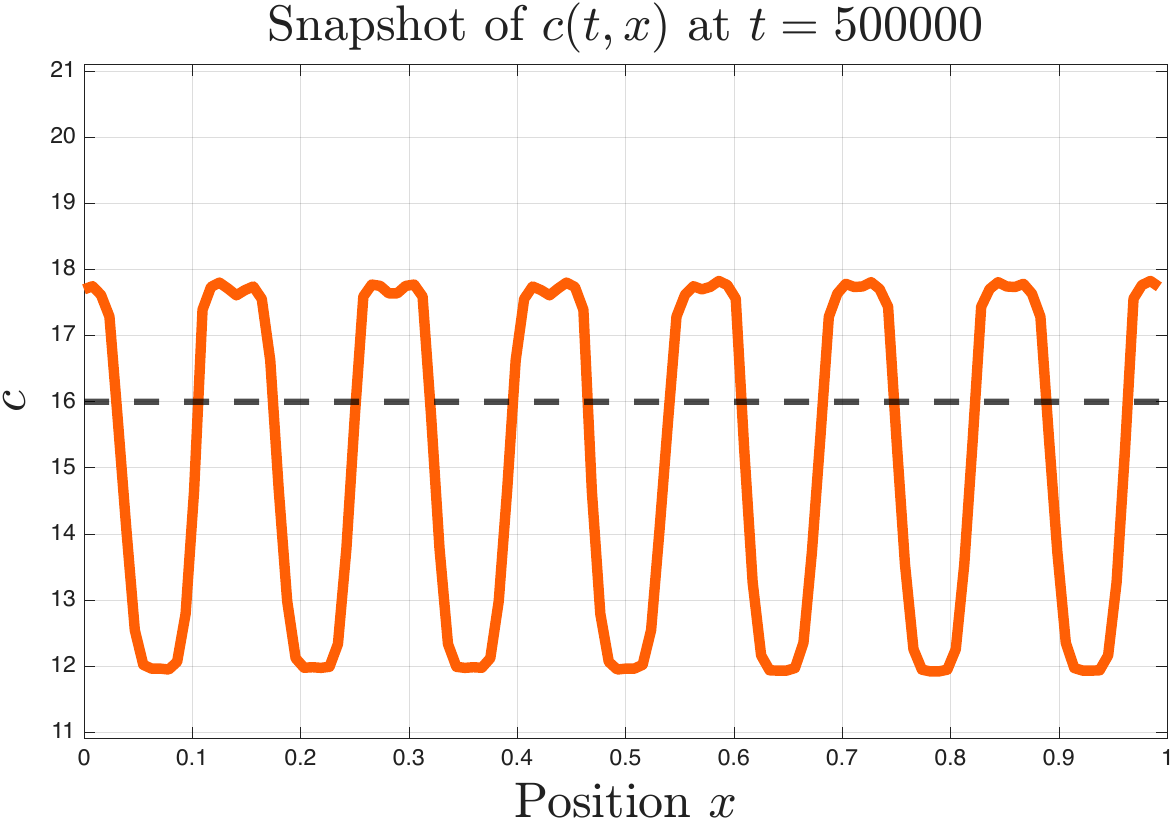}
        \caption{Snapshot of $c(t,x)$ with $\chi=1.8\chi^{**}$}
    \end{subfigure}

    \caption{Numerical simulation results are shown for the PDE parameters $\eta_n=0.05$, $\eta_c=0.08$, $\alpha=0.2$, $\beta=0.1$, $r=0.1$. In panels (a) and (b), we take $\chi=\chi^*+\varepsilon^2\approx0.3985$, in panels (c) and (d), we take $\chi=\chi^{**}+\varepsilon^2\approx0.4376$, and in panels (e) and (f), we take $\chi=1.8\chi^{**}\approx 0.7875$. The homogeneous equilibria are $\bar{n}=8$ and $\bar{c}=16$. In each figure, the horizontal axis represents the spatial domain $[0,1]$, the horizontal black dashed lines represent the homogeneous equilibria $\bar{n}$ and $\bar{c}$, and the orange solid lines represent snapshots of the solutions $n(500{,}000,x)$ and $c(500{,}000,x)$. }
    \label{fig:snapshots}
\end{figure}

Finally, we use numerical simulations to explore how the long-time values of the total population side $\int_0^1 n(t,x) dx$ and the total chemical concentration $\int_0^1 c(t,x) dx$ across the spatial domain change with increasing chemotactic sensitivity $\chi$. In Figure \ref{fig:w-parameters}, we show plot the values of these quantities achieved at time $t = 50,000$ as a function of $\chi$ for the case of the logistic reaction function (Figure \ref{fig:w-parameters}a) and for the case of our exponential Allee-logistic model for a sufficiently large value of the parameter $\theta$ governing the sensitivity of the Allee effect (Figure \ref{fig:w-parameters}b). For each value of $\chi$ considered, the plot depends on the average values of the integrals for over a set of 10 of simulations, with initial conditions drawn according to the same set of 10 random seeds. We see that both the total population size and chemical concentration achieve the values of uniform equilibrium state when $\chi \leq \chi^*$, and then these quantities move away from the uniform equilibrium value as $\chi$ increases past $\chi^*$. Above $\chi^*$, we see that the total masses of $n(t,x)$ and $c(t,x)$ decrease with $\chi$ for the logistic reaction function and increase with $\chi$ for the exponential Allee-logistic model for our chosen value of the parameter $\theta$. These observations on the change of total mass agree with the predictions from weakly nonlinear analysis based on the signs of the quantities $t_{0,n}$ and $t_{0,c}$ in~\eqref{eq:t_0-expressions} for these two choices of reaction functions, but the qualitative prediction actually holds for a range of $\chi$ much larger than the range over which the amplitude equation and weakly nonlinear expansions are expected to produce good agreement with numerical solutions. Furthermore, the continued decrease in population size and chemical concentration for larger $\chi$ for the logistic case and the corresponding gain in the exponential Allee-logistic model highlights the role that increased chemotactic bias plays in changing the ecological conditions across the spatial domain.

\begin{figure}[!ht]
    \centering

    \begin{subfigure}[b]{0.48\textwidth}
        \includegraphics[width=\linewidth]{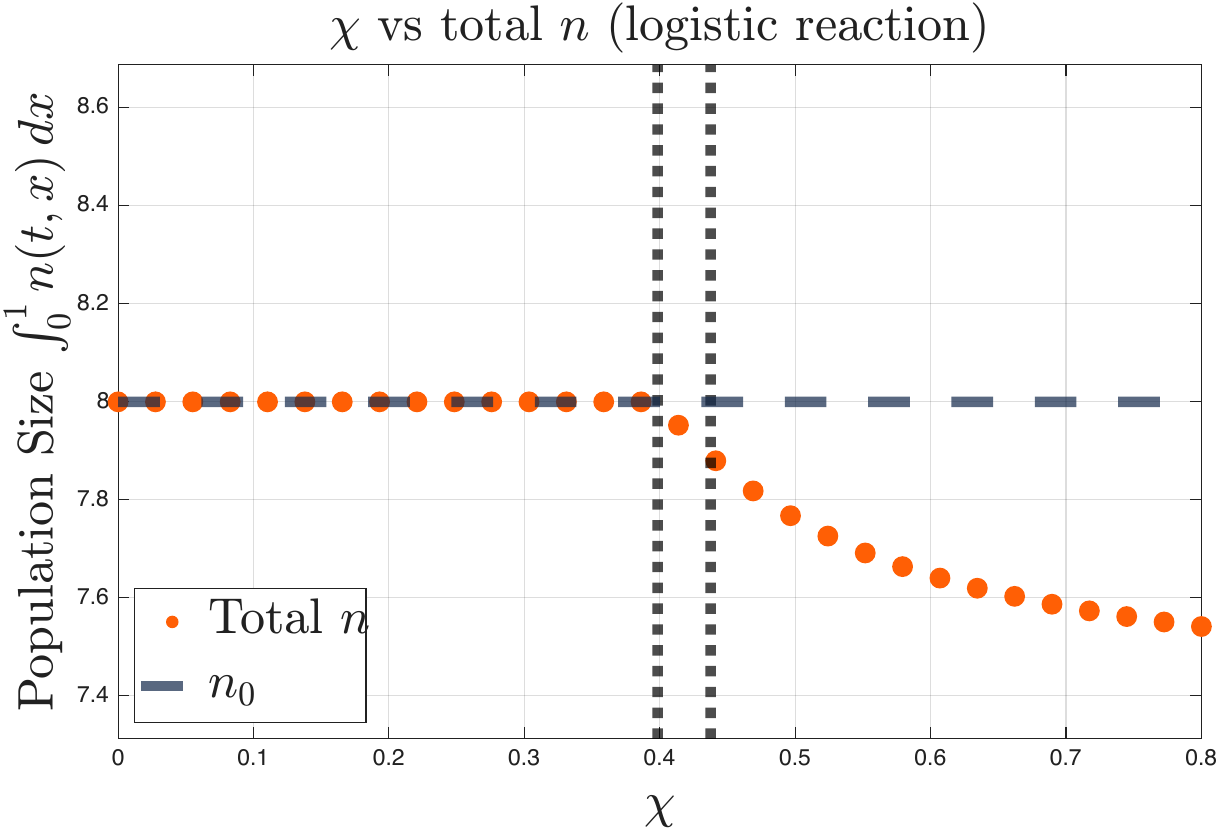}
    \end{subfigure}
    \hfill
    \begin{subfigure}[b]{0.48\textwidth}
        \includegraphics[width=\linewidth]{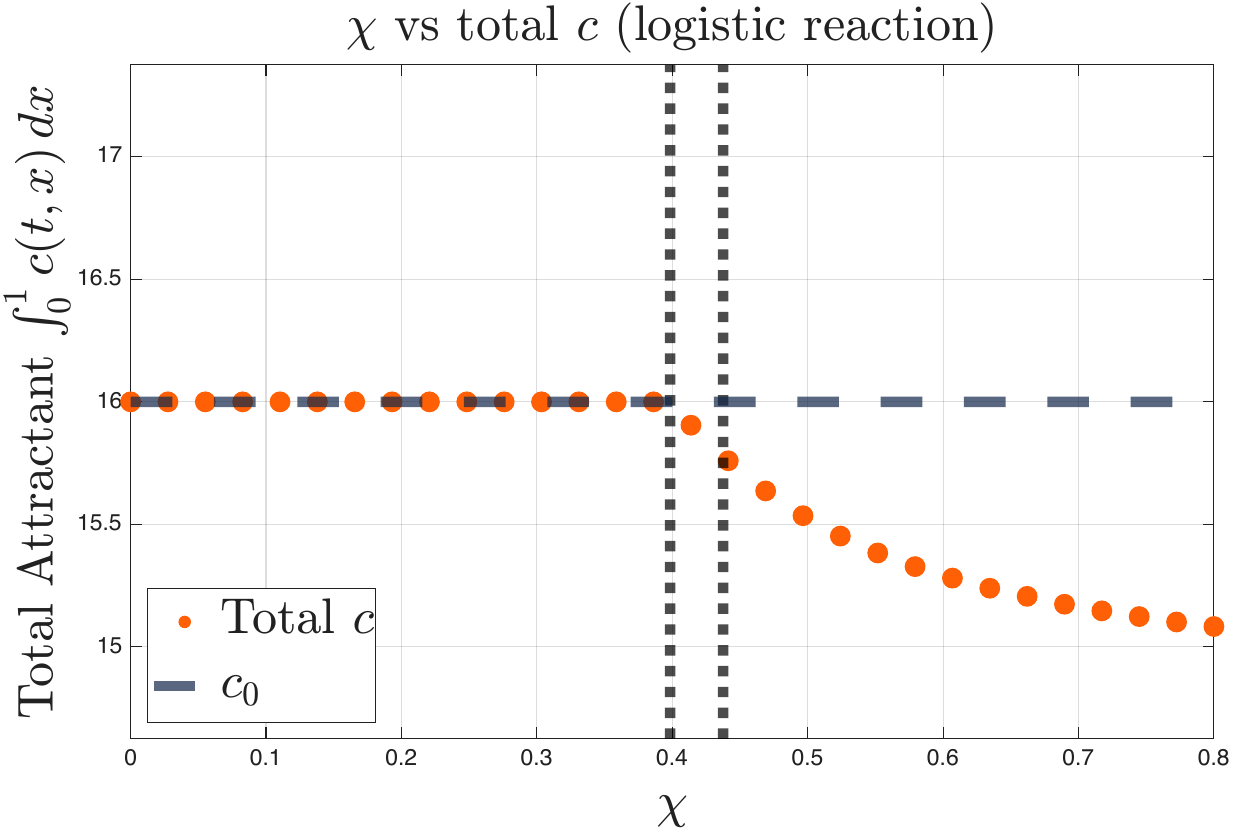}
    \end{subfigure}\vspace{-24pt}
    \begin{subfigure}[b]{\textwidth}
        \centering
        \vspace{2em}
        \caption{Total number of $n$ and $c$ with logistic reaction at $t=50{,}000$}
    \end{subfigure}

    \vspace{1em}

    \begin{subfigure}[b]{0.48\textwidth}
        \includegraphics[width=\linewidth]{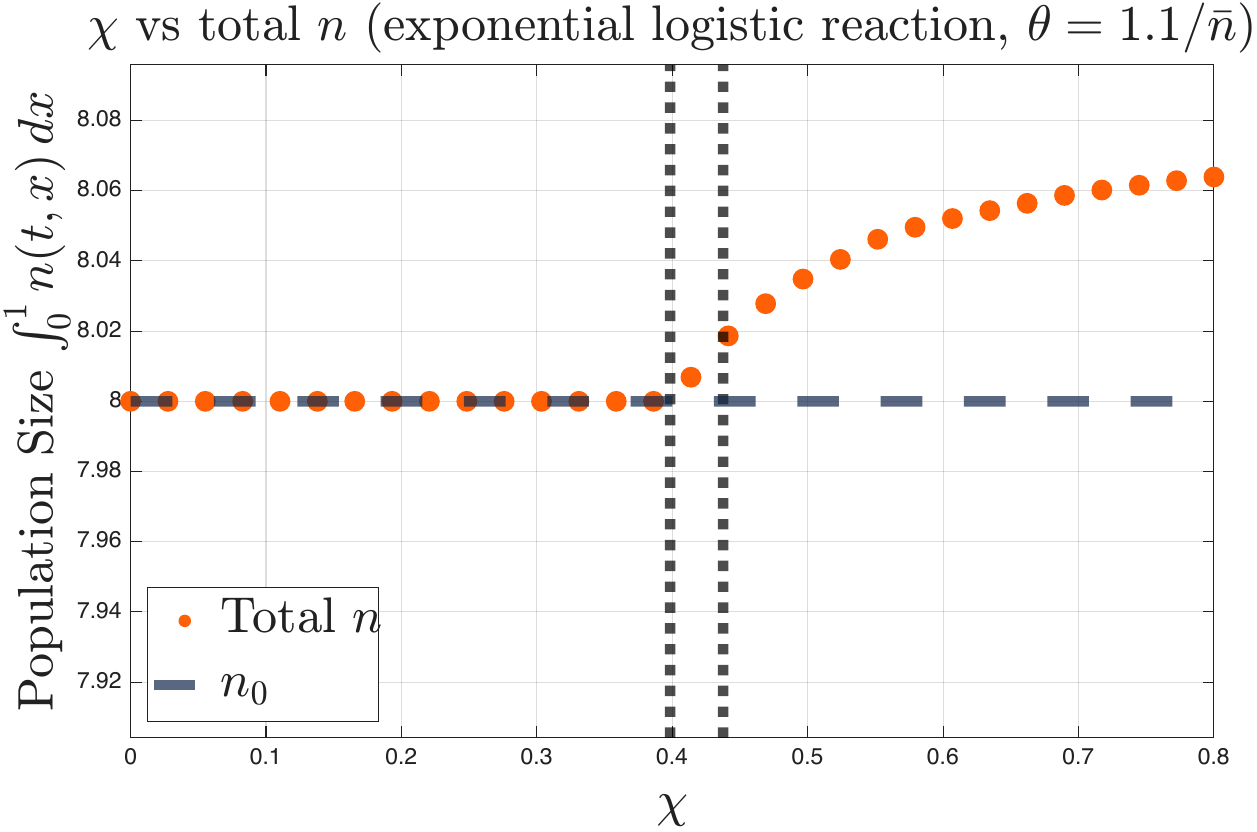}
    \end{subfigure}
    \hfill
    \begin{subfigure}[b]{0.48\textwidth}
        \includegraphics[width=\linewidth]{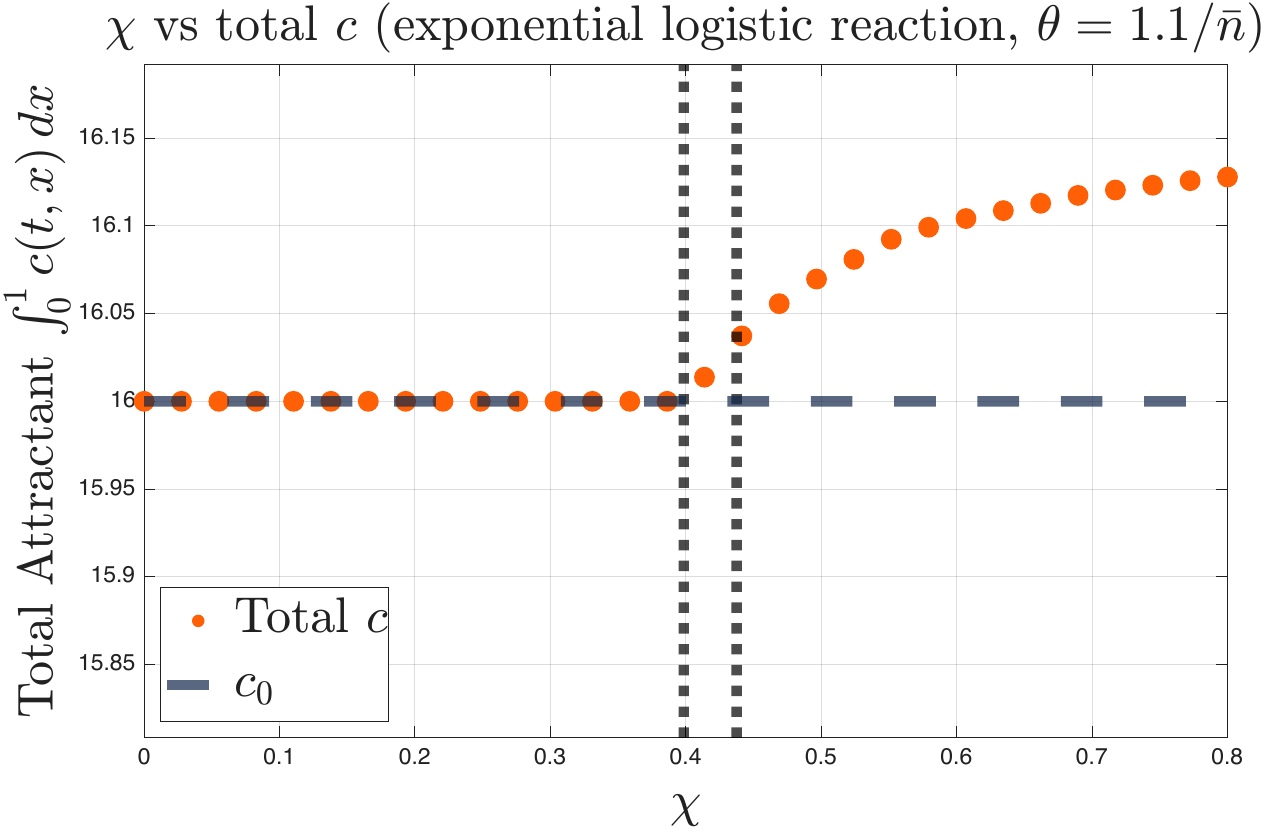}
    \end{subfigure}\vspace{-24pt}
    \begin{subfigure}[b]{\textwidth}
        \centering
        \vspace{2em}
        \caption{Total number of $n$ and $c$ with exponential logistic reaction at $t=50{,}000$}
    \end{subfigure}

    \vspace{1em}

    \caption{Spatial masses $\int_0^1n(t,x)\,dx$ and $\int_0^1c(t,x)\,dx$ at $t=50{,}000$, averaged over ten initial conditions. The upper panels use $f(n)=rn(1-n/\bar n)$, and the lower panels use $f_\theta(n)=r e^{\theta(\bar n-n)}n(1-n/\bar n)$ with $\theta=1.1/\bar n$. The parameters are $\eta_n=0.05$, $\eta_c=0.08$, $\alpha=0.2$, $\beta=0.1$, $r=0.1$, and $\bar n=8$. We use thirty values of $\chi$ in $[0,0.8]$. Horizontal dashed lines show $\bar n=8$ and $\bar c=16$; vertical dashed lines show $\chi^*$ and $\chi^{**}$.}
\label{fig:w-parameters}
\end{figure}

\section{Discussion}
\label{sec:discussion}

In this paper, we formulated and analyzed a version of the Keller--Segel model for chemotaxis in a discrete metapopulation, describing the directed motion of individuals in terms of biased random walks on the edge between a focal patch and a chosen neighbor in a metapopulation. We show that directed motion towards increasing concentrations of the chemical attractant can destabilize a spatially uniform equilibrium state, but that the conditions for instability and resulting pattern-forming properties rely heavily on the eigenvalues for the graph Laplacian for the connectivity network between patches. To obtain a greater degree of analytical tractability for these metapopulation models, we use graph limit theory to derive a continuum limit of our Keller--Segel model and obtain a system of integro-differential equations to describe spatial pattern formation with directed motion on graphons. Finally, we use a mix of linear stability analysis, weakly nonlinear analysis, and numerical simulations to highlight qualitative properties of the patterns that emerge from chemotactic motion on graphons. 

Prior work on directed motion on metapopulation models has considered piecewise motion rules describing a discrete analogue of advection \citep{shaw2024dynamical,young2018fast,zou2017mathematical,bezgachev2024identification}, while recent work on integro-differential models have also considered rules for spatial motion based on an underlying attractiveness field for spatial motion \citep{whiteley2022modelling}. However, the framework for directed motion introduced in this paper both formulates the rules for directed motion in terms of an underlying biased random walk on a metapopulation network and highlights the connection to continuum limits using the convergence of (random) graph models to graphons. The use of the Keller--Segel model on graphs and graphons in this paper serves as a proof of concept for exploring instabilities due to directed motion on metapopulations, as well as the use of weakly nonlinear stability analysis to study qualitative behavior for graphon models. As there have been similar recent developments in the formulation models of nonlinear diffusion equations introduced in the context of random walks on graphs \citep{fahimipour2022sharp,falco2022random,skardal2023pattern,koike2024bifurcation}, our approach for studying graphon models with state-dependent movement rules may be helpful for generalizing features of PDE models for continuous spatial domains to understanding related questions for spatial pattern formation and spreading phenomena in metapopulation models of biological and social systems.

In the following subsections, we present additional discussion on our results related to pattern formation on graphs and graphons (Section \ref{sec:patternformingdiscussion}), provide additional discussion on our convergence results and the derivation of our graphon model (Section \ref{sec:discussiongraphonconvergence}), and highlight directions for future research for modeling directed motion on metapopulation graphs and graphons (Section \ref{sec:futurework}).

\subsection{Pattern-Forming Behavior on Graphs and Graphons}
\label{sec:patternformingdiscussion}

Using weakly nonlinear stability analysis for our graphon model, we were able to characterize the conditions under which chemotactic motion produced a supercritical or subcritical pitchfork bifurcation. This kind of analysis is often used to characterize the form and criticality of spatial patterns in PDE and integro-differential models \citep{SegelLevin1976,topaz2006nonlocal,woolley2022boundary,woolley2025pattern,aydogmus2015patterns,aydogmus2018phase,zhao2026pattern}, and we are able to study how the form of the reaction term $f(n)$ helps to determine the type of pitchfork bifurcation that emerges from directed motion. Beyond characterizing the bifurcation type, we also use the asymptotic approximation we obtain from the resulting amplitude equation to validate numerical simulations close to the onset of pattern formation, and we can integrate the approximate patterned state from weakly nonlinear analysis to determine whether spatial pattern formation tends to increase or decrease the overall population size or level of chemoattractant relative to the levels achieved in the spatially uniform state. 

Notably, we see that the predicted total population densities and chemoattractant concentrations was always lower in the spatially patterned state than in the uniform equilibrium state in the case of logistic growth for the population, indicating that the aggregation mechanism of the Keller--Segel model tends to hinder population growth in the case of logistic kinetics. Logistic growth is typically taken as a baseline assumption in biological models of chemotaxis that incorporate population dynamics \citep{hillen2009user,painter2019mathematical,kong2023existence,kong2024existence}, and logistic reaction terms have also been used to model population dynamics and wealth accumulation in spatial models of socioeconomic systems \citep{juchem2015capital,neto2019returns,hasan2020transport,kong2025boundary}, suggesting that understanding the interaction between directed motion of a population and ecological dynamics featuring logistic growth may be important for understanding predictions for emergent patterns. Furthermore, numerical simulations stochastic process and PDE models of evolutionary games with payoff-driven or environment-driven motion have similarly shown a decrease in population size or payoff for both cheater and cooperator game-theoretic strategies \citep{yao2025spatial,yao2026pattern,zhang2026spatial}, and weakly nonlinear analysis has shown similar declines in population size and resource availability in a PDE model of public good dynamics with a chemotaxis-like term describing directed motion towards increasing concentration of public goods \citep{zhao2026pattern}.

This result from our weakly nonlinear analysis suggests that logistic growth and directed motion may be incompatible as a means for increasing population size through directed motion. In fact, we show in Proposition \ref{prop:logistic-total-mass-deficit} that, for any globally concave reaction function, the total population size across the ring graphon can never exceed the carrying capacity $\bar{n}$, and that the mass will always be lower than the carrying capacity $\bar{n}$ in a spatially patterned state. This means that directed motion and spatial pattern formation can only increase population size in our graph and graphon Keller--Segel model when the population dynamics feature local convexity of the reaction function $f(n)$ at the equilibrium population size $\bar{n}$, so we turned to an alternative reaction function $f_{\theta}(n) = r e^{\theta (\bar{n}-n)} n \left( 1- \frac{n}{\bar{n}} \right)$ to describe locally accelerating growth for $n = \bar{n}$. Incorporating this locally convex function form for the population dynamics allowed for the achievement of spatial patterns with an increased total population size and chemoattractant concentration when the parameter $\theta$ was sufficiently large, indicating that directed motion can potentially lead to greater population abundance depending on the functional form of birth and death rates for a species. In particular, this observation motivates a more careful exploration of the rules for population dynamics and the formulation of interaction terms with attractive or repulsive species when considering models of chemotaxis-like directed motion with birth and death dynamics.

In addition to our observations about the total population size and abundance of chemoattractant, we also used our weakly nonlinear analysis to characterize the form of pattern-forming bifurcations achieved for different choices of graphon $W(x,y)$ and reaction function $f(n)$ for the population dynamics. For the logistic kernel on a uniform band graphon, we found that directed motion produces a supercritical pitchfork bifurcation for low band radius $\rho$ and that a subcritical pitchfork bifurcation could be achieved when $\rho$ was sufficiently close to $\frac{1}{2}$. We saw a similar behavior for the case of a logistic reaction term and a Watts--Strogatz graphon with fixed band radius $\rho$, with a supercritical pitchfork bifurcation occurring for low rewiring probability $p$ and a transition to subcritical pitchfork bifurcation for higher rewiring probabilities. In each of these cases, we see that supercritical patterns tend to be supported when the movement kernel is more concentrated on local motion, while either the extension of the local neighborhood radius $\rho$ or the rewiring probability $p$ allowing longer movement ranges tends to allow the emergence of subcritical bifurcations and the corresponding bistability between uniform and patterned states below the linear instability threshold $\chi^*$. As we also see a transition between subcritical and supercritical patterns in the case of the exponential Allee-logistic model by changing the parameter $\theta$ characterizing the Allee effect, it appears that the resilience of spatial patterns can be achieved by a variety of movement and reaction parameters in the graphon Keller--Segel model. 

We also supplemented the analytical study of the graph and graphon models with the numerical simulations in Section \ref{sec:numericalsimulations}. Through the simulations of the Keller--Segel model on discrete graphs in Section \ref{sec:simulation-graph}, we saw that cycling behavior is possible for metapopulations with small numbers of patches. As the patch number $N$ increased, we saw that the long-time behavior of the graph Keller--Segel model began to resemble the behavior predicted for the graphon model through linear and weakly nonlinear analysis, resulting in sinusoidal patterns along the nodes of a ring graph for sufficiently strong $\chi$. We then used numerical simulations of the graphon model to confirm the accuracy of our predictions from weakly nonlinear stability analysis for $\chi$ close to the onset of the pattern-forming instability and to explore the behavior of the model far from the onset of instability. In particular, we see that the qualitative predictions on change in total population size and the total mass of chemoattractant achieved through spatial patterns continue to hold far from onset, and that a more substantial deviation of these quantities from their values in the uniform state is achieved as the chemotactic sensitivity $\chi$ further increases. 

Furthermore, by plotting the patterns achieved for high chemotactic sensitivity, we are able to observe that the patterns achieved far from threshold have the form of mesa patterns that are typically associated with the behavior of local reaction--diffusion models \citep{kolokolnikov2007self,hillen2007classification,mckay2012stability,zhao2026pattern}. This behavior stands in contrast to the spike-like patterns that often arise in local chemotaxis-like models \citep{potapov2005metastability,kolokolnikov2014basic,buttenschoen2020cops,kong2022stability,kong2023existence,kong2024existence,zhao2026pattern}, suggesting that the nonlocal nature of our model of directed motion may smooth out the localizing nature of chemotactic instabilities and allow for different qualitative behaviors in the limit of strong directed motion. An interesting question for future work would be to further explore the properties of patterns far from onset in a range of graphon models, exploring how modes of motion including diffusion, cross-diffusion, and directed motion can interact with different reaction dynamics and network connectivity to produce different possible long-time equilibrium states. 

Finally, another notable feature gleaned from the analysis of our graphon Keller--Segel model is the form of the eigenvalues achieved in our linearized system. In particular, for both the uniform-band and small-world graphons, the magnitudes of the eigenvalues of the graphon Laplacian remained bounded even in the limit as the wavenumber $k$ tended to infinity, ruling out the possibility of a short-wave instability characterized by the dominance of emerging patterns oscillating on arbitrarily fine spatial scales. Although short-wave instability is not a property of the local PDE version Keller--Segel model with logistic reaction kinetics in any spatial dimension \citep{hillen2009user,wrzosek2004global}, there have been many PDE models of directed motion in the theoretical ecology and evolutionary games literature featuring short-wave instabilities due to directed motion of individuals towards members of their own population \citep{helbing2009pattern,deforest2013spatial,xu2017strong,deforest2018game,funk2019directed,yao2026pattern}. Such instabilities do not appear to be possible in our framework for modeling directed motion on metapopulation graphs or graphons for the example connectivity kernels studied in this paper, suggesting that our approach for modeling directed motion between patches in metapopulations may present a natural setting for exploring more general rules for directed motion in multispecies population models of ecological and social systems.

\subsection{Discussion on Graphon Convergence Results}
\label{sec:discussiongraphonconvergence}

\paragraph{Graph limits (Section~\ref{sec:graphon-limits}).}
For every fixed $T>0$ and $1\le p<\infty$, Theorem~\ref{thm:det-Lp-conv} establishes the convergence of the cell-average weighted-graph dynamics in $C([0,T];L^p(I)^2)$, subject to uniform kernel and solution bounds. For independent-edge random graphs sampled from the same cell averages, Theorem~\ref{thm:RG-limit} establishes convergence in probability in the same space. The associated error estimate separates the deterministic discretization error from the $O_{\pr}(N^{-1/2})$ fluctuation between the sampled and weighted solutions.

Under the assumptions of Theorem~\ref{thm:det-cut}, cut-distance convergence of the kernels combined with $L^1$ convergence of the initial data under identical measure-preserving relabelings implies convergence of the relabeled solutions in $C([0,T];L^p(I)^2)$ for every $1\le p<\infty$. The proof first gives an $L^1$ estimate for the solution error and then interpolation with the uniform $L^\infty$ solution bounds gives convergence in $L^p$ for every $1\le p<\infty$.

\paragraph{Scope of the convergence results.}
These results rigorously justify the comparison of population and chemoattractant profiles between the discrete graph and continuous graphon models on fixed finite time intervals, which allows us to observe good agreement with the behavior of the initial emergence of spatial patterns starting from initial conditions close to the uniform equilibrium in both the graph and graphon model. This convergence result provides further confirmation of the intuition gleaned from analogous pattern-formation threshold formulas obtained for the classes of the graph and graphon Keller--Segel model, and the extension of our convergence results to the case of large families of dense random graphs extends our utility of the graphon model to glean some understanding of the statistical behavior of the onset of pattern formation for sufficiently large random metapopulation structures. However, these finite-time convergence results do not allow us to conclude that the long-time behavior of our graphon model should necessarily agree with the corresponding asymptotic behavior of our graph Keller--Segel model for a sufficiently large number of patches. Establishing such asymptotic correspondence requires additional spectral and nonlinear analysis, as discussed in Section~\ref{sec:futurework}.

More broadly, this graphon-limit approach readily extends to other directed-motion models derived from biased random walks on networks. The resulting nonlocal continuum equations provide a unified setting to analyze how network topology drives macroscopic pattern formation while naturally mediating high-frequency spatial modes. Existing models of metapopulation dynamics in ecological and social systems have frequently been formulated using piecewise modes of directed motion on discrete patches \citep{young2018fast,shaw2024dynamical,zou2017mathematical}. Exploring analytical characterizations of pattern formation using the smooth movement rules derived in our framework may therefore yield deeper insights into the dynamics of these existing models.

\subsection{Future Work and Open Problems}
\label{sec:futurework}

Finally, we outline several directions for future research motivated by this work. These include generalizing the families of metapopulation networks that can be described in terms of our graphon dynamical system, extending the analysis of dynamical behavior in the graph and graphon models to explore long-time agreement between these models for sufficiently large $N$, and connecting these deterministic systems to underlying individual-based stochastic processes that describe finite populations with edge-based directed motion.  

\paragraph{Beyond dense graph sequences.}
Extending the present framework to sparse random graphs is an interesting direction for modeling realistic metapopulations. In the sparse regime, the scaling of the average degree fundamentally alters the relevant Laplacian spectrum and topology. This structural shift requires a modified analysis of the linear instability thresholds and the critical modes selected by the Keller--Segel dynamics \citep{kaliuzhnyi2017semilinear,lovasz2012large,janson2010graphons}. In particular, the graphons arising from the limits of sparse graphs typically feature heterogeneous degree distributions, and therefore our approach using Fourier series to study distance-dependent graphons will not be able to hold for these families of graphons. Although prior work on graphon games \citep{carmona2022stochastic} and Kuramoto models on graphons \citep{chiba2018bifurcations,chiba2019meanfield2,Chiba2019MeanField} has studied eigenvalues and eigenfunctions of the adjacency function for power-law graphons, these calculations cannot be extended to study eigenvalues of the graph Laplacian in the power-law case. Therefore we may need to develop different approaches for analyzing the onset of spatial patterns in the Keller--Segel model on metapopulations with heterogeneous patch-degree distributions.

\paragraph{Center Manifold Reduction and Persistence of bifurcations on large random graphs.}  In our study of solutions to the graphon model, we employed formal asymptotic methods to describe heuristic approximations for emerging spatial patterns and used these to predict qualitative behaviors for example kernels and reaction functions. This suggests that one natural direction for future research is to pursue a more rigorous of the asymptotic behavior of our graphon model, as well as to link these behaviors with a center-manifold reduction for the graph model as well. Center manifold reduction typically produces an analogous expansion to the one achieved through the formal approach from weakly nonlinear stability analysis \citep{clifton2026kernel}, with the added advantage of a guarantee that the solution of the spatial model is described by the amplitude equation on the center manifold. By extending our exploration of the instability to this approach, we can further justify our theoretical predictions for the mass of solutions close to the onset of instability.

In addition, exploring center manifold reduction for the graphon model may also serve as a first step in understanding the long-time behavior for the Keller--Segel model on sufficiently large discrete metapopulations. Because our convergence results apply only on time intervals of fixed length, we are not guaranteed that the long-time behavior of our graph and graphon models will agree even for arbitrarily large families of graphs. Such comparisons between the center manifold dynamics and steady state solutions for graph and graphon models have been provided in recent work on a network version of the Swift-Hohenberg equation \citep{bramburger2023pattern,bramburger2025persistence}, showing the agreement in qualitative behavior between the graphon model and corresponding large graphs with sufficiently many nodes and demonstrating how the convergence of the behaviors depends on the eigenvalues of the graph and graphon Laplacians. As the Swift-Hohenberg graphon model represents a scalar problem, there may be some challenges in extending these kinds of results featuring multiple equations per node and a resulting system of graphon equations. Correspondingly, it may be helpful to see how the Kronecker decomposition of the linearization matrix for our graph Keller--Segel model may be helping in performing a center manifold reduction or other nonlinear analysis for our graph Keller--Segel model, and to explore additional exploration on the behavior of our discrete metapopulation model.

\paragraph{Connection to Pattern Formation with Directed Motion in Individual-Based Models.}

In this paper, we started with a baseline model of an infinite population living on a spatially discrete metapopulation, then explored the behavior of incorporating a continuum of patches into our metapopulation model. However, the motivation for our model of directed motion through a biased random walk on the edges of the metapopulation graph naturally suggests a corresponding individual-based particle system for a finite population of individuals living on a graph metapopulation, with birth, death, and directed motion between patches characterized by rates in a continuous-time stochastic process that depend on the current spatial arrangement of the population. A similar stochastic model on graphs was recently considered by \cite{zhang2026graphon} in the case of simple random walks between patches, who proved the convergence in probability of a family of birth-death-migration processes on a growing family of graphs to a nonlocal reaction-diffusion equation with movement on a corresponding graphon metapopulation. The comparison between stochastic spatial models and corresponding infinite-population limits has long been a point of interest in the study of spatial modeling of populations \citep{DURRETT1994363,shnerb2000importance,cantrell2004deriving,pacala2020importance}, and similar particle models of reaction-diffusion and reaction-transport processes have also been considered to explore the role of finite-population effects on the emergence of Turing patterns on spatial lattices \citep{butler2009robust,biancalani2010stochastic,butler2011fluctuation,asslani2012stochastic,woolley2011power,woolley2011influence,maini2012turing,cao2014stochastic,karig2018stochastic,kim2020stochastic}. The use of stability analysis in Fokker-Planck descriptions and direct stochastic simulation, it has often been observed that Turing patterns can arise in stochastic spatial models for ratios of diffusivities that would not produce pattern formation in the corresponding PDE limit \citep{butler2011fluctuation,kim2020stochastic,yao2026pattern}, and therefore establishing corresponding pattern-formation thresholds in a particle-based Keller--Segel model on graphs would be a first step in exploring the role of the discreteness of individuals and space on pattern formation can play in spatial dynamics on graphs.

In addition, the current paper focused only on the case of local interactions within patches, while the coupling between patches was limited to the terms involving spatial movement. Many nonlocal PDE models of population dynamics have been proposed to study distance-dependent interactions in continuous space, exploring how nonlocal facilitation and competition can help to promote spatial patterns in ways that cannot be achieved by local interactions alone \citep{martinez2013vegetation,clifton2026kernel}. This suggests that one possible direction for future research is to formulate metapopulation models featuring both movement between patches and interactions between individuals with neighboring graphs. In particular, it may make sense to build on recent work in the evolutionary games literature featuring game-theoeretic interactions across neighboring patches in metapopulations \citep{DURRETT1994363,aydogmus2015patterns,hwang2013deterministic,tan2019graphical,tan2021distributed,giles2026scaling}, which may produce graphon models with pattern-formation due to anti-coordination effects \citep{hwang2013deterministic} or due to a mixture of game-theoretic payoffs and density-dependent regulation \citep{DURRETT1994363}. For such graph and graphon models with interactions across patches, it may be possible that a Kronecker decomposition of the linearization fo the corresponding patch model may depend both on the graph Laplacian $L$ and directly on the graph's adjacency matrix $A$ as well.

\paragraph{Generalization of Chemotactic Sensitivity Function.}

The choice to formulate the biased random walk along edges in terms of the exponential weight $e^{\chi c_i}$ was made for analytical convenience, as the exponential terms simplified calculations in performing the linear and weakly nonlinear stability analysis of our model. However, exploring the choice of functional form for chemotactic sensitivity has been a point of emphasis in the literature on PDE versions of the Keller--Segel model \citep{hillen2009user,kalinin2009logarithmic,rodriguez2010local,desvillettes2019logarithmic,shemtaga2025well}. We can consider a similar approach for generalizing the form of chemotactic sensitivity used in the biased random walk, formulating a new version of our graph Keller--Segel model with the system of ODEs
\begin{subequations}
\begin{align}
\frac{dn_i}{dt} &= \frac{\eta_n}{N} \sum_{j=1}^N A_{ij} \left[ \frac{g\left( \chi c_i \right) n_j - g\left( \chi c_j\right) n_i}{g\left( \chi c_i \right) + g\left( \chi c_j \right)} \right] + f(n_i) \\
\frac{dc_i}{dt} &= \frac{\eta_c}{N} \sum_{j=1}^N A_{ij} \left( \frac{c_j - c_i}{2}\right) + \alpha n_i - \beta c_i
\end{align}
\end{subequations}
for a non-decreasing smooth function $g(\cdot)$ satisfying $g(0) > 0$. This function $g(\cdot)$ can describe the rules for the bias individuals have in moving towards patches with higher concentrations of chemoattractant.

The corresponding generalization of the graphon Keller--Segel model is then given by the following system of integro-differential equations
\begin{subequations} \label{eq:our-graphon-model-generalized}
\begin{align}
\partial_t n(x,t)
&= \eta_n \int_0^1 W(x,y)
\left[
\frac{g\left(\chi c(x,t)\right) n(y,t) - g\left(\chi c(y,t) \right) n(x,t)}
     {g\left(\chi c(x,t)\right) + g\left(\chi c(y,t)\right)}
\right] \, dy
+ f(n(x,t)),
\label{eq:KS_graphon_n-generalized}
\\
\partial_t c(x,t)
&= \eta_c \int_0^1 W(x,y)\left[\frac{c(y,t)-c(x,t)}{2}\right]\,dy
+ \alpha n(x,t) - \beta c(x,t),
\label{eq:KS_graphon_c-generalized}
\end{align}
\end{subequations}
and one can apply similar techniques like linear and weakly nonlinear stability analysis for this model to see how the chemotactic sensitivity function $g(\cdot)$ impacts the emergence and qualitative properties of spatial patterns. 

Notably, the equivalent generalization of the one-dimensional PDE Keller--Segel model with chemotactic sensitivity function $g(\cdot)$ is given by the system 
\begin{subequations}
\begin{align}
\dsdel{n}{t} &= \eta_n \nabla \cdot \left( \nabla n - n \nabla \log\left(g(\chi c) \right) \right) + f(n)\\
\dsdel{c}{t} &= \eta_c \Delta c + \alpha n - \beta c,
\end{align}
\end{subequations}
so the choices of sensitivity function of an exponential mapping $g(\chi c) = e^{\chi c}$ and a linear mapping $g(\chi c) = 1 + \chi c$ respectively produce the traditional Keller--Segel model of Equation~\eqref{eq:KS-PDE} and the logarithmic Keller--Segel model~\citep{sleeman2005existence,hillen2009user,shemtaga2025well}. In particular, exploring our graph and graphon Keller--Segel model with sensitivity function $g(\chi c) = 1 + \chi c$ would serve as interesting next step in exploring the role of the sensitivity function on pattern formation due to directed motion in metapopulations.

\begin{quote}
\begin{center}
\textbf{Acknowledgements}
\end{center}
Yefei Zhang and Daniel B. Cooney are grateful for support from the Mrs. E.J. Hoover Program for Advanced Students in Mathematics, and would like to thank Amy Bradu, Stanley Combes, and Rui (David) Sang for many helpful conversations on spatial pattern formation on graph metapopulations through the 2025 Illinois Mathematics Laboratory (IML) Summer Program. 
\end{quote}

\begin{quote}
\begin{center}
    \textbf{Statement on Code and Data Availability}
\end{center}
    The numerical simulations in this paper were implemented in Python and MATLAB. The code used to generate the figures and results is publicly available in the GitHub repository: \href{https://github.com/yxuanzhao/Keller-Segel-Model-on-Graphs-and-Graphons}{yxuanzhao/Keller-Segel-Model-on-Graphs-and-Graphons}. There is no additional data required to reproduce the results in the paper. 
\end{quote}

%\medskip\noindent\textbf{Declaration on the use of AI.}
\begin{quote}
\begin{center}
\textbf{Declaration on the Use of AI}
\end{center}
In the spirit of the transparency recommendations in the Leiden Declaration on Artificial Intelligence and Mathematics~\citep{LeidenDeclaration2026}, large language models were used as assistive tools for grammar and language editing and for improving the formal presentation of mathematical arguments, including assistance in formulating theorem statements. The code used for the numerical simulations of the Keller–Segel dynamics was also drafted with the assistance of large language models. The mathematical ideas, arguments, and proofs were developed by the authors. All mathematical statements and proofs were subsequently checked and reviewed by the authors, who take full responsibility for the content of the manuscript.
\end{quote}

\bibliography{References}

\newpage
%%%%%%%%%%%%%%%%%%%%%%%%%%%%%%%%%%%%%%%%%%%%%%%%%%%%%%%%
\appendix

\section*{Summary of Appendix Sections}
In the following appendices, we have included additional auxiliary results on convergence and well-posedness of our graphon model, and we have also provided additional derivations for our linear and weakly nonlinear stability analysis of our metapopulation Keller--Segel dynamics. Appendix~\ref{app:graph-jacobian} provides the detailed Jacobian calculations for the graph Keller--Segel model at the homogeneous equilibrium.
Appendix~\ref{app:graph-limits-proofs} proves graphon well-posedness and uniform finite-time bounds, from which the corresponding discrete results follow by step reconstruction. It also contains the proof of the cut-norm consistency estimate and the auxiliary results on pointwise discretization and cut-norm concentration for sampled graphs.
Appendix~\ref{sec:symmetry-invariant} proves that the graphon dynamics preserve reflection symmetry when the kernel and initial data have this symmetry.
Finally, Appendix~\ref{sec:derivation_stability_analysis} derives the perturbation equations through cubic order used in the weakly nonlinear analysis of Section~\ref{sec:graphon-weakly-general}.
%%%%%%%%%%%%%%%%%%%%%%%%%%%%%%%%%%%%%%%%%%%%%%%%%%%%%%%
\section{Additional Steps in Deriving Jacobian for Graph Metapopulation Model}\label{app:graph-jacobian}

In this section, we will provide additional details in the derivation of the linearization for the Keller--Segel model on metapopulation graphs. As mentioned in Section \ref{sec:graph-LSA}, we may represent the state variables for our system of ODEs from Equation \eqref{eq:KS-graph} in terms of the vector
\begin{equation}
x = \left( n_1, \dots, n_N, \, c_1, \dots, c_N \right)^{\top} \in \mathbb{R}^{2N}.
\end{equation}
Using these variables, we can rewrite our system of ODEs in the form
\begin{subequations}
\begin{align}
\dot n_i &= F_i(x) \\
\dot c_i &= G_i(x),
\end{align}
\end{subequations}
where the righthand sides $F_i(x)$ and $G_i(x)$ are a shorthand notation for the ODEs from Equation \eqref{eq:KS-graph}, and are therefore respectively given by
\begin{subequations}
\begin{align}
F_i(x) &= \frac{\eta_n}{N} \sum_{j=1}^N A_{ij} \left[  \frac{e^{\chi c_i} n_j - e^{\chi c_j} n_i}{e^{\chi c_i} + e^{\chi c_j}}\right] + f(n_i) \\
G_i(x) &= \frac{\eta_c}{N} \sum_{j=1}^N A_{ij} \left[\frac{c_j - c_i}{2} \right] + \alpha n_i - \beta c_i
\end{align}
\end{subequations}

With this notation, we can describe the linearization of our nonlinear system around the uniform equilibrium based on the values taken by eight representative partial derivatives. Namely, the Jacobian matrix can be characterized by the partial derivatives of the righthand sides $F_i(x)$ and $G_i(x)$ for patch $i$ with respect to both the population $n_i$ and chemical concentration $c_i$ at patch $i$ as well as with respect to the equivalent quantities $n_j$ and $c_j$ corresponding to a different patch $j$. 

Using the convention that $A_{ii} = 0$ (corresponding to the case of graph with no self loops), we see that the partial derivatives for $F_i(x)$ are given by
\begin{subequations}
\begin{align}
\frac{\partial F_i(x)}{\partial n_i}\bigg|_{(\bar{n},\bar{c})} &= - \frac{ \eta_n}{2N} \sum_{\substack{j=1 \\ j \ne i}}^{N} A_{ij} + f'(\bar{n}) \\
\frac{\partial F_i(x)}{\partial c_i}\bigg|_{(\bar{n},\bar{c})}  &= \frac{\chi \eta_n \bar{n}}{2N} \sum_{j=1}^{N} A_{ij} \\
\frac{\partial F_i(x)}{\partial n_j}\bigg|_{(\bar{n},\bar{c})} &=  \left(\frac{ \eta_n}{2N}\right) A_{ij} \\
\frac{\partial F_i(x)}{\partial c_j}\bigg|_{(\bar{n},\bar{c})}  &= -\left(\frac{\chi \eta_n \bar{n}}{2N} \right)A_{ij}.
\end{align}
\end{subequations}
In the righthand side $G_i(x)$, each of our quantities of interest is either absent or appears as a linear factor in the expression for $G_i(x)$. The corresponding partial derivatives for these terms are given by
\begin{subequations}
\begin{align}
\frac{\partial G_i(x)}{\partial n_i}\bigg|_{(\bar{n},\bar{c})} &= \alpha \\
\frac{\partial G_i(x)}{\partial c_i}\bigg|_{(\bar{n},\bar{c})}  &= - \frac{\eta_c}{2N} \sum_{j=1}^{N} A_{ij} - \beta \\
\frac{\partial G_i(x)}{\partial n_j}\bigg|_{(\bar{n},\bar{c})} &=  0 \\
\frac{\partial G_i(x)}{\partial c_j}\bigg|_{(\bar{n},\bar{c})}  &= \left(\frac{\eta_c}{2N} \right) A_{ij}.
\end{align}
\end{subequations}

We can then use each of the partial derivatives computed above to see that our Jacobian matrix for the graph Keller--Segel system can be described in terms of the reaction dynamics at each patch, considering the rules for motion between any given patch $i$ and $j$, and using the corresponding graph Laplacian describing the connectivity between patches. In particular, we can see these partial derivatives can be used to show that the $2N\times 2N$ linearization matrix $J$ around the spatially uniform equilibrium can be written in terms of the following decomposition
\begin{equation}
    J = P \otimes I_N + C \otimes L,
\end{equation}
where $\otimes$ represents the Kronecker product, $I_N$ is the $N \times N$ identity matrix, $L$ is the ($N \times N$) graph Laplacian matrix for our metapopulation network, and the $2\times2$ matrices $P$ and $C$ respectively describe the rules for local reactions and the rules for motion between patches due to diffusive effects and chemotactic motion. The matrices $P$ and $C$ can be written as
\begin{subequations}
\begin{align}
P &= \begin{pmatrix}
f'(\bar{n}) & 0 \\
\alpha & -\beta
\end{pmatrix} \\
C &= \frac{1}{2} \begin{pmatrix}
\eta_n & -\eta_n \chi \bar{n} \\
0 & \eta_c 
\end{pmatrix},
\end{align}
\end{subequations}
where we see that $P$ is the Jacobian matrix for the reaction dynamics in a single patch and $C$ describes the linearized effects of spatial motion (analogous to the linearization matrix describing the movement terms in a PDE Keller--Segel model).

%%%%%%%%%%%%%%%%%%%%%%%%%%%%%%%%%%%%%%%%%%%%%%%%%%%%%%%
\section{Auxiliary results and proofs for graph limits}\label{app:graph-limits-proofs}
\subsection{Proof of graphon well-posedness (Proposition~\ref{prop:graphon_bounds})}\label{app:graphon-bounds}

We use the following integral comparison principle.

\begin{lemma}[Integral comparison]\label{lem:app-integral-comparison}
Let $q\in C^1([0,T];L^\infty(I))$ and let $a\in L^1(0,T)$ be nonnegative. 
If $\partial_t q(t,x) \ge -a(t)$ for a.e.\ $(t,x)$ with $q(t,x) < 0$, then $\|q^-(t)\|_{L^\infty} \le \|q^-(0)\|_{L^\infty} + \int_0^t a(s)\,ds$. Analogously, if $\partial_t q(t,x) \le a(t)$ on the set $\{q > 0\}$, the same bound holds for $\|q^+(t)\|_{L^\infty}$.
\end{lemma}

\begin{proof}
Absolute continuity $q(t,x)=q(0,x)+\int_0^t\partial_sq(s,x)\,ds$ and the scalar chain rule yield
\begin{align*}
q^\pm(t,x)=q^\pm(0,x)\pm\int_0^t\ind_{\{\pm q(s,x)>0\}}\partial_sq(s,x)\,ds.
\end{align*}
Applying the assumed bounds and taking the essential supremum over $x$ yields the claims.
\end{proof}

\begin{proof}[Proof of Proposition~\ref{prop:graphon_bounds}]
\emph{Step 1: Local well-posedness.} On every $L^\infty$-ball, $0\le S\le1$,
$|\partial_aS|=|\partial_bS|\le|\chi|/4$, and $M_x(|W|)<\infty$ imply
\begin{align*}
\|\mathcal F_W[U_1]-\mathcal F_W[U_2]\|_{L^\infty(I)^2}
\le C_R\|U_1-U_2\|_{L^\infty(I)^2}.
\end{align*}
The Picard--Lindel\"of theorem in Banach spaces thus guarantees the unique maximal $C^1$ solution and the
blow-up alternative in part \textup{(i)}.

\smallskip
\emph{Step 2: Positivity.} Assume the hypotheses of part \textup{(ii)} and fix $T'<t_{\max}$. Define $R:=\sup_{t\in[0,T']}\|U(t)\|_{L^\infty(I)^2}$, and let $L_f$ be a Lipschitz constant of $f$ on $[-R,R]$.

Setting $Q_n(t):=\|n^-(t)\|_{L^\infty}$, the evolution equation implies that for a.e.\ $(t,x)$ where $n(t,x)<0$, 
\begin{align*}
    \partial_tn(t,x)\ge-(\eta_nM_x+L_f)Q_n(t).
\end{align*}
Lemma~\ref{lem:app-integral-comparison} and Gr\"onwall's inequality yield $Q_n\equiv0$. Similarly, defining
$Q_c(t):=\|c^-(t)\|_{L^\infty}$ and using the fact that $n\ge0$, we have on the set $\{c<0\}$, $\partial_tc(t,x)\ge-\frac{\eta_cM_x}{2}Q_c(t)$, so the same integral argument gives $Q_c\equiv0$. Since $T'<t_{\max}$ was arbitrary,
$n,c\ge0$ on the maximal interval.

\smallskip
\emph{Step 3. Global bounds.}
Let $M_n(t):=\|n(t)\|_{L^\infty}$. Positivity and $h\le C_h$ imply $\partial_tn(t,x)\le(\eta_nM_x+C_h)M_n(t)$.
Integrating in time and taking the essential supremum over $x$ yields
\begin{align*}
M_n(t)
\le
\|n_{\mathrm{in}}\|_{L^\infty}
+(\eta_nM_x+C_h)\int_0^tM_n(s)\,ds,
\end{align*}
hence $M_n(t)\le R_1(T)$ on $[0,T]$ by Gr\"onwall's inequality.

Fix $T<t_{\max}$ and let $\bar c$ solve
\begin{align*}
\dot{\bar c}=\alpha R_1(T)-\beta\bar c,
\qquad
\bar c(0)=\|c_{\mathrm{in}}\|_{L^\infty}.
\end{align*}
Set $\mathfrak{c}=c-\bar c$ and $Q_{\mathfrak{c}}(t):=\|\mathfrak{c}^+(t)\|_{L^\infty}$. Since $\mathfrak{c}^+(0)=0$, and on the set
$\{\mathfrak{c}>0\}$,
\begin{align*}
\partial_t\mathfrak{c}(t,x)
&\le
\frac{\eta_c}{2}\int_IW(x,y)\mathfrak{c}^+(t,y)\,dy
\le
\frac{\eta_cM_x}{2}Q_\mathfrak{c}(t),
\end{align*}
Lemma~\ref{lem:app-integral-comparison} gives
$Q_\mathfrak{c}(t)\le\frac{\eta_cM_x}{2}\int_0^tQ_\mathfrak{c}(s)\,ds$.
Thus $Q_\mathfrak{c}\equiv0$ and $c\le\bar c$. Solving the scalar equation yields
$\sup_{t\in[0,T]}\|c(t)\|_{L^\infty}\le R_2(T)$.

These bounds hold on every compact subinterval of $[0,t_{\max})$. If $t_{\max}<\infty$,
they give
\begin{align*}
\sup_{t<t_{\max}}\|U(t)\|_{L^\infty(I)^2}
\le\max\{R_1(t_{\max}),R_2(t_{\max})\}<\infty,
\end{align*}
contradicting the blow-up alternative. Hence $t_{\max}=\infty$.
\end{proof}

%%%%%%%%%%%%%%%%%%%%%%%%%%%%%%%%%%%%%%%%%%%%%%%%%%%%%%%%%%
\subsection{Proof of the cut-norm consistency estimate}\label{app:cut-consistency-proof}

\begin{proof}[Proof of Lemma~\ref{lem:cut-consistency}]
Set $H:=K-W$ and $h:=\|H\|_{\infty\to1}$. 
For the $c$-component, we have
\begin{align*}
    F_K^{(c)}(U)-F_W^{(c)}(U)
    =\frac{\eta_c}{2}\bigl(T_Hc-cT_H1\bigr).
\end{align*}
By the definition of the operator norm, $\|T_Hv\|_{L^1}\le h\|v\|_{L^\infty}$ for any $v\in L^\infty(I)$. This yields $\|F_K^{(c)}(U)-F_W^{(c)}(U)\|_{L^1}\le \eta_c Rh$.
For the $n$-component, since $\|c\|_{L^\infty}\le R$, the fundamental theorem of calculus gives
\begin{align*}
    S(c(x),c(y))
    =S(-R,c(y))+\int_{-R}^R
       \ind_{\{s<c(x)\}}\partial_1S(s,c(y))\,ds.
\end{align*}
Substituting this identity into the transport term difference yields
\begin{align*}
    F_K^{(n)}(U)-F_W^{(n)}(U)
    &=\eta_n\bigg(
        T_H\big(S(-R,c)n\big)-nT_H\big(1-S(-R,c)\big)\\
        &\qquad+\int_{-R}^R\ind_{\{s<c\}}
          \Big(T_H\big(\partial_1S(s,c)n\big)+nT_H\big(\partial_1S(s,c)\big)\Big)\,ds
      \bigg).
\end{align*}
Using $0\le S\le 1$, $|\partial_1S|\le |\chi|/4$, and the $L^1$ bound for $T_H$, we obtain
\begin{align*}
    \|F_K^{(n)}(U)-F_W^{(n)}(U)\|_{L^1}
    \le\eta_n h
       \left(2R+\int_{-R}^R 2R\frac{|\chi|}{4}\,ds\right)
    =\eta_n(2R+|\chi|R^2)h.
\end{align*}
Combining the estimates concludes the proof.
\end{proof}

%%%%%%%%%%%%%%%%%%%%%%%%%%%%%%%%%%%%%%%%%%%%%%%%%%%%%%%%%%
\subsection{Pointwise discretization}
\label{app:R4-pointwise}

While the cell-averaged approximation $W^{(N)}=(\Pi_N\otimes\Pi_N)W$ accommodates arbitrary $W\in L^p(I^2)$ and depends only on its a.e.\ equivalence class, applications often evaluate $W$ directly at prescribed nodes. The following result establishes that this pointwise discretization yields the same deterministic limit, provided $W$ admits a uniformly continuous representative and the corresponding solutions satisfy the bounds required in Theorem~\ref{thm:det-Lp-conv}. We use this representative in all pointwise evaluations below.

Fix gridpoints $x_k^{(N)}:=(k-\tfrac12)/N$ and define
\begin{align*}
W^{\mathrm{pw},(N)}(x,y):=W\big(x_k^{(N)},x_\ell^{(N)}\big),
\quad (x,y)\in I_k\times I_\ell.
\end{align*}

\begin{lemma}[Pointwise step-kernel convergence]
\label{lem:pw-Lp}
Assume that $W$ admits a uniformly continuous representative on $I^2$.
Then
\begin{align*}
\|W^{\mathrm{pw},(N)}-W\|_{L^\infty(I^2)}\to0
\quad\text{as }N\to\infty.
\end{align*}
Consequently, convergence also holds in $L^p(I^2)$ for every $1\le p<\infty$.
\end{lemma}

\begin{proof}
Let $\omega$ be a modulus of continuity for $W$. For $(x,y)\in I_k\times I_\ell$,
\begin{align*}
|W^{\mathrm{pw},(N)}(x,y)-W(x,y)|
\le\omega(\sqrt2/N).
\end{align*}
Taking the essential supremum and letting $N\to\infty$ proves the claim.
\end{proof}

\begin{corollary}[Deterministic limit for pointwise discretization]
\label{cor:pw-det-limit}
Let $W$ satisfy the hypotheses of Lemma~\ref{lem:pw-Lp}. Consider the discrete system with weights $\mathsf w^{\mathrm{pw},(N)}_{k\ell}:=W\big(x_k^{(N)},x_\ell^{(N)}\big)$ and cell-average initial data. Retain the hypotheses of Theorem~\ref{thm:det-Lp-conv} on the initial data and corresponding solutions, including its uniform $L^\infty$ bound. Then the estimate and convergence conclusion of that theorem hold with $W^{(N)}$ replaced by $W^{\mathrm{pw},(N)}$.
\end{corollary}

\begin{proof}
Uniform continuity on $I^2$ implies that $W$ is bounded, so
$M_x(|W^{\mathrm{pw},(N)}|),M_y(|W^{\mathrm{pw},(N)}|)\le\|W\|_{L^\infty(I^2)}$ uniformly in $N$. Lemma~\ref{lem:pw-Lp} gives convergence of the kernels in $L^p(I^2)$. The stability--consistency argument of Theorem~\ref{thm:det-Lp-conv} therefore applies to the corresponding step reconstructions, as in Remark~\ref{rmk:general-KN}.
\end{proof}

%%%%%%%%%%%%%%%%%%%%%%%%%%%%%%%%%%%%%%%%%%%%%%%%%%%%%%%%%%
\subsection{Cut-norm concentration for sampled graphs}
\label{app:R5-cutconc}

We use the random-graph construction of Section~\ref{ssec:random-limit}. The following estimate compares the empirical step kernel with its mean in cut norm. Together with cell-average convergence, it gives convergence to $W$ under the fixed labeling and connects the sampled graphs with the cut-distance formulation.

\begin{lemma}[Cut-norm concentration for sampled step graphs]
\label{lem:cut-conc-step}
Under the random-graph set-up of Section~\ref{ssec:random-limit}, for every $\delta\in(0,1)$, with probability at least $1-\delta$,
\begin{align*}
\|W_\xi^{(N)}-W^{(N)}\|_\square
\le \frac1N+\frac2N\sqrt{N\log4+\log(2/\delta)}.
\end{align*}
\end{lemma}

\begin{proof}
Set $H^{(N)}:=W_\xi^{(N)}-W^{(N)}$. Since this is a step kernel, its cut norm is attained by sets that are unions of partition cells. For any fixed pair of such sets, the diagonal contribution to the cut integral is at most $1/N$ in absolute value, and the off-diagonal contribution has the form
\begin{align*}
\frac1{N^2}\sum_{1\le k<\ell\le N}
   a_{k\ell}\big(\xi_{k\ell}^{(N)}-\mathsf p_{k\ell}^{(N)}\big),
\qquad a_{k\ell}\in\{0,1,2\}.
\end{align*}
The variables in this sum are independent and centered. Hoeffding's inequality and a union bound over at most $4^N$ pairs of cell unions give, for $\tau>0$,
\begin{align*}
\pr\big(\|H^{(N)}\|_\square>1/N+\tau\big)
\le 2\cdot4^N\exp(-N^2\tau^2/4).
\end{align*}
Taking $\tau=2N^{-1}\sqrt{N\log4+\log(2/\delta)}$ proves the estimate.
\end{proof}

\begin{corollary}[Random cut-distance graphon limit]
\label{cor:rand-cut}
Under the random-graph set-up of Section~\ref{ssec:random-limit},
\begin{align*}
\|W_\xi^{(N)}-W\|_\square\stackrel{\pr}{\to}0,
\qquad
\delta_\square(W_\xi^{(N)},W)\stackrel{\pr}{\to}0.
\end{align*}
For every fixed $T>0$ and $1\le p<\infty$, the sampled-graph solution with cell-average initial data satisfies
\begin{align*}
\sup_{t\in[0,T]}
\|\Pi_NU^{\xi,(N)}(t)-U(t)\|_{L^p(I)^2}
\stackrel{\pr}{\to}0.
\end{align*}
\end{corollary}

\begin{proof}
By Lemmas~\ref{lem:cut-conc-step} and~\ref{lem:PiN-Lp},
\begin{align*}
\|W_\xi^{(N)}-W\|_\square
\le \|W_\xi^{(N)}-W^{(N)}\|_\square
   +\|W^{(N)}-W\|_{L^1(I^2)}
\stackrel{\pr}{\to}0.
\end{align*}
The cut-distance convergence follows since
$\delta_\square(W_\xi^{(N)},W)\le\|W_\xi^{(N)}-W\|_\square$.

Both kernels have row and column integrals bounded by $1$, and the solutions have deterministic uniform $L^\infty$ bounds on $[0,T]$. The stability estimate in the proof of Theorem~\ref{thm:det-cut}, applied under the identity relabeling, gives
\begin{align*}
\sup_{t\in[0,T]}\|\Pi_NU^{\xi,(N)}(t)-U(t)\|_{L^1(I)^2}
\le C_T\big(
   \|\Pi_NU_{\mathrm{in}}-U_{\mathrm{in}}\|_{L^1(I)^2}
   +\|W_\xi^{(N)}-W\|_\square
\big).
\end{align*}
The right-hand side converges to zero in probability. Interpolation with the uniform bounds gives convergence in $C([0,T];L^p(I)^2)$ for $1<p<\infty$.
\end{proof}

%%%%%%%%%%%%%%%%%%%%%%%%%%%%%%%%%%%%%%%%%%%%%%%%%%%%%%%
\section{Symmetry Invariance of the Graphon Model}\label{sec:symmetry-invariant}

In this section, we prove the invariance of symmetry for the nonlocal graphon PDE model \eqref{eq:our-graphon-model}. The result is stated in the following theorem.

\begin{theorem}[Symmetry invariance]
    Assume the hypotheses of Proposition~\ref{prop:graphon_bounds}~\textup{(ii)}. If the graphon kernel satisfies $W(x,y)=W(1-x,1-y)$ for all $x,y\in [0,1]$, and the initial conditions satisfy 
    \begin{equation}\label{eq:initial-symmetry}
        n(0, x)=n(0,1-x)\geq 0,\quad c(0,x)=c(0,1-x)\geq 0,\text{ for a.e. }x\in[0,1],
    \end{equation}
    then $n(t,x)=n(t,1-x)$, $c(t,x)=c(t,1-x)$ for all $t>0$ and $x\in[0,1]$.
\end{theorem}
\begin{proof}
The map $\varphi(x)=1-x$ is a measure-preserving bijection, with $W^\varphi=W$ and $U_{\mathrm{in}}^\varphi=U_{\mathrm{in}}$. By Lemma~\ref{lem:equivariance}, $U^\varphi$ solves the same initial-value problem as $U$. Uniqueness in Proposition~\ref{prop:graphon_bounds} gives $U^\varphi=U$.
\end{proof}

%%%%%%%%%%%%%%%%%%%%%%%%%%%%%%%%%%%%%%%%%%%%%%%%%%%%%%%
\section{Derivation of Equations for Weakly Nonlinear Stability Analysis}\label{sec:derivation_stability_analysis}

In this section, we use our perturbation expansion of the profiles $n(t,x)$ and $c(t,x)$ for population and chemoattractant near the spatially uniform state to derive a hierarchy of graphon models in which the nonlinear terms at a given order are expressed in terms of solutions that have been obtained in previous orders. 

By introducing a small bookkeeping parameter $\varepsilon>0$, considering chemoattractant sensitivity $\chi=\chi^* +\varepsilon^2\tilde{\chi}$ close to the stability threshold, and defining a slow timescale $\tau=\varepsilon^2 t$, we can look to perform a perturbation expansion of our graphon Keller--Segel in terms of the following asymptotic series
\begin{equation}\label{eq:linear_perturbation_3}
    \begin{aligned}
        n(\tau,x)&=\bar{n} + \varepsilon n_1(\tau,x) + \varepsilon^2 n_2(\tau,x) + \varepsilon^3 n_3(\tau,x) + \cdots,\\
        c(\tau,x)&=\bar{c} + \varepsilon c_1(\tau,x) + \varepsilon^2 c_2(\tau,x) + \varepsilon^3 c_3(\tau,x) + \cdots.
    \end{aligned}
\end{equation}

To perform the expansion of the differential equation, we use the following lemma to expand the reaction term $f(n)$.

\begin{lemma}\label{lemma:analytic-fn-expansion}
    Let $F$ be $C^4$ in a neighborhood of $x_0$, and let $x(\varepsilon)=x_0+\varepsilon x_1+\varepsilon^2 x_2 + \varepsilon^3 x_3$, Then as $\varepsilon\to 0$,
    \begin{equation}\label{eq:analytic-fn-expansion}
        \begin{aligned}
            F\big(x(\varepsilon)\big)=&\,\,F(x_0)+\varepsilon F'(x_0)x_1+\varepsilon^2\left[F'(x_0)x_2+\frac{1}{2}F''(x_0)x_1^2\right]\\
            &\,\,+\varepsilon^3\left[F'(x_0)x_3+F''(x_0)x_1x_2+\frac{1}{6}F'''(x_0)x_1^3\right]+O(\varepsilon^4).
        \end{aligned}
    \end{equation}
\end{lemma}

As part of our expansion, we will often have to deal with the sums and differences of quantities measured at two locations $x$ and $y$ on the graphon that have connectivity $W(x-y)$ under the graphon adjacency function. To describe these sums and differences, we introduce the following helpful notation
\begin{equation}
    \sigma_x^yh:=h(x)+h(y),\quad \delta_x^yh:=h(y)-h(x).
\end{equation}
We then apply the expansions provided in Equation \eqref{eq:linear_perturbation_3} and our assumed forms of $\chi$ and $\tau$ to  expand our model \eqref{eq:our-graphon-model} with respect to $\varepsilon$. Before dealing with the entire equation, we first expand the following term that describes the effects of biased movement between locations $x$ and $y$
\begin{equation*}
    \begin{aligned}
        \mathcal{F}[n, c](\tau, x, y)=&\,\,\frac{n(\tau,y)e^{\chi c(\tau,x)}-n(\tau,x)e^{\chi c(\tau,y)}}{e^{\chi c(\tau,x)}+e^{\chi c(\tau,y)}}=\frac{1}{1+e^{\chi \delta_x^y c}}n(\tau, y)-\frac{1}{1+e^{-\chi\delta_x^yc}}n(\tau, x)\\
        =&\,\,\varepsilon\left[\frac{1}{2}\delta_x^y n_1-\frac{1}{2}\chi^* \bar{n}\delta_x^y c_1\right]\\
        &\,\,+\varepsilon^2\left[\frac{1}{2}\delta_x^yn_2-\frac{1}{2}\chi^* \bar{n}\delta_x^y c_2-\frac{1}{4}\chi^* \sigma_x^y n_1\delta_x^y c_1\right]\\
        &\,\,+\varepsilon^3\bigg[\frac{1}{2}\delta_x^yn_3-\frac{1}{2}\chi^* \bar{n}\delta_x^y c_3-\frac{1}{4}\chi^* \sigma_x^y n_2\delta_x^y c_1\\
        &\,\,\qquad\,\,-\frac{1}{4}\chi^*\sigma_x^y n_1\delta_x^y c_2+\frac{1}{24}(\chi^*)^3\bar{n}(\delta_x^y c_1)^3-\frac{1}{2}\tilde{\chi}\bar{n}\delta_x^yc_1\bigg]+O(\varepsilon^4).
    \end{aligned}
\end{equation*}
Finally, we use the following notation to represent the graphon Laplacian operator in the following form
\begin{equation*}
    \mathcal{L}_WF:=\int_{0}^{1}W(y-x)\delta_x^y F\,\mathrm{d}y=\int_{0}^{1}W(y-x)\big[F(y)-F(x)\big]\,\mathrm{d}y,
\end{equation*}
which will allow us to describe both diffusive and directed motion in the linear portion of each equation in our perturbation hierarchy. 

After using our perturbation expansion and applying Lemma~\ref{lemma:analytic-fn-expansion} to $f(n)$, we can describe the equations at each order in terms of the quantities defined above. In particular, we see that all terms at order $\mathcal{O}(1)$ recover the fact that $(n(x),c(x)) = \left( \bar{n}, \frac{\alpha \bar{n}}{\beta} \right)$ is a spatially uniform equilibrium for our graphon system. At order $\mathcal{O}(\varepsilon)$, we see that our graphon equation is the linearized system we also obtained through linear stability analysis in Section \ref{sec:linear-stability-graphon}:
\begin{equation}\label{eq:first_order}
    \begin{aligned}
        0&=\frac{\eta_n}{2}\mathcal{L}_W n_1-\frac{\eta_n}{2} \chi^* \bar{n}\mathcal{L}_W c_1 + f'(\bar{n})n_1,\\
        0&=\frac{\eta_c}{2}\mathcal{L}_W c_1+\alpha n_1-\beta c_1.
    \end{aligned}
\end{equation}
Next, at order $\mathcal{O}(\varepsilon^2)$, we obtain the following graphon system featuring linear terms involving $n_2$ and $c_2$ and nonlinear forcing terms depending only on $n_1$ and $c_1$:
\begin{equation}\label{eq:second_order}
    \begin{aligned}
        0=&\,\,\frac{\eta_n}{2}\mathcal{L}_W n_2-\frac{\eta_n}{2}\chi^* \bar{n}\mathcal{L}_W c_2 + f'(\bar{n}) n_2\\
        &\,\,-\frac{\eta_n}{4}\chi^*\int_{0}^{1}W(y-x)\sigma_x^yn_1\delta_x^y c_1\,\mathrm{d}y+\frac{1}{2}f''(\bar{n})n_1^2,\\
        0=&\,\,\frac{\eta_c}{2}\mathcal{L}_W c_2+\alpha n_2-\beta c_2.
    \end{aligned}
\end{equation}
Finally, at order $\mathcal{O}(\varepsilon^3)$, we obtain the following graphon system featuring linear terms in $n_3$ and $c_3$ and nonlinear forcing functions depending on the solutions to problems at the two previous orders:
\begin{equation}\label{eq:third_order}
    \begin{aligned}
        \frac{\partial n_1}{\partial \tau}
        ={}&\frac{\eta_n}{2}\mathcal{L}_W n_3
        -\frac{\eta_n}{2}\chi^* \bar{n}\mathcal{L}_W c_3
        +f'(\bar{n}) n_3\\
        &-\frac{\eta_n}{4}\chi^*
        \int_{0}^{1}W(y-x)\sigma_x^yn_2\,\delta_x^yc_1\,\mathrm{d}y
        -\frac{\eta_n}{4}\chi^*
        \int_{0}^{1}W(y-x)\sigma_x^yn_1\,\delta_x^yc_2\,\mathrm{d}y\\
        &+\frac{\eta_n}{24}(\chi^*)^3\bar{n}
        \int_{0}^{1}W(y-x)\left(\delta_x^yc_1\right)^3\,\mathrm{d}y
        -\frac{\eta_n}{2}\tilde{\chi}\bar{n}\mathcal{L}_W c_1\\
        &+f''(\bar{n})n_1n_2
        +\frac{1}{6}f'''(\bar{n})n_1^3,\\
        \frac{\partial c_1}{\partial \tau}
        ={}&\frac{\eta_c}{2}\mathcal{L}_W c_3+\alpha n_3-\beta c_3.
    \end{aligned}
\end{equation}

\end{document}